\documentclass[a4paper,11pt]{amsart}

\usepackage[colorlinks=true, linkcolor=red, linktoc=page]{hyperref}
\usepackage{amssymb}
\usepackage{enumitem}
\usepackage{url}
\usepackage{xcolor}
\usepackage{graphicx}
\usepackage{geometry}
\usepackage{eufrak}
\usepackage{subcaption}
\usepackage[capitalise]{cleveref}

\let\para\S
\crefformat{section}{#2\para#1#3}
\crefformat{subsection}{#2\para#1#3}
\crefformat{subsubsection}{#2\para#1#3}

\usepackage[T1]{fontenc}
\usepackage[english]{babel}
\usepackage{times}

\usepackage{dsfont}
\usepackage[mathscr]{euscript}
\usepackage{stmaryrd}
 \usepackage{upgreek}

\makeatletter

\def\section{\@startsection{section}{1}%
\z@{.7\linespacing\@plus\linespacing}{.5\linespacing}%
{\normalfont\bfseries\centering}}

\def\@settitle{\begin{center}%
  \baselineskip14\p@\relax
    \bfseries
    \LARGE\@title
  \end{center}%
}

\def\@setauthors{%
  \begingroup
  \trivlist
  \centering\footnotesize \@topsep30\p@\relax
  \advance\@topsep by -\baselineskip
  \item\relax
  \andify\authors
  \def\\{\protect\linebreak}%
 {\Large\authors}%
  \endtrivlist
  \endgroup
}

\def\maketitle{\par
  \@topnum\z@ 
  \@setcopyright
  \thispagestyle{firstpage}
  \ifx\@empty\shortauthors \let\shortauthors\shorttitle
  \else \andify\shortauthors
  \fi
  \@maketitle@hook
  \begingroup
  \@maketitle
  \toks@\@xp{\shortauthors}\@temptokena\@xp{\shorttitle}%
  \toks4{\def\\{ \ignorespaces}}
  \edef\@tempa{%
    \@nx\markboth{\the\toks4
      \@nx{\the\toks@}}{\the\@temptokena}}%
  \@tempa
  \endgroup
  \c@footnote\z@
  \def\do##1{\let##1\relax}%
  \do\maketitle \do\@maketitle \do\title \do\@xtitle \do\@title
  \do\author \do\@xauthor \do\address \do\@xaddress
  \do\email \do\@xemail \do\curraddr \do\@xcurraddr
  \do\commby \do\@commby
  \do\dedicatory \do\@dedicatory \do\thanks \do\thankses
  \do\keywords \do\@keywords \do\subjclass \do\@subjclass
}
\makeatother

\newtheorem{defi}{Definition}[section]
\newtheorem{lem}[defi]{Lemma}

\newtheorem{prop}[defi]{Proposition}
\newtheorem{theo}[defi]{Theorem}
\newtheorem{cor}[defi]{Corollary}
\newtheorem{rk}[defi]{Remark}

\newcommand{\vta}{\vartheta}
\newcommand{\cal}{\mathcal}
\newcommand{\ES}{\mathbb{E}}
\newcommand{\ER}{\mathbb{R}}
\newcommand{\PE}{\mathbb{P}}
\newcommand{\un}{\underline}
\newcommand{\Exp}{\mathbb{E}}
\newcommand{\Var}{\mathrm{Var}}
\renewcommand{\Pr}{\mathbb{P}}
\newcommand{\R}{\mathbb{R}}
\newcommand{\N}{\mathbb{N}}

\newcommand{\dd}{\mathrm{d}}

\newcommand{\ee}{\mathrm{e}}
\newcommand{\ind}[1]{\mathds{1}_{\{#1\}}}
\newcommand{\iind}[1]{\mathds{1}_{#1}}

\DeclareMathOperator*{\sgn}{sgn}
\newcommand{\dto}{\downarrow}

\renewcommand{\tilde}{\widetilde}
\renewcommand{\bar}{\overline}

\newcommand{\pn}{\mathfrak{p}}
\renewcommand{\emptyset}{\varnothing}
\newcommand{\tauel}{\mathfrak{s}}
\newcommand{\mubar}{{\mu}}

\newcommand{\bfeta}{{\boldsymbol{\eta}}}
\newcommand{\bfnu}{\boldsymbol{\nu}}

\newcommand{\dten}{t}

\title{Asymptotically unbiased approximation of the QSD of diffusion processes with a decreasing time step Euler scheme}
\thanks{This work benefitted from the support of the project ANR QuAMProcs (ANR-19-CE40-0010) from the French National Research Agency and of the Centre Henri Lebesgue ANR-11-LABX-0020-01.}

\author{Fabien Panloup}
\address{{\bf Fabien Panloup}\newline
{\rm \indent {Univ Angers, CNRS, LAREMA, SFR MATHSTIC, F-49000 Angers, France}}}
\email{\href{mailto: fabien.panloup@univ-angers.fr}{fabien.panloup@univ-angers.fr}}

\author{Julien Reygner}
\address{{\bf Julien Reygner}\newline
{\rm \indent CERMICS, ENPC, Institut Polytechnique de Paris, Marne-la-Vallée, France}}
\email{\href{mailto:julien.reygner@enpc.fr}{julien.reygner@enpc.fr}}

\begin{document}

\begin{abstract} 
  We build and study a recursive algorithm based on the occupation measure of an Euler scheme with decreasing step for the numerical approximation of the quasistationary distribution (QSD) of an elliptic diffusion in a bounded domain. We prove the almost sure convergence of the procedure for a family of redistributions and show that we can also recover the approximation of the rate of survival and the convergence in distribution of the algorithm. This last point follows from some new bounds on the weak error related to diffusion dynamics with renewal.
\end{abstract}

\maketitle

\section{Introduction and statement of the results}
\subsection{Motivation}\label{ss:motivation}

This article is dedicated to the construction and study of a numerical scheme aimed at approximating the quasistationary distribution, in some bounded domain $D \subset \R^d$, $d \geq 1$, of the solution $(Y_t)_{t \geq 0}$ to the stochastic differential equation
\begin{equation}\label{eq:SDE}
  \dd Y_t = b(Y_t)\dd t + \sigma(Y_t)\dd B_t,
\end{equation}
where $b : \R^d \to \R^d$, $\sigma:\R^d\to \mathbb{M}_{d,d}$, and $(B_t)_{t \geq 0}$ is a $d$-dimensional Brownian motion. {Here and throughout the article, we denote by $\mathbb{M}_{d,d}$ the set of $d \times d$ matrices.}

\medskip
We shall work under the following assumptions:
\begin{enumerate}[label=$\mathbf{(H_{\arabic*})}$,ref=$\mathbf{(H_{\arabic*})}$]
  \item\label{cond:D} $D$ is a non-empty bounded and connected open set of $\ER^d$, whose boundary is ${\cal C}^2$.
  \item\label{cond:coeffs} \begin{enumerate}[label=$(\mathbf{\alph*})$,ref=$(\mathbf{H_2.\alph*})$]
    \item\label{it:coeffs-1} $b$ and $\sigma$ are bounded measurable on $\R^d$ and  $\sigma\sigma^\top$ is uniformly elliptic in $\R^d$.
    \item\label{it:coeffs-2} $b$ and $\sigma$ are Lipschitz continuous on $\R^d$.
  \end{enumerate}
\end{enumerate} 
Under these assumptions, it is known~\cite[Section~5.3]{ChaCouVil18} that if {one defines}
\begin{equation*}
  \tau_D = \inf\{t \geq 0: Y_t \not\in D\},
\end{equation*}
then there is a unique probability measure $\mu^\star$ on $D$ such that, for any $t \geq 0$,
\begin{equation*}
  \Pr_{\mu^\star}(Y_t \in \cdot | \tau_D > t) = \mu^\star(\cdot).
\end{equation*}
This measure is called the \emph{quasistationary distribution} (QSD) of $(Y_t)_{t \geq 0}$ in $D$. From a spectral point of view, it is the left eigenvector associated with the principal Dirichlet eigenvalue $\lambda^\star>0$ on $D$ of the infinitesimal generator $L$ of~\eqref{eq:SDE}. The eigenvalue $\lambda^\star$ is usually called \emph{rate of survival} in the literature. QSDs arise in many fields of applied probability, for instance in population dynamics~\cite{MV12}, molecular dynamics~\cite{DiGLelLePNec16}, or Monte Carlo methods~\cite{PolFeaJohRob20}. We refer to the monograph~\cite{CMM13} for an overall mathematical introduction.

\medskip
The numerical approximation of QSDs is a nontrivial task. There exist two main classes of methods for this purpose: particle systems, such as the Fleming--Viot process~\cite{BHM00,GriKan04,Vil14}, and interaction with the occupation measure~\cite{AFP,BenCloPan18,WRS20, BenChaVil22}. In the latter method, one constructs a random process $(X_t)_{t \geq 0}$ in $D$ according to the following rules:
\begin{itemize}
  \item in $D$, $X_t$ follows the stochastic differential equation~\eqref{eq:SDE},
  \item when $X_t$ reaches $\partial D$, it is killed and restarted from a point randomly chosen in $D$ according to the occupation measure ${t}^{-1}\int_0^t \delta_{X_s}\dd s$.
\end{itemize}
Under assumptions which are implied by~\ref{cond:D} and~\ref{cond:coeffs}, Bena\"im, Champagnat and Villemonais~\cite[Theorem~2.1]{BenChaVil22} proved that for any bounded and measurable function $f : D \to \R$,
\begin{equation}\label{eq:BCV}
  \lim_{t \to +\infty} \frac{1}{t}\int_0^t f(X_s) \dd s = \mu^\star(f), \qquad \text{almost surely.}
\end{equation}

\medskip
In practice, in order to simulate trajectories of $X_t$, the stochastic differential equation~\eqref{eq:SDE} needs to be discretized. Thus, the purpose of the present article is to provide a similar statement to~\eqref{eq:BCV} for an Euler scheme associated with~\eqref{eq:SDE}. If this equation is discretized with a constant step size $\gamma$, then it yields a homogeneous Markov chain in $\R^d$, and the discrete-time adaptation of the algorithm described above converges to the QSD $\mu^{\star,\gamma}$ of this chain in $D$. This QSD depends on $\gamma$, and it is known to converge to $\mu^\star$ when $\gamma \to 0$~\cite[Theorem~3.9]{BenCloPan18}, although to the best of our knowledge, no quantitative error estimate is available.

In order to remove the bias introduced by the approximation of $\mu^\star$ by $\mu^{\star,\gamma}$, in this article we follow the idea, initially introduced by Lamberton and Pag\`es~\cite{LP1} for stationary distributions of diffusions, to use an Euler scheme with \emph{decreasing} step size $\gamma_n$. Our main result shows that the algorithm above then directly converges to the QSD $\mu^\star$, without any discretization bias, as is the case for stationary distributions.

\subsection{The numerical scheme and a simplified statement}

Our scheme is based on a sequence of positive time steps $\boldsymbol{\gamma} = (\gamma_n)_{n \geq 1}$ which satisfies the following assumptions, in which $\Gamma_n := \sum_{k=0}^{n-1} \gamma_{k+1}$:
\begin{enumerate}[label=$\mathbf{(H_3)}$,ref=$\mathbf{(H_3)}$]
  \item\label{cond:steps}\begin{enumerate}[label=$(\mathbf{\alph*})$,ref=$(\mathbf{H_3.\alph*})$]
    \item\label{it:steps-1} $\lim_{n \to +\infty} \gamma_n = 0$, $\lim_{n \to +\infty} \Gamma_n = +\infty$.
    \item\label{it:steps-2} There exists $p>1$ such that $\sum_{n=1}^{+\infty} \gamma_n^p < +\infty$.
     \item\label{it:steps-3} $\sup_{n \geq 1}\gamma_n/\gamma_{n+1}<+\infty$.
  \end{enumerate}
\end{enumerate}
These three conditions are for instance satisfied if $\gamma_n= Cn^{-\rho}$ with $\rho\in(0,1]$. Whereas~\ref{it:steps-1} is standard in stochastic algorithms and used throughout the paper, the precise roles of~\ref{it:steps-2} and~\ref{it:steps-3} are respectively discussed in~Remark~\ref{rem:pas} and~Remark~\ref{rk:H3c}.

\medskip
With this step sequence at hand, on a probability space $(\Omega,\mathcal{F},\Pr)$ equipped with a $d$-dimensional Brownian motion $(B_t)_{t \geq 0}$, we define the cadlag process $(\bar{X}_t)_{t \geq 0}$ as follows. We set $\bar{X}_0 = x_0 \in D$, and for any $n \geq 0$, for any $t \in [\Gamma_n, \Gamma_{n+1})$,
\begin{equation*}
  \bar{X}_t = \bar{X}_{\Gamma_n} + b\left(\bar{X}_{\Gamma_n}\right)(t-\Gamma_n) + \sigma\left(\bar{X}_{\Gamma_n}\right)(B_t-B_{\Gamma_n}).
\end{equation*}
In particular, given $\bar{X}_{\Gamma_n}$, the numerical simulation of
\begin{equation*}
  \bar{X}_{\Gamma_{n+1}^-} = \bar{X}_{\Gamma_n} + b\left(\bar{X}_{\Gamma_n}\right)(\Gamma_{n+1}-\Gamma_n) + \sigma\left(\bar{X}_{\Gamma_n}\right)(B_{\Gamma_{n+1}}-B_{\Gamma_n})
\end{equation*}
only requires to sample the Brownian increment $B_{\Gamma_{n+1}}-B_{\Gamma_n}$. Defining the Bernoulli variable 
\begin{equation}\label{eq:theta}
  \theta_{n+1} := \ind{\bar{X}_{\Gamma_{n+1}^-} \not\in D},
\end{equation}
we next let
\begin{equation*}
  \bar{X}_{\Gamma_{n+1}} := \theta_{n+1} U_{n+1} + (1-\theta_{n+1}) \bar{X}_{\Gamma_{n+1}^-},
\end{equation*}
where, conditionally on $(\bar{X}_{\Gamma_0}, \ldots, \bar{X}_{\Gamma_n})$, $U_{n+1}$ is drawn according to some random measure $\pn_{n+1}$ on $D$, which we call \emph{redistribution measure} and is measurable with respect to the $\sigma$-algebra generated by $(\bar{X}_{\Gamma_0}, \ldots, \bar{X}_{\Gamma_n})$. {Here, and throughout the article, the space $\mathcal{M}_1(D)$ of Borel probability measures on $D$ is endowed with the topology of weak convergence and the associated Borel $\sigma$-algebra.}

\medskip
A natural choice for the redistribution measure $\pn_n$ is the occupation measure of the scheme
\begin{equation*}
  \mubar_n = \frac{1}{\Gamma_n}\sum_{k=0}^{n-1} \gamma_{k+1} \delta_{\bar{X}_{\Gamma_k}},
\end{equation*}
which mimics the `true' occupation measure $\Gamma_n^{-1}\int_0^{\Gamma_n} \delta_{X_s} \dd s$ associated with the continuous time process $(X_t)_{t \geq 0}$ defined above. For this choice, our first main result reads as follows. 

\begin{theo}[Convergence to the QSD for $\pn_n=\mubar_n$]\label{theo:main-mun}
  Under Assumptions~\ref{cond:D}, \ref{cond:coeffs} and~\ref{cond:steps}, and if $\pn_n=\mubar_n$, then for any measurable and bounded function $f : D \to \R^d$, 
  \begin{equation*}
    \lim_{n \to +\infty} \mubar_n(f) = \mu^\star(f), \qquad \text{almost surely.}
  \end{equation*}
\end{theo}

\subsection{More general redistribution measures} As a discrete measure, $\mu_n$ can be easily simulated. However, its sampling requires to keep all the memory of the path of the Euler scheme $(\bar{X}_{\Gamma_0}, \ldots, \bar{X}_{\Gamma_n})$ and its cost of computation therefore strongly depends on $n$. This suggests to wonder about some strategies to manage these potential numerical drawbacks and to consider the possibility of replacing the true occupation measure of the Euler scheme $\mu_n$ by an approximation $\pn_n$ (see Subsection~\ref{subsec:practical} for a detailed discussion). Below, we thus propose to generalize Theorem~\ref{theo:main-mun} to  general redistribution measures $\pn_n$ under the (natural) following assumption:

\begin{enumerate}[label=$\mathbf{(H_4)}$,ref=$\mathbf{(H_4)}$]
  \item\label{cond:pn} for any Lipschitz continuous function $f:D \to \ER$, 
  \begin{equation*}
    \lim_{n \to +\infty} \mubar_n(f)-\pn_n(f) = 0, \qquad \text{almost surely}.
  \end{equation*}
\end{enumerate}

Our main result shows the convergence to the QSD of the occupation measure of $(\bar{X}_{\Gamma_n})_{n \geq 0}$, and almost as a by-product (see Subsection~\ref{ss:pf-main} for details), also provides a way to approximate the related \emph{rate of survival} $\lambda^{\star}$ by counting the number of jumps of the underlying continuous-time process $\bar{X}_t$ between $0$ and $\Gamma_n$.

\begin{theo}[Convergence to the QSD]\label{theo:main}
  Under Assumptions~\ref{cond:D}, \ref{cond:coeffs}, \ref{cond:steps} and~\ref{cond:pn}, for any measurable and bounded function $f : D \to \R^d$, 
  \begin{equation*}
    \lim_{n \to +\infty} \mubar_n(f) = \mu^\star(f), \qquad \text{almost surely.}
  \end{equation*}
  Besides,
  \begin{equation*}
    \lim_{n \to +\infty} \frac{1}{\Gamma_n}\sum_{k=1}^n \theta_k = \lambda^\star, \qquad \text{almost surely.}
  \end{equation*}
\end{theo}

In the setting of~Theorem~\ref{theo:main}, Assumption~\ref{cond:pn} combined with the tightness of $(\mubar_n)_{n \geq 1}$, which will be proved in~Proposition~\ref{prop:tightness}, implies that we also have $\pn_n(f) \to \mu^\star(f)$ almost surely for any continuous and bounded function $f : D \to \R$. Therefore, one deduces the following corollary, which will be used in the proof of~Theorem~\ref{theo2} below, from a standard separability argument.

\begin{cor}[Almost sure weak convergence of $(\pn_n)_{n \geq 1}$]\label{cor:pnconverges}
  In the setting of~Theorem~\ref{theo:main}, $\pn_n$ converges weakly to $\mu^\star$, almost surely.
\end{cor}

The proof of Theorem~\ref{theo:main} follows the strategy of \cite{BenChaVil22} and is classically divided into two parts respectively devoted to the almost sure tightness of the sequence $(\mu_n)_{n \geq 1}$ and to the identification of the limit. Below, we give some comments on the related proofs.

\medskip
$\rhd$ \textit{Tightness:} In this part, the main difficulty is to control the time spent by the sequence $(\bar{X}_{\Gamma_n})_{n \geq 0}$ near the boundary of $D$ and indirectly the number of redistributions induced by exits of $D$. Actually, $D$ being a bounded subset of $\ER^d$, proving the tightness reduces to showing that the dynamics does not concentrate any mass close to the boundary of $D$. Roughly, our proof consists of coupling the (time-changed) distance of $(\bar{X}_{\Gamma_n})_{n \geq 0}$ to the boundary with an appropriately reflected process, which may be seen as a discretization of a one-dimensional reflected Brownian motion with drift, and for which it is possible to better control the time spent near the boundary. The construction of this coupling and the study of the dynamics of this reflected-type process are probably among the main challenges overcome in this paper. 
It is worth noting that in this section, we do not require  $b$ and $\sigma$ to be Lipschitz and only use Assumption~\ref{it:coeffs-1} (see Proposition~\ref{prop:tightness} for details). As well, we do not use \ref{cond:pn} in this part of the proof, which means that our tightness result is (surprisingly) available for schemes constructed with \emph{any} adapted redistribution sequence $(\pn_n)_{n \geq 1}$ on $D$.

\medskip
$\rhd$ \textit{Identification of the limit:} As shown in \cite{BenChaVil22} (see also \cite{BenCloPan18} for a discrete-time counterpart), $\mu^\star$ can be viewed as an attractive point of an ODE $\dot{\nu}_t=F(\nu_t)$ on ${\cal M}_1(D)$ (see Proposition~\ref{prop:ODE} for details). With this property, the strategy is to show that the dynamics of $(\mu_k)_{k\ge n}$ asymptotically fits the one of the ODE, \emph{i.e.} that an appropriate continuous-time extension of $(\mu_k)_{k\ge n}$ is an \emph{asymptotic pseudo-trajectory} of this ODE. Note that as in \cite{BenChaVil22}, we choose in fact to mainly study the empirical measure restricted to the renewal times of $(\bar{X}_{\Gamma_k})_{k\ge1}$ (denoted by $(\vta_\ell)_{\ell\ge1}$), which allows to look at an asymptotically homogeneous and ergodic sequence (see Section~\ref{sec:prooftheomain} for details).  In view of our discretized dynamics, the main difficulty of this part of the proof is to show some uniform convergence properties of the weak error related to the dynamics of the diffusion killed at $\tau_D=\inf\{t\ge0, Y_t\in D^c\}$ (see Proposition~\ref{prop:weakA}). The extension to measurable functions is also a challenge: with an adaptation of  \cite{LemMen10}, we obtain some bounds on the density of an Euler scheme with a non-constant step sequence satisfying~\ref{it:steps-3}  (see Lemma~\ref{lem:gauss-density}).

\begin{rk}\label{withoutH3c}
  We shall see in Remark~\ref{rk:H3c} that without Assumption~\ref{it:steps-3}, the conclusions of Theorem~\ref{theo:main-mun} and Theorem~\ref{theo:main} hold true for continuous and bounded functions $f : D \to \R$. As a consequence, Corollary~\ref{cor:pnconverges} holds without this assumption. In fact, Assumption~\ref{it:steps-3} is only employed to extend the convergence to measurable and bounded functions, through a regularization argument which requires to get some uniform bounds on the density of Euler schemes with nonconstant step sizes (see Proposition~\ref{prop:weakA} and Lemma~\ref{lem:gauss-density}).
\end{rk}

{\begin{rk}
  Assumption~\ref{cond:coeffs} is stated for coefficients $b$ and $\sigma$ defined on the whole space $\R^d$. However, the construction of our scheme only depends on these coefficients through their restriction to the set $D$. We chose this formulation because in the proofs of Sections~\ref{sec:prooftheomain} and~\ref{sec:converg-distrib}, we shall sometimes work with extensions of the diffusion $Y_t$ or the scheme $\bar{X}_t$ beyond $D$. Yet, it is clear that all our results remain true as soon as $b$ and $\sigma$ only satisfy the condition that $b_{|D}$ and $\sigma_{|D}$ coincide with $\tilde{b}_{|D}$ and $\tilde{\sigma}_{|D}$, for some functions $\tilde{b} : \R^d \to \R^d$ and $\tilde{\sigma} : \R^d \to \mathbb{M}_{d,d}$ which satisfy Assumption~\ref{cond:coeffs}.
\end{rk}}

\subsection{Convergence in distribution of $\bar{X}_{\Gamma_n}$}\label{ss:cvdistrib}

The previous results show the almost sure convergence of the occupation measure of the discretization scheme $(\bar{X}_{\Gamma_n})_{n\ge0}$. Such a result is thus related to pathwise averages of this sequence. {Similarly to~\cite[Theorem~2.6]{BenCloPan18}, which is established for continuous-time Markov chains in compact state spaces,} we can also obtain a `spatial' counterpart, \emph{i.e.} the convergence in distribution of $(\bar{X}_{\Gamma_n})_{n\ge0}$. 

\begin{theo}[Convergence in distribution of $(\bar{X}_{\Gamma_{n}})_{n\ge0}$]\label{theo2} 
  Suppose that the assumptions of Theorem~\ref{theo:main} hold. Then, $(\bar{X}_{\Gamma_n})_{n\ge0}$ converges in distribution to $\mu^\star$.
\end{theo}
To deduce this result from  Theorem~\ref{theo:main}, the main tool is a {(nonquantitative)} control of the \emph{weak error in finite horizon} related to the dynamics induced in this problem. More precisely, to prove this convergence in distribution, our proof relies on Theorem~\ref{theo:weakerrorf} below which shows that the  weak error in finite time between the diffusion with renewal $\mu$ and its discretization with close \emph{{redistribution}} vanishes under {natural} conditions.

\medskip
For some $\mu \in \mathcal{M}_1(D)$, let us define the \emph{$\mu$-return process} associated with the SDE~\eqref{eq:SDE} in $D$ as follows. Starting from $x \in D$, we let $X^\mu_t$ evolve as the (strong) solution to~\eqref{eq:SDE} up to the first time $\tau_D$ at which it reaches $\partial D$, and then restart it from some position $X^\mu_{\tau_D}$ drawn independently in $D$ according to $\mu$. The construction goes on iteratively, with \emph{renewal epochs} at each time the process reaches $\partial D$. This defines a strong, cadlag Markov process $(X^\mu_t)_{t \geq 0}$ in $D$ whose semigroup is defined and denoted by $P^\mu_t f(x) = \Exp_x[f(X^\mu_t)]$.

The introduction of this process to study QSDs dates back to the first works dedicated to QSDs~\cite{Bar60,DarSen65}. In particular, Ferrari, Kesten, Martinez and Picco~\cite{FerKesMarPic95} noted that $\mu$ is a QSD if and only if it is stationary for the $\mu$-return process, {a fact that we shall also use in a slightly different formulation, see Lemma~\ref{lem:ergo-return}}. In this case, the $\mu^\star$-return process is also related to  stationary Fleming--Viot particle systems as it describes the limit dynamics of a tagged particle in such systems~\cite{GroJon13}, and its semigroup was noted to appear in the asymptotic variance of the fluctuations of these particle systems in~\cite{LelPilRey18}.

\medskip
Our next result establishes the consistency of Euler schemes with a redistribution mechanism when they attempt to exit $D$ toward the $\mu$-return process. We first provide a precise definition of such schemes.

\begin{defi}[Euler schemes with redistribution] \label{defi:eulerdistrib} 
  Let $\bfeta = (\eta_n)_{n\ge 1}$ be a sequence of positive numbers such that 
  \begin{equation*}
    \lim_{n \to +\infty} \dten_n = +\infty, \qquad \dten_n := \sum_{k=0}^{n-1} \eta_{k+1}.
  \end{equation*}
  On a probability space $(\Omega,\mathcal{F},\Pr)$, an \emph{Euler scheme with redistribution} is a triple $(({\cal G}_t)_{t\ge0}, \bfnu, (\bar{X}^{\bfeta,\bfnu}_t)_{t \geq 0})$ such that:
  \begin{itemize}
    \item $({\cal G}_t)_{t\ge0}$ is a filtration;
    \item $\bfnu = (\nu_n)_{n \geq 1}$ is a sequence of random variables in $\mathcal{M}_1(D)$ which is predictable with respect to the filtration $(\mathcal{G}_{\dten_n})_{n \geq 0}$;
    \item the random variable $\bar{X}^{\bfeta,\bfnu}_0$ is $\mathcal{G}_0$-measurable, and there exist a $(\mathcal{G}_t)_{t \geq 0}$-Brownian motion $(B_t)_{t \geq 0}$ and a $(\mathcal{G}_{\dten_n})_{n \geq 1}$-adapted sequence of random variables $(U^{\bfeta,\bfnu}_n)_{n \geq 1}$ such that, for any $n \geq 0$,
    \begin{equation*}
      \forall t \in [\dten_n, \dten_{n+1}), \qquad \bar{X}^{\bfeta,\bfnu}_t = \bar{X}^{\bfeta,\bfnu}_{\dten_n} + b\left(\bar{X}^{\bfeta,\bfnu}_{\dten_n}\right)(t-\dten_n) + \sigma\left(\bar{X}^{\bfeta,\bfnu}_{\dten_n}\right)\left(B_t-B_{\dten_n}\right),
    \end{equation*}
    and
    \begin{equation*}
      \bar{X}^{\bfeta,\bfnu}_{\dten_{n+1}} = \theta^{\bfeta,\bfnu}_{n+1} U^{\bfeta,\bfnu}_{n+1} + \left(1-\theta^{\bfeta,\bfnu}_{n+1}\right)\bar{X}^{\bfeta,\bfnu}_{\dten_{n+1}^-}, 
    \end{equation*}
    where $\theta^{\bfeta,\bfnu}_{n+1} = \ind{\bar{X}^{\bfeta,\bfnu}_{\dten_{n+1}^-}\not\in D}$ and conditionally on $\mathcal{G}_{\dten_{n+1}^-}$, the random variable $U^{\bfeta,\bfnu}_{n+1}$ has law~$\nu_n$.
  \end{itemize}
\end{defi}
Notice that for any $t \geq 0$, the random variable $\bar{X}^{\bfeta,\bfnu}_t$ only depends on the sequence $\bfnu$ through the elements $\nu_n$ for which $\dten_n \leq t$ and therefore the process $(\bar{X}^{\bfeta,\bfnu}_t)_{t \geq 0}$ is adapted to $(\mathcal{G}_t)_{t \geq 0}$. In the sequel, we shall often simply refer to the process $(\bar{X}^{\bfeta,\bfnu}_t)_{t \geq 0}$ as the Euler scheme with redistribution, and keep the filtration $(\mathcal{G}_t)_{t \geq 0}$, the step sequence $\bfeta$ and the redistribution sequence $\bfnu$ implicit.

\begin{rk} \label{rk:informationavantzero} For every $n\ge0$, $(\bar{X}_{\Gamma_n+t})_{t \geq 0} \overset{(d)}{=} (\bar{X}^{\bfeta,\bfnu}_t)_{t \geq 0}$ where $(\bar{X}^{\bfeta,\bfnu}_t)_{t \geq 0}$ is the Euler scheme with redistribution with parameters $\bfeta=(\gamma_{n+k})_{k\ge1}$, ${\cal G}_0=\sigma(\bar{X}_0,\ldots,\bar{X}_{\Gamma_n})$, ${\cal G}_t={\cal G}_0 \vee
\sigma((B_{\Gamma_n+s}-B_{\Gamma_n})_{0 \le s\le t}, (\bar{X}_{\Gamma_n+k})_{1 \leq k \leq t-\Gamma_n})$ and $\bfnu=(\pn_{n+k})_{k \geq 1}$. In particular, $\bar{X}^{\bfeta,\bfnu}_0=\bar{X}_{\Gamma_n}$.
\end{rk}

Let ${\cal L}(Y)$ denote the distribution of a random variable $Y$. In the next statement, we shall express our convergence result in terms of the $1$-Wasserstein distance on ${\cal M}_1(D)$, which is defined by
\begin{equation}\label{def:wass}
  {\cal W}_1(\mu,\nu)=\inf_{\Pi\in{\cal C}(\mu,\nu)}\int |x-y|\Pi(\dd x,\dd y),
\end{equation}
where
\begin{equation}\label{def:coupl-alpha-beta}
  {\cal C}(\mu,\nu):=\{\Pi\in {\cal M}_1(D\times D), \Pi(\cdot \times D)=\mu,\Pi(D \times \cdot)=\nu\}
\end{equation}
denotes the set of couplings of $\mu$ and $\nu$.

\begin{theo}[{Weak consistency of Euler schemes with redistribution}]\label{theo:weakerrorf}
  Assume~\ref{cond:D} and \ref{cond:coeffs}. Let $\mu\in{\cal M}_1(D)$. For any $n \geq 1$, let $\mu_0^n \in \mathcal{M}_1(D)$, and let $(\bar{X}^{\bfeta^n,\bfnu^n}_t)_{t \geq 0}$ be an Euler scheme with redistribution for some step sequence $\bfeta^n = (\eta^n_k)_{k \geq 1}$ and redistribution sequence $\bfnu^n = (\nu^n_k)_{k \geq 1}$. Denote $\nu_0^n = \mathcal{L}(\bar{X}^{\bfeta^n,\bfnu^n}_0)$, and assume that the sequence $(\mu_0^n)_{n \geq 1}$ is tight, and that
  \begin{equation*}
    \lim_{n \to +\infty} \sup_{k\ge1}\eta_k^n = 0; \qquad \lim_{n \to +\infty} \mathcal{W}_1(\mu_0^n,\nu_0^n) = 0;
  \end{equation*}
  and
  \begin{equation*}
    \forall \rho > 0, \quad \lim_{n \to +\infty} \PE\left(\sup_{k\ge1} {\cal W}_1\left(\nu^n_k,\mu\right)>\rho\right) = 0.
  \end{equation*}
  Then, for every $T>0$,
  \begin{equation*}
    \lim_{n \to +\infty} \sup_{t\in[0,T]}{\cal W}_1\left(\mu_0^n P_t^\mu, {\cal L}\left({{\bar X}^{\bfeta^n,\bfnu^n}_t}\right)\right) = 0.
  \end{equation*}
\end{theo}

Even if this result is `only' a tool for Theorem~\ref{theo2}, we chose to state it in this section since it may have some interest independently of the rest of the paper. This result is a direct corollary of Proposition~\ref{propdistrib} which is a slightly more uniform result (see Remark~\ref{rk:weakerrorrf}). The main difficulty for its proof comes from the dephasing between the jumps of the diffusion and its approximation and the strategy to overcome it is based on the derivation and resolution of a renewal inequality.
 
\subsection{Practical implementation}\label{subsec:practical} The interest of the sequence $(\pn_n)_{n\ge1}$ is to open the door to algorithms based on a redistribution measure which is less memory-consuming than the occupation measure $(\mu_n)_{n \geq 1}$. Before detailing such issues, let us consider the following example:
$$ \pn_n=\frac{1}{\sum_{k=0}^{n-1}\eta_{k+1,n}}\sum_{k=0}^{n-1} \eta_{k+1,n}{\delta_{\bar{X}_{\Gamma_k}}},$$
for some triangular array of positive steps $(\eta_{k,n})_{1 \leq k \leq n}$. In this case, one easily checks that Assumption~\ref{cond:pn} holds as soon as
\begin{equation}\label{eq:conditionpnn}
  \lim_{n \to +\infty} \frac{1}{\sum_{k=0}^{n-1}\eta_{k+1,n}}\sum_{k=0}^{n-1} |\eta_{k+1,n}-\gamma_{k+1}| = 0.
\end{equation}
If one wants to consider an approximation of $\mubar_n$ which is less memory-consuming, the natural idea is thus to forget the beginning of the path, $i.e.$ to consider 
$$ \pn_n=\frac{1}{\Gamma_n-\Gamma_{t(n)}}\sum_{k=t(n)}^{n-1}\gamma_{k+1}{\delta_{\bar{X}_{\Gamma_k}}},$$
where $(t(n))_{n\ge1}$ is a nondecreasing sequence of integers such that $0\le t(n)\le n-1$. In this case, one can deduce from \eqref{eq:conditionpnn} with $\eta_{k,n} = \gamma_k\ind{t(n)+1 \leq k \leq n}$ that Assumption~\ref{cond:pn} holds true if and only if
$$ \lim_{n \to +\infty} \frac{\Gamma_{t(n)}}{\Gamma_n} = 0.$$
If we further assume that $\gamma_n= Cn^{-\rho}$ with $\rho\in(0,1)$, the condition \eqref{eq:conditionpnn} is equivalent to 
$ t(n)=o(n)$ as $n\to+\infty$, {while if $\gamma_n= Cn^{-1}$ then the condition becomes $\log(t(n)) = o(\log(n))$}. 

One can observe that the gain of memory induced by this modification is unfortunately negligible with respect to the number of iterations. A probably more efficient alternative is to `quantize' the measure $\mubar_n$, $i.e$ to replace $\mubar_n$ by an approximation built on a partition of $D$: for a given $\varepsilon>0$, consider a partition $(A_\ell^{(\varepsilon)})_{\ell=1}^{L_\varepsilon}$ of $D$ and assume that $\max_\ell {\rm diam}(A_\ell^{(\varepsilon)})\le\varepsilon$. Set
$$\pn_n^{(\varepsilon)}=\sum_{\ell=1}^{L_\varepsilon} a_{\ell,n}^{(\varepsilon)}\pi_{\ell,{(\varepsilon)}},$$	
where $\pi_{\ell,{(\varepsilon)}}$ is a given distribution on $A_\ell^{(\varepsilon)}$  and
for each $\ell\in \{1,\ldots, L_\varepsilon\}$, 
$$a_{\ell,n}^{(\varepsilon)}=\frac{1}{\Gamma_n} \sum_{k=0}^{n-1} \gamma_{k+1}\ind{\bar{X}_{\Gamma_k} \in A_\ell^{(\varepsilon)}} = \mu_n(A_\ell^{(\varepsilon)}).$$
For instance, $\pi_{\ell,{(\varepsilon)}}$ can be the uniform distribution on $A_\ell^{(\varepsilon)}$ or a Dirac mass at a point $x_\ell^{(\varepsilon)}$ of $A_\ell^{(\varepsilon)}$. If $f$ is a Lipschitz continuous function on $D$, one easily checks that
$$|\pn_n^{(\varepsilon)}(f)-\mubar_n(f)|\le [f]_1 \varepsilon, $$
with $[f]_1$ the Lipschitz constant of $f$, so that if $(\varepsilon_n)_{n\ge1}$ converges to $0$ as $n\to+\infty$, then the sequence $(\pn_n^{(\varepsilon_n)})_{n\ge1}$ satisfies Assumption~\ref{cond:pn}.

In practice, one may fix $\varepsilon$ all along the simulation. In this case, one has `only' to keep a vector of length $L_\varepsilon$ for which each coordinate is the (non normalized) mass $\Gamma_n a_{\ell,n}^{(\varepsilon)}$ related to the area $A_\ell^{(\varepsilon)}$.   Quantifying how the final error depends on $\varepsilon$ may be an interesting task which is left to a future paper.

\subsection{Organization of the article and notation}\label{sec:orgnot}

The sequel of the paper is devoted to the proofs. Since Theorem~\ref{theo:main-mun} is a particular case of Theorem~\ref{theo:main}, we only prove Theorem~\ref{theo:main} whose proof is divided into two parts (as mentioned before). The almost sure tightness of the sequence $(\mu_n)_{n \geq 1}$ is proved in Section~\ref{sec:tightness} whereas the identification of the limit is achieved in Section~\ref{sec:prooftheomain} (see Subsection~\ref{ss:pf-main} for a synthesis). The proof of the convergence in distribution is the objective of Section~\ref{sec:converg-distrib}, which contains the proofs of~Theorem~\ref{theo2} and~Theorem~\ref{theo:weakerrorf}. 

Some proofs of results in these sections, which are either technical or close to standard arguments, are detailed in a Supplementary Material document, which is attached to this preprint. This document also contains auxiliary technical results on one-dimensional reflected Brownian motions, and various estimates related with the Euler discretization of the SDE~\eqref{eq:SDE}.

\medskip
Throughout the article, we denote by $|\cdot|$ the Euclidean norm on $\R^d$. For a function $f:\ER^d\rightarrow\ER$, $[f]_1$ and $\|f\|_\infty$ respectively denote its Lispchitz constant and its sup norm. 

For any $t \geq 0$, we denote by $\mathcal{F}_t$ the $\sigma$-algebra jointly generated by $(B_s)_{0 \leq s \leq t}$ and the family of random variables $\{\bar{X}_{\Gamma_n}, \Gamma_n \leq t\}$. In particular, the sequence of random probability measures $(\pn_n)_{n \geq 1}$ is $(\mathcal{F}_{\Gamma_n})_{n \geq 1}$-predictable, \emph{i.e.} $\pn_n$ is $\mathcal{F}_{\Gamma_{n-1}}$-measurable.

\section{Tightness}\label{sec:tightness}

We let $\psi_D : D \to (0,+\infty)$ be defined by
\begin{equation*}
  \forall x \in D, \qquad \psi_D(x) := \inf_{y \not\in D} |x-y|.
\end{equation*}
Since, by Assumption~\ref{cond:D}, $D$ is bounded, for any $\eta > 0$, the set $K_\eta := \{x \in D: \psi_D(x) \geq \eta\}$ is a compact subset of $D$. The following result is the main statement of this section.

\begin{prop}[Almost sure tightness of $(\mubar_n)_{n \geq 1}$]\label{prop:tightness}
If Assumptions~\ref{cond:D}, \ref{it:coeffs-1} and~\ref{cond:steps} hold true, then
  \begin{equation*}
    \lim_{\eta \dto 0} \limsup_{n \to +\infty} \mubar_n\left(D \setminus K_\eta\right) = 0, \qquad \text{almost surely.}
  \end{equation*}
\end{prop}

Remarkably, in this tightness result, we do not require all the assumptions of the main result. On the one hand, we do not need $b$ and $\sigma$ to be continuous. On the other hand, Proposition~\ref{prop:tightness} does not require Assumption~\ref{cond:pn} to hold, and thus holds for any choice of redistribution measure $\pn_n$ on $D$. As mentioned before, this property is rather unexpected. This means in particular that even in the case where the redistribution is mainly concentrated close to the boundary of $D$, the scheme $(\bar{X}_{\Gamma_n})_{n\ge0}$ does not spend much time in this area. Our proof shows that we have a reflection-type property which allows the (discretized) process to move away from the boundary.

Last, if Assumption~\ref{cond:pn} holds, Proposition~\ref{prop:tightness} easily implies the almost sure tightness of $(\pn_n)_{n \geq 1}$.

\begin{cor}[Almost sure tightness of $(\pn_n)_{n \geq 1}$]\label{cor:tightness-pn}
  Under the assumptions of Proposition~\ref{prop:tightness} and the supplementary Assumption~\ref{cond:pn}, we have
  \begin{equation*}
    \lim_{\eta \dto 0} \limsup_{n \to +\infty} \pn_n\left(D \setminus K_\eta\right) = 0, \qquad \text{almost surely.}
  \end{equation*}
\end{cor}
\begin{proof}
By Proposition~\ref{prop:tightness}, $(\mubar_n)_{n\ge1}$ is almost surely tight. By Prokhorov's theorem and Remark~\ref{rk:topoPolish} below, $(\mubar_n)_{n\ge1}$ is almost surely relatively compact. Then, \ref{cond:pn}
certainly implies that $(\pn_n)_{n\ge1}$ also is. By what precedes, $(\pn_n)_{n\ge1}$ is thus almost surely tight on $D$. The result follows.
\end{proof}
\begin{rk} \label{rk:topoPolish} It is well-known that the open set $D$ is topologically Polish: one can build a distance $\delta$ which is such that the metric space $(D,\delta)$ is separable and complete and whose induced topology is equivalent to the topology induced by the Euclidean distance on $D$ (for instance, one can set  $\delta(x,y)=|x-y|+|\frac{1}{d(x,\partial D)}-\frac{1}{d(y,\partial D)}|$).
\end{rk}

Setting $\xi_n := \psi_D(\bar{X}_{\Gamma_n}) \in (0,+\infty)$, Proposition~\ref{prop:tightness} rewrites equivalently
\begin{equation}\label{eq:prop-tightness-xi}
  \lim_{\eta \dto 0} \limsup_{n \to +\infty} \frac{1}{\Gamma_n}\sum_{k=0}^{n-1}\gamma_{k+1}\ind{\xi_k < \eta} = 0.
\end{equation}

The proof of~\eqref{eq:prop-tightness-xi} is divided into two parts. In Subsection~\ref{ss:coupling}, we construct a one-dimensional process $Z_{\tau^{-1}(t)}$ which essentially bounds from below the process $\psi_D(\bar{X}_t)$. Denoting by $\zeta_n$ the value of this process at time $\Gamma_n$, the main result of this first part, Corollary~\ref{cor:final-coupling}, states that $\zeta_n \leq \xi_n$ for $n$ large enough and thus allows to reduce the proof of~\eqref{eq:prop-tightness-xi} to a similar statement on the sequence $(\zeta_n)_{n \geq 0}$. This is the object of the second part of the proof, which is detailed in Subsection~\ref{ss:zeta}. 

\medskip
Throughout the section, we let Assumption~\ref{cond:D}, \ref{it:coeffs-1} and~\ref{cond:steps} be in force, and do not mention them in the statement of our results.

\subsection{Construction of the coupling}\label{ss:coupling} Let $(X_t)_{t \geq 0}$ be the continuous-time process introduced in Subsection~\ref{ss:motivation}. The tightness of its occupation measure is proved in~\cite[Proposition~4.1]{BenChaVil22}, using a coupling argument. More precisely, the starting point of the proof is the observation that there is $\eta_0 > 0$ such that, by Itô's formula, as long as $\psi_D(X_t) \in (0,\eta_0)$, this process behaves as an Itô process with bounded coefficients. When it reaches $0$, the redistribution mechanism for $X_t$ induces a positive jump for $\psi_D(X_t)$, and when $\psi_D(X_t)$ exceeds $\eta_0$, it may lose the Itô decomposition due to the possible lack of regularity of $\psi_D$. Therefore, applying a time change $t = \tau(r)$ to make the noise part of $\psi_D(X_t)$ additive, it is possible to construct a coupling between $\psi_D(X_{\tau(r)})$ and a drifted Brownian motion $Z_r$, reflected at the boundaries of the interval $(0,\eta_0)$, such that $\psi_D(X_{\tau(r)}) \geq Z_r$ for all $r \geq 0$, and to conclude by controlling the time spent by $Z_r$ near $0$. Our proof of Proposition~\ref{prop:tightness} follows the same lines but requires to take into account the discretization mechanism. The latter allows in particular $\bar{X}_t$ to exit from $D$ on the interval $[\Gamma_n, \Gamma_{n+1})$, so the (signed) distance of $\bar{X}_t$ to $\partial D$ may become negative, and so must be the bounding process $Z_r$ that will be constructed.

\subsubsection{Preliminary material} We first extend the function $\psi_D$ to the whole space $\R^d$ by letting
\begin{equation*}
  \forall x \not\in D, \qquad \psi_D(x) := -\inf_{y \in D} |x-y|.
\end{equation*}
Then $\psi_D : \R^d \to \R$ is the \emph{signed distance to $\partial D$}, which is positive in $D$ and negative in $\R^d \setminus \bar{D}$. 

\medskip
Since, by Assumption~\ref{cond:D}, $D$ is bounded with a $\mathcal{C}^2$-boundary, the following statement follows from~\cite[Lemma 14.16]{Gilbarg-Trudinger}.

\begin{lem}[Function $\tilde{\psi}_D$]\label{lem:tildepsi}
  There exist $\eta_0>0$ and a $\mathcal{C}^2$ function $\tilde{\psi}_D : \R^d \to \R$ such that:
  \begin{enumerate}[label=(\roman*),ref=\roman*]
    \item $\psi_D$ and $\nabla \psi_D$ coincide with $\tilde{\psi}_D$ and $\nabla \tilde{\psi}_D$  on the strip $\{x \in \R^d : |\psi_D(x)| \leq \eta_0\}$;
   
    \item $\tilde{\psi}_D$, $\nabla \tilde{\psi}_D$ and $\nabla^2 \tilde{\psi}_D$ are bounded on $\R^d$;
    \item $\psi_D(x)>\eta_0$ if and only if $\tilde{\psi}_D(x)>\eta_0$, and $\psi_D(x)<-\eta_0$ if and only if $\tilde{\psi}_D(x)<-\eta_0$.
  \end{enumerate}
\end{lem}
A particular consequence of the first assertion in Lemma~\ref{lem:tildepsi} is that for any $x \in \R^d$ such that $|\psi_D(x)| \leq \eta_0$, it follows from the \emph{eikonal equation} that 
\begin{equation}\label{eq:eikonal}
  |\nabla\tilde{\psi}_D(x)|=|\nabla\psi_D(x)|=1.
\end{equation}

By It\^o's formula, we then deduce that for any $n \geq 0$, for any $t \in [\Gamma_n,\Gamma_{n+1})$,
\begin{equation*}
  \tilde{\psi}_D(\bar{X}_t) = \tilde{\psi}_D(\bar{X}_{\Gamma_n}) + \int_{s=\Gamma_n}^t \tilde{K}_s \dd s + \int_{s=\Gamma_n}^t \tilde{H}_s \cdot \dd B_s,
\end{equation*}
where
\begin{align}\label{eq:itopsidtilde}
  \tilde{K}_s &:= b(\bar{X}_{\Gamma_n}) \cdot \nabla \tilde{\psi}_D(\bar{X}_s)  + \frac{1}{2} \mathrm{tr}\left(\sigma\sigma^\top(\bar{X}_{\Gamma_n})\nabla^2 \tilde{\psi}_D(\bar{X}_s)\right) \in \R,\\
  \tilde{H}_s &:= \sigma^\top(\bar{X}_{\Gamma_n}) \nabla \tilde{\psi}_D(\bar{X}_s) \in \R^d.
\end{align}

\subsubsection{Time change} In order to make the noise part in the Itô decomposition of $\tilde{\psi}_D(\bar{X}_t)$ additive, we define the time change $(\tau(r))_{r \geq 0}$ by the identity
\begin{equation*}
  \forall r \geq 0, \qquad \int_{s=0}^{\tau(r)} \left(|\tilde{H}_s|^2\ind{|\tilde{\psi}_D(\bar{X}_s)| \leq \eta_0} + \ind{|\tilde{\psi}_D(\bar{X}_s)| > \eta_0}\right)\dd s = r.
\end{equation*}
Since, by Assumption~\ref{it:coeffs-1}, $\sigma\sigma^\top$ is bounded and uniformly elliptic on $D$, and $\tilde{\psi}_D$ satisfies~\eqref{eq:eikonal}, there exists $0 < c_0 \leq 1$ such that, for any $r \geq 0$,
\begin{equation}\label{eq:c0}
  c_0 \leq \tau'(r) \leq \frac{1}{c_0},
\end{equation}
which allows to define the inverse function $\tau^{-1} : [0,+\infty) \to [0,+\infty)$.

For any $n \geq 0$, we set $\Delta_n = \tau^{-1}(\Gamma_n)$ and $\delta_{n+1} = \Delta_{n+1}-\Delta_n$. We then have
\begin{equation}\label{eq:gamma-delta}
  c_0 \Gamma_n \leq \Delta_n \leq \frac{\Gamma_n}{c_0}, \qquad c_0 \gamma_n \leq \delta_n \leq \frac{\gamma_n}{c_0}.
\end{equation}

For all $r \geq 0$, we set $\mathcal{G}_r := \mathcal{F}_{\tau(r)}$. Then $(\mathcal{G}_r)_{r \geq 0}$ is a filtration, with respect to which the sequence $(\Delta_n)_{n \geq 0}$ is an increasing sequence of stopping times. Besides, by the Dambis--Dubins--Schwarz theorem~\cite[Theorem~1.6, p.~181]{RevYor99} the process $(W_r)_{r \geq 0}$ defined by
\begin{equation*}
  W_r = \int_{s=0}^{\tau(r)} \left(\ind{|\tilde{\psi}_D(\bar{X}_s)| \leq \eta_0} \tilde{H}_s \cdot \dd B_s + \ind{|\tilde{\psi}_D(\bar{X}_s)| > \eta_0} \dd B^1_s\right),
\end{equation*}
where $B^1$ is the first coordinate of $B$, is a one-dimensional $(\mathcal{G}_r)_{r \geq 0}$-Brownian motion.

\subsubsection{The sequence $(\zeta_n)_{n \geq 0}$ and the process $(Z_r)_{r \geq 0}$} Combining the boundedness of $\sigma$ and $b$ on $D$ provided by Assumption~\ref{it:coeffs-1}, the boundedness of $\nabla\tilde{\psi}_D$ and $\nabla^2\tilde{\psi}_D$ provided by Lemma~\ref{lem:tildepsi}, and~\eqref{eq:c0}, we deduce that there exists $c_1 \geq 0$ such that for any $n \geq 0$, for any $\Delta_n \leq r' \leq r < \Delta_{n+1}$,
\begin{equation*}
  \left|\int_{s=\tau(r')}^{\tau(r)} \tilde{K}_s \dd s\right| \leq c_1(r-r').
\end{equation*}

We introduce the drifted Brownian motion $(\omega_r)_{r \geq 0}$ defined by $\omega_r := -c_1 r + W_r$, and for any $n \geq 0$, we set
\begin{equation*}
  \iota_{n+1} := \sup_{\Delta_n \leq u < v \leq \Delta_{n+1}} |\omega_v-\omega_u|.
\end{equation*}

\begin{lem}[Increments of $(\omega_r)_{r \geq 0}$ on the grid $(\Delta_n)_{n \geq 0}$]\label{lem:iota}
  Almost surely, $\displaystyle\lim_{n \to +\infty} \iota_n = 0$.
\end{lem}
\begin{proof}
  Let $n \geq 0$. By~\eqref{eq:gamma-delta}, we have
  \begin{equation*}
    \iota_{n+1} \leq \iota'_{n+1} := \sup_{\Delta_n \leq u < v \leq \Delta_n + c_0^{-1} \gamma_{n+1}} |\omega_v-\omega_u|.
  \end{equation*}
  Since $\Delta_n$ is a $(\mathcal{G}_r)_{r \geq 0}$-stopping time, while $(W_r)_{r \geq 0}$ is a $(\mathcal{G}_r)_{r \geq 0}$-Brownian motion, the strong Markov property yields 
  \begin{equation*}
    \iota'_{n+1} = \sup_{0 \leq u < v \leq c_0^{-1}\gamma_{n+1}} |\omega_v-\omega_u| \qquad \text{in distribution,}
  \end{equation*}
  and the right-hand side is bounded from above by $\kappa (c_0^{-1}\gamma_{n+1})^{1/4}$, where $\kappa$ is the (random) $1/4$-H\"older constant of the drifted Brownian motion $(\omega_r)_{r \geq 0}$ on the bounded and deterministic interval $[0, c_0^{-1} \sup_{n \geq 0} \gamma_{n+1}]$. Using the Markov inequality, we deduce that for any $\epsilon > 0$,
  \begin{equation*}
    \Pr\left(\iota_{n+1} \geq \epsilon\right) \leq \Pr\left(\kappa (c_0^{-1}\gamma_{n+1})^{1/4} \geq \epsilon\right) \leq \frac{(c_0^{-1} \gamma_{n+1})^p}{\epsilon^{4p}}\Exp[\kappa^{4p}],
  \end{equation*}
  where $p$ is given by Assumption~\ref{it:steps-2}. Since $\kappa$ has finite moments of all orders~\cite[Theorem~10.1, p.~152]{SchPar12}, the conclusion follows from the Borel--Cantelli lemma.
\end{proof}

\begin{rk}\label{rem:pas}
While Assumption~\ref{it:steps-1} is clearly fundamental in all the paper, it is worth noting that Assumption~\ref{it:steps-2} plays only a role in the proof of Lemma~\ref{lem:iota}. 
\end{rk}

In the sequel of the section, we shall define reflected Brownian motions in terms of the \emph{positive} and \emph{negative reflection maps} at $z \in \R$, constructed as follows: given a real-valued trajectory $\beta_\bullet = (\beta_r)_{r \in [r_0,r_1)}$ for some $-\infty < r_0 < r_1 \leq +\infty$, we set 
\begin{equation}\label{eq:positive-reflection-map}
  \boldsymbol{\Gamma}^{+,z}(\beta_\bullet) = Z^+_\bullet = (Z^+_r)_{r \in [r_0,r_1)}, \qquad Z^+_r := \beta_r + \max_{0 \leq u \leq r} [\beta_u-z]_- \quad \in [z,+\infty),
\end{equation}
and
\begin{equation}\label{eq:negative-reflection-map}
  \boldsymbol{\Gamma}^{-,z}(\beta_\bullet) = Z^-_\bullet = (Z^-_r)_{r \in [r_0,r_1)}, \qquad Z^-_r := \beta_r - \max_{0 \leq u \leq r} [\beta_u-z]_+ \quad \in (-\infty,z].
\end{equation}
Detailed properties of these reflection maps are gathered in Section~\ref{app:RBM}.

We may next define the sequence $(\zeta_n)_{n \geq 0}$ in $[0,\eta_0]$ and the process $(Z_r)_{r \geq 0}$ in $(-\infty,\eta_0]$ as follows:
\begin{itemize}
  \item $Z_0=\zeta_0=z_0$ for some $z_0 \in [0,\eta_0]$ which will be specified in the next subsection;
  \item for any $n \geq 0$, $(Z_r)_{r \in [\Delta_n, \Delta_{n+1})}$ is the one-dimensional Brownian motion started at $\zeta_n$, with constant drift $-c_1$, driven by $(W_r-W_{\Delta_n})_{r \in [\Delta_n, \Delta_{n+1})}$ and negatively reflected at the level $\eta_0$; in other words, with the notation ntroduced above, we set  
  \begin{equation*}
    Z_\bullet = \boldsymbol{\Gamma}^{-,\eta_0}(\zeta_n + \omega_\bullet - \omega_{\Delta_n}) \qquad \text{on $[\Delta_n, \Delta_{n+1})$;}
  \end{equation*}
  \item $\zeta_{n+1} = [Z_{\Delta_{n+1}^-}]_+$.
\end{itemize}
This construction is illustrated on Figure~\ref{fig:coupling}. It follows from this definition that $(Z_r)_{r \geq 0}$ is adapted to the filtration $(\mathcal{G}_r)_{r \geq 0}$.

\begin{center}
  \begin{figure}
    \includegraphics[width=\textwidth]{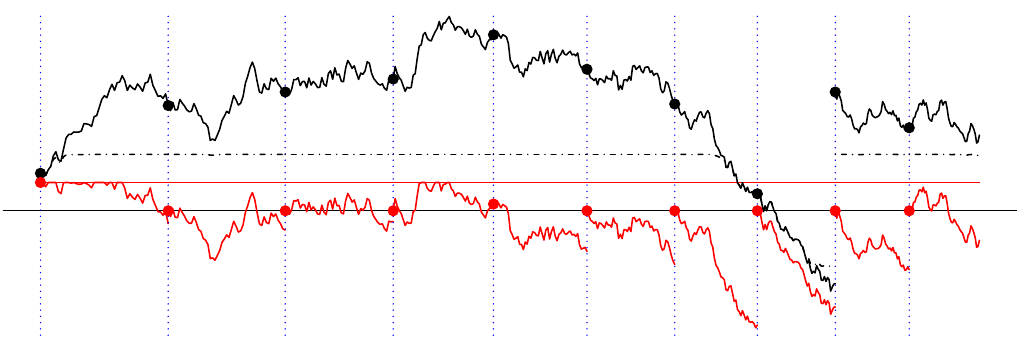}
    \caption{Construction of the process $(Z_r)_{r \geq 0}$. Vertical dotted blue lines indicate the times $\Delta_n$. The solid black curve is $\psi_D(\bar{X}_{\tau(r)})$, the dashed black curve is $\tilde{\psi}_D(\bar{X}_{\tau(r)})$. Black points represent the values of $\xi_n = \psi_D(\bar{X}_{\tau(\Delta_n)})$. The red curve is $Z_r$, the horizontal red line has coordinate $\eta_0$, and red points represent the values of $\zeta_n$.}
    \label{fig:coupling}
  \end{figure}
\end{center}

\begin{rk}\label{rk:edsZ}
  The evolution of the process $(Z_r)_{r \geq 0}$ can be concisely described by the reflected stochastic differential equation with jumps
  \begin{equation*}
    \dd Z_r = -c_1 \dd r + \dd W_r - \dd L^{Z_\bullet,\eta_0}_r + [Z_{r^-}]_-\dd \mathcal{N}^\Delta_r,
  \end{equation*}
  where $(L^{Z_\bullet,\eta_0}_r)_{r \geq 0}$ is the local time of $(Z_r)_{r \geq 0}$ at $\eta_0$ and $(\mathcal{N}^\Delta_r)_{r \geq 0}$ is the adapted counting process defined by
  \begin{equation*}
    \mathcal{N}^\Delta_r = \sum_{n=0}^{+\infty} \ind{\Delta_n \leq r}.
  \end{equation*}
\end{rk}
For any $n \geq 0$, we then let
\begin{equation*}
  \tilde{\Delta}_{n+1} := \inf\{r \in [\Delta_n,\Delta_{n+1}): \tilde{\psi}_D(\bar{X}_{\tau(r)}) \leq -\eta_0\},
\end{equation*}
with the convention that $\tilde{\Delta}_{n+1} = \Delta_{n+1}$ if $\tilde{\psi}_D(\bar{X}_{\tau(r)})$ remains above $-\eta_0$ on the interval $[\Delta_n,\Delta_{n+1})$.

\medskip
The first step of our coupling argument is detailed in the next statement.

\begin{lem}[Coupling $\psi_D(\bar{X}_{\tau(r)})$ with $Z_r$]\label{lem:coupling}
  For any $n \geq 0$, if $\zeta_n \leq \xi_n$, then for any $r \in [\Delta_n, \tilde{\Delta}_{n+1})$, $Z_r \leq \psi_D(\bar{X}_{\tau(r)})$.
\end{lem}
\begin{proof}
  Let $n \geq 0$ be such that $\zeta_n \leq \xi_n = \psi_D(\bar{X}_{\Gamma_n})$. Let $r \in [\Delta_n, \tilde{\Delta}_{n+1})$ and $r'$ be the largest time in $[\Delta_n, r]$ for which $\tilde{\psi}_D(\bar{X}_{\tau(r')}) \geq \eta_0$; if there is no such time, we set $r'=\Delta_n$. In both cases, it is easily checked that $\tilde{\psi}_D(\bar{X}_{\tau(r')}) \geq Z_{r'}$, so that if $r'=r$ then the claimed inequality is immediate. 
  
  If $r'<r$, by the definition of $r'$ and the fact that $r < \tilde{\Delta}_{n+1}$, we have $\tilde{\psi}_D(\bar{X}_{\tau(r)}) \in (-\eta_0,\eta_0)$ and thus
  \begin{equation*}
    \psi_D(\bar{X}_{\tau(r)}) = \tilde{\psi}_D(\bar{X}_{\tau(r)}) = \tilde{\psi}_D(\bar{X}_{\tau(r')}) + \int_{s=\tau(r')}^{\tau(r)} \tilde{K}_s \dd s + \int_{s=\tau(r')}^{\tau(r)} \tilde{H}_s \cdot \dd B_s.
  \end{equation*}
  The first term in the right-hand side is larger than $Z_{r'}$, the second term is larger than $-c_1(r-r')$, and since $\tilde{\psi}_D(\bar{X}_{\tau(u)})$ remains in $(-\eta_0,\eta_0]$ for $u \in [r',r]$, the third term coincides with $W_r-W_{r'}$. Therefore we get
  \begin{equation*}
    \psi_D(\bar{X}_{\tau(r)}) \geq Z_{r'} -c_1(r-r') + W_r - W_{r'} = Z_{r'} + \omega_r - \omega_{r'},
  \end{equation*}
  which by Proposition~\ref{prop:refmap}~\eqref{it:prop-refmap:accr} is larger than $Z_r$.
\end{proof}

We deduce from Lemma~\ref{lem:coupling} that if $\zeta_n \leq \xi_n$, then 
\begin{equation}\label{eq:ineq-Delta}
  \tilde{\Delta}_{n+1} \geq \inf\{r \in [\Delta_n, \Delta_{n+1}): Z_r \leq -\eta_0\} =: \Delta'_{n+1},
\end{equation}
with the same convention as before that $\Delta'_{n+1} = \Delta_{n+1}$ if $Z_r$ remains above $-\eta_0$ on the interval $[\Delta_n,\Delta_{n+1})$. 

\begin{lem}[Asymptotic behavior of $\Delta'_n$]\label{lem:Deltaprime}
  Almost surely, there exists $N \geq 1$ such that for any $n \geq N$, $\Delta'_{n+1}=\Delta_{n+1}$.
\end{lem}
\begin{proof}
  Assume that $n \geq 0$ is such that $\Delta'_{n+1} < \Delta_{n+1}$. Then $Z_{\Delta'_{n+1}} = -\eta_0$, and since $Z_{\Delta_n} = \zeta_n \geq 0$, the largest $r' \in [\Delta_n, \Delta'_{n+1})$ such that $Z_{r'}=0$ is well-defined. On the interval $[r',\Delta'_{n+1}]$, the reflection at $\eta_0$ does not act and therefore 
  \begin{equation*}
    -\eta_0 = Z_{\Delta'_{n+1}}-Z_{r'} = \omega_{\Delta'_{n+1}} - \omega_{r'}.
  \end{equation*}
  We deduce that if $\Delta'_{n+1} < \Delta_{n+1}$ then $\iota_{n+1} \geq \eta_0$, which by Lemma~\ref{lem:iota} implies that almost surely, $\Delta'_{n+1} = \Delta_{n+1}$ for $n$ large enough.
\end{proof}

Let $N \geq 1$ be given by Lemma~\ref{lem:Deltaprime}. If there exists $n \geq N$ such that $\zeta_n \leq \xi_n$, then by~\eqref{eq:ineq-Delta}, $\tilde{\Delta}_{n+1} = \Delta_{n+1}$ and therefore Lemma~\ref{lem:coupling} can be applied on the whole interval $[\Delta_n, \Delta_{n+1})$ to yield 
\begin{equation*}
  Z_{\Delta_{n+1}^-} \leq \tilde{\psi}_D(\bar{X}_{\Gamma_{n+1}^-}).
\end{equation*}
If $\tilde{\psi}_D(\bar{X}_{\Gamma_{n+1}^-}) \leq 0$ then $Z_{\Delta_{n+1}^-} \leq 0$ and therefore $\zeta_{n+1} = 0 \leq \xi_{n+1}$. Otherwise, $\bar{X}_{\Gamma_{n+1}^-} = \bar{X}_{\Gamma_{n+1}}$ and therefore we also have $\zeta_{n+1} \leq \xi_{n+1}$. Thus the argument may be iterated and yields $\zeta_{n+k} \leq \xi_{n+k}$ for any $k \geq 1$.

\medskip
The existence of $n \geq N$ such that $\zeta_n \leq \xi_n$ is ensured by the next lemma.

\begin{lem}[Regeneration times for $(\zeta_n)_{n \geq 0}$]\label{lem:regen-zeta}
  Almost surely, the set $\{n \geq 0: \zeta_n=0\}$ is unbounded.
\end{lem}
\begin{proof}
  For any $n \geq 0$, on the event $\{\forall n' > n, \zeta_{n'} > 0\}$ we have $Z_{\Delta_{n'}^-}>0$ for all $n' > n$ and therefore
  \begin{equation*}
    Z_\bullet = \boldsymbol{\Gamma}^{-,\eta_0}(\zeta_n + \omega_\bullet - \omega_{\Delta_n}), \qquad \text{on $[\Delta_n, +\infty)$,}
  \end{equation*}
  with the notation introduced in~\eqref{eq:negative-reflection-map}. As a consequence, by Corollary~\ref{cor:unboundZ}, there exists a finite time $r \geq \Delta_n$ such that $Z_r = -1$. Let $n' > n$ be such that $\Delta_{n'-1} \leq r < \Delta_{n'}$. Since $\zeta_{n'} = Z_{\Delta_{n'}^-} > 0$, we thus have
  \begin{equation*}
    \iota_{n'} \geq Z_{\Delta_{n'}^-}-Z_{r'} \geq 1,
  \end{equation*}
  which by Lemma~\ref{lem:iota} shows that $\Pr(\forall n' > n, \zeta_{n'} > 0)=0$. The result follows.
\end{proof}

We deduce from the discussion preceding Lemma~\ref{lem:regen-zeta} that almost surely, for $n$ large enough, $\zeta_n \leq \xi_n$. Combining this result with the bounds~\eqref{eq:gamma-delta} yields our final coupling estimate.

\begin{cor}[Final coupling estimate]\label{cor:final-coupling}
  Almost surely, for any $\eta > 0$, 
  \begin{equation*}
    \limsup_{n \to +\infty} \frac{1}{\Gamma_n} \sum_{k=0}^{n-1} \gamma_{k+1}\ind{\xi_k < \eta} \leq \frac{1}{c_0^2} \limsup_{n \to +\infty} \frac{1}{\Delta_n} \sum_{k=0}^{n-1} \delta_{k+1}\ind{\zeta_k < \eta}.
  \end{equation*}
\end{cor}

\begin{rk}\label{rk:z0}
  Since the coupling argument works as soon as $\zeta_n \leq \xi_n$ for some $n$ larger than the index $N$ given by Lemma~\ref{lem:Deltaprime}, and Lemma~\ref{lem:regen-zeta} ensures the existence of such an $n$, the choice of the initial value $z_0$ for $(Z_r)_{r \geq 0}$ and $(\zeta_n)_{n \geq 0}$ does not affect the results of this subsection. In the next subsection, where we will slice the trajectory of $(Z_r)_{r \geq 0}$ into segments between alternate crossings of the levels $\eta_0/3$ and $2\eta_0/3$, it will be convenient to start with $z_0 = \eta_0/3$.
\end{rk}

\subsection{Study of the sequence $(\zeta_n)_{n \geq 0}$}\label{ss:zeta} The aim of this subsection is to show that, almost surely,
\begin{equation}\label{eq:limsup-zeta}
  \lim_{\eta \dto 0} \limsup_{n \to +\infty} \frac{1}{\Delta_n} \sum_{k=0}^{n-1} \delta_{k+1}\ind{\zeta_k < \eta} = 0.
\end{equation}
By Corollary~\ref{cor:final-coupling}, this proves Proposition~\ref{prop:tightness}.

To proceed, in §~\ref{sss:SqTq} we shall set $Z_0=\eta_0/3$ and introduce intertwinned stopping times $0 = S_0 < T_0 < S_1 < T_1 < \cdots$ such that, on $[S_q,T_q)$, $Z_r$ goes from $\eta_0/3$ to $2\eta_0/3$, and on $[T_q,S_{q+1})$, it goes from $2\eta_0/3$ to $\eta_0/3$. As a consequence, on $[T_q,S_{q+1})$, $Z_r$ behaves as a drifted Brownian motion negatively reflected at $\eta_0$, while for step sizes $\Delta_n$ small enough, on $[S_q,T_q)$, $Z_r$ should be close to a drifted Brownian motion positively reflected at $0$, which we will denote by $Z^+_{q,r}$. The quantification of this assertion is an important technical point of our argument, it is stated in Proposition~\ref{prop:ZZpq}. Then, for large $n$ and $\eta < \eta_0/3$, the prelimit in the right-hand side of~\eqref{eq:limsup-zeta} should approximately coincide with the average time spent by $Z^+_{q,r}$ in $[0,\eta)$, on each interval $[S_q,T_q)$, divided by the average length of the interval $[T_q,S_{q+1})$. This statement is made precise in §~\ref{sss:framing}.

We refer to Figure~\ref{fig:SqTq} for an illustration of several quantities which are introduced in this subsection.

\begin{center}
  \begin{figure}
    \includegraphics[width=\textwidth]{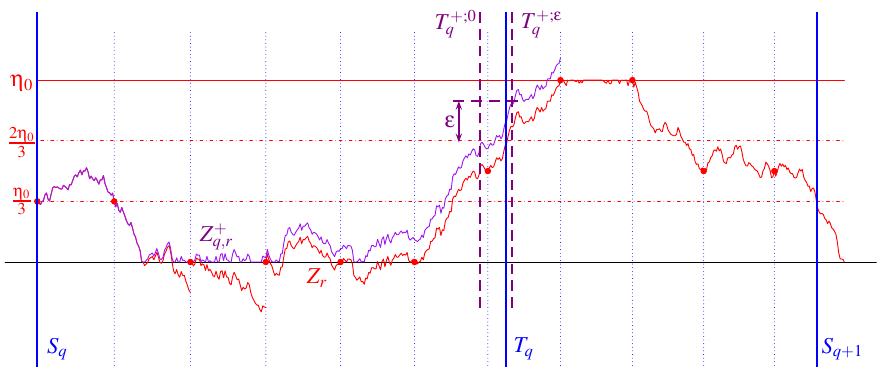}    
    \caption{The process $Z_r$ on $[S_q,S_{q+1})$. On the interval $[S_q,T_q)$, it is bounded from above by $Z^+_{q,r}$, which is reflected at $0$ and therefore does not depend on the grid $(\Delta_n)_{n \geq 0}$. The statement of Proposition~\ref{prop:ZZpq} is that the distance between $Z_r$ and $Z^+_{q,r}$ does not exceed $2\epsilon_q$. For $\epsilon>0$, the hitting times of the respective levels $2\eta_0/3$ and $2\eta_0/3+2\epsilon$ for $Z^+_{q,r}$ are denoted by $T^{+;0}_q$ and $T^{+;\epsilon}_q$.}
    \label{fig:SqTq}
  \end{figure}
\end{center}

\subsubsection{The sequences $(S_q)_{q \geq 0}$ and $(T_q)_{q \geq 0}$}\label{sss:SqTq} As is announced in Remark~\ref{rk:z0}, from now on, we set $z_0 := \eta_0/3$ in the definition of the process $(Z_r)_{r \geq 0}$ and of the sequence $(\zeta_n)_{n \geq 0}$, and we let $S_0 := 0$.

For any $q \geq 0$, assuming that the random variable $S_q \in [0,+\infty)$ has been defined and is such that $Z_{S_q} = \eta_0/3$, we set
\begin{equation*}
  T_q := \inf\left\{r \geq S_q: Z_r \geq \frac{2\eta_0}{3}\right\} \in (S_q, +\infty], 
\end{equation*}
and if $T_q < +\infty$,
\begin{equation*}
  S_{q+1} := \inf\left\{r \geq T_q: Z_r \leq \frac{\eta_0}{3}\right\} \in (T_q, +\infty].
\end{equation*}
Since $(Z_r)_{r \geq 0}$ is adapted and right-continuous, both $T_q$ and $S_{q+1}$ are stopping times for the filtration $(\mathcal{G}_r)_{r \geq 0}$. Besides, on each interval $[S_q,T_q)$, the negative reflection at $\eta_0$ does not act; while on each interval $[T_q,S_{q+1})$, there is no return to $0$ on the grid $(\Delta_n)_{n \geq 0}$.

\begin{rk}
  With the notation of Remark~\ref{rk:edsZ}, the process $(Z_r)_{r \geq 0}$ satifies
  \begin{align*}
    \dd Z_r &= -c_1 \dd r + \dd W_r + [Z_{r^-}]_-\dd \mathcal{N}^\Delta_r  &&\text{on $[S_q,T_q)$,}\\
    \dd Z_r &= -c_1 \dd r + \dd W_r - \dd L^{Z_\bullet,\eta_0}_r  &&\text{on $[T_q,S_{q+1})$.}
  \end{align*}
\end{rk}

As a consequence of this remark and the fact that, on the event $\{T_q < +\infty\}$, the Brownian motion $(W_r-W_{T_q})_{r \geq T_q}$ is independent from $\mathcal{G}_{T_q}$, we deduce that the variables $S_1-T_0, \ldots, S_{q+1}-T_q$ are independent copies of the hitting time of the level $\eta_0/3$ for a Brownian motion started at $2\eta_0/3$, negatively reflected at $\eta_0$, and with drift $-c_1$. By a symmetry argument, this random variable has the same law as the hitting time of the level $2\eta_0/3$ for a Brownian motion started at $\eta_0/3$, positively reflected at $0$, and with drift $c_1$. We denote by $\mathsf{T}^{c_1}_{\eta_0/3, 2\eta_0/3}$ such a random variable and refer to Subsection~\ref{ss:T-mbr} for details.

However, the same argument does not apply to the variables $T_0-S_0, \ldots, T_q-S_q$ because on each interval $[S_q,T_q)$, the process $Z_\bullet$ not only depends on the randomness induced by the Brownian motion $(W_r-W_{S_q})_{r \geq S_q}$, but also on all stopping times $\Delta_n$ in the interval, which need not be such that $\Delta_n - S_q$ is independent from $\mathcal{G}_{S_q}$. To remove this dependency, we denote by $(Z^+_{q,r})_{r \geq S_q}$ the Brownian motion started in $\eta_0/3$ at $r=S_q$, with drift $-c_1$, driven by the Brownian motion $(W_r-W_{S_q})_{r \geq S_q}$, and positively reflected at $0$.
In other words, we have
\begin{equation*}
  Z^+_{q,\bullet} = \boldsymbol{\Gamma}^{+,0}\beta_\bullet, \quad \beta_\bullet := \frac{\eta_0}{3} + \omega_\bullet - \omega_{S_q} \qquad \text{on $[S_q, +\infty)$},
\end{equation*}
with the notation of~\eqref{eq:positive-reflection-map}. This process no longer depends on the stopping times $\Delta_n$ larger than $S_q$, and therefore it is independent from $\mathcal{G}_{S_q}$. The next proposition, whose proof is detailed in Subsection~\ref{ss:pf-ZZpq}, shows that statistics of $Z_\bullet$ on $[S_q,T_q)$ are well approximated by statistics of $Z^+_{q,\bullet}$.

\begin{prop}[Estimate between $Z^+_{q,r}$ and $Z_r$]\label{prop:ZZpq}
  For any $q \geq 0$, let
  \begin{equation*}
    n_q := \min\{n \geq 0: \Delta_n \geq S_q\} \leq m_q := \min\{n \geq 0: \Delta_n \geq T_q\},
  \end{equation*}
  and with the notation of Lemma~\ref{lem:iota}, let 
  \begin{equation*}
  	\epsilon_q := \sup_{n_q \leq n \leq m_q} \iota_n.
  \end{equation*}
  For any $r \in [S_q,T_q)$, we have 
  \begin{equation*}
    0 \leq Z^+_{q,r} - Z_r \leq 2 \epsilon_q.
  \end{equation*}
\end{prop}

We gather the results from this paragraph in the following final statement.

\begin{cor}[Estimates on $(S_q)_{q \geq 0}$ and $(T_q)_{q \geq 0}$]\label{cor:SqTq}
  \begin{enumerate}[label=(\roman*),ref=\roman*]
    \item\label{it:cor-SqTq:1} Almost surely, the stopping times $S_q$ and $T_q$ are finite and they grow to $+\infty$ when $q \to +\infty$.
    \item\label{it:cor-SqTq:2} The random variables $(S_{q+1}-T_q)_{q \geq 0}$ are independent copies of {$\mathsf{T}^{c_1}_{ \eta_0/3,2\eta_0/3}$}.
    \item\label{it:cor-SqTq:3} The random variable $T_q-S_q$ satisfies the estimate
    \begin{equation*}
      T^{+;0}_q-S_q \leq T_q-S_q \leq T^{+;\epsilon_q}_q-S_q,
    \end{equation*}
    where $\epsilon_q$ is defined in Proposition~\ref{prop:ZZpq} and for any $\epsilon \geq 0$,
    \begin{equation*}
      T^{+;\epsilon}_q := \inf\left\{r \geq S_q: Z^+_{q,r} \geq \frac{2\eta_0}{3}+2\epsilon\right\} \in (S_q, +\infty)
    \end{equation*}
    is such that $T^{+;\epsilon}_q-S_q$ is independent from $\mathcal{G}_{S_q}$ and has the law of $\mathsf{T}^{-c_1}_{\eta_0/3, 2\eta_0/3+2\epsilon}$.
    \item\label{it:cor-SqTq:4} Almost surely, $\lim_{q \to +\infty} \epsilon_q = 0$.
  \end{enumerate}
\end{cor}
\begin{proof}
  Let $q \geq 0$ be such that $S_q < +\infty$ and $Z_{S_q} = \eta_0/3$. As a consequence of Lemma~\ref{lem:iota}, $\epsilon_q < +\infty$. Therefore Corollary~\ref{cor:unboundZ} shows that $T^{+;\epsilon_q}_q < +\infty$. It now follows from Proposition~\ref{prop:ZZpq} that
  \begin{equation*}
    T^{+;0}_q \leq T_q \leq T^{+;\epsilon_q}_q,
  \end{equation*}
  which implies in particular that $T_q < +\infty$. On the other hand, $S_{q+1}-T_q$ has the law of $\mathsf{T}^{c_1}_{\eta_0/3, 2\eta_0/3}$ (using a symmetry argument) and therefore it is finite, almost surely. We deduce that $S_{q+1} < +\infty$ and $Z_{S_{q+1}}=\eta_0/3$, which allows to show by induction that $S_q$ and $T_q$ are finite for all $q \geq 0$.
  
  In addition, for any $q \geq 0$, $S_{q+1}-T_q$ is independent from $\mathcal{G}_{T_q}$, and thus from $\mathcal{G}_{S_q}$, while it is $\mathcal{G}_{S_{q+1}}$-measurable. Thus, the variables $(S_{q+1}-T_q)_{q \geq 0}$ are independent copies of $\mathsf{T}^{c_1}_{\eta_0/3, 2\eta_0/3}$. As a consequence, the sequences $(S_q)_{q \geq 0}$ and $(T_q)_{q \geq 0}$ do not accumulate and therefore they grow to $+\infty$ when $q \to +\infty$. This completes the proof of the points~\eqref{it:cor-SqTq:1},~\eqref{it:cor-SqTq:2} and~\eqref{it:cor-SqTq:3}. We last deduce from~\eqref{it:cor-SqTq:1} that the indices $n_q$ and $m_q$ defined in Proposition~\ref{prop:ZZpq} grow to infinity when $q \to +\infty$, and therefore by Lemma~\ref{lem:iota} we have
  \begin{equation*}
    \lim_{q \to +\infty} \epsilon_q \leq \lim_{n \to +\infty} \sup_{n' \geq n} \iota_{n'} = 0, \qquad \text{almost surely,}
  \end{equation*}
  which yields the point~\eqref{it:cor-SqTq:4}.
\end{proof}

\begin{rk}\label{rk:Tp0q}
  It follows from Corollary~\ref{cor:SqTq}~\eqref{it:cor-SqTq:3} that for any $q \geq 0$, $\mathcal{G}_{T^{+;0}_q} \subset \mathcal{G}_{T_q} \subset \mathcal{G}_{S_{q+1}}$, therefore the random variables $(T^{+;0}_q-S_q)_{q \geq 0}$ are independent copies of $\mathsf{T}^{-c_1}_{\eta_0/3, 2\eta_0/3}$. However, as soon as $\epsilon>0$, since $T^{+;\epsilon}_q$ may \emph{a priori} be larger than $S_{q+1}$, the variables $(T^{+;\epsilon}_q-S_q)_{q \geq 0}$ are not necessarily independent.
\end{rk}

\subsubsection{Framing~\eqref{eq:limsup-zeta}}\label{sss:framing} For any $q \geq 0$, we set
\begin{equation*}
  \bar{R}_q := \sum_{k=0}^{+\infty} \ind{S_q \leq \Delta_k < S_{q+1}} \delta_{k+1}, \qquad R_q(\eta) := \sum_{k=0}^{+\infty} \ind{S_q \leq \Delta_k < S_{q+1}} \delta_{k+1} \ind{\zeta_k < \eta}.
\end{equation*}
Since the sequences $(S_q)_{q \geq 0}$ and $(T_q)_{q \geq 0}$ are almost surely well-defined and do not accumulate, these quantities are almost surely well-defined and finite; besides, for any $n \geq 0$, denoting by $Q_n \geq 0$ the unique index such that $S_{Q_n} \leq \Delta_n < S_{Q_{n+1}}$, 
we obtain the two-sided inequalities
\begin{equation}\label{eq:two-sided}
  \sum_{q=0}^{Q_n-1} \bar{R}_q \leq \Delta_n \leq \sum_{q=0}^{Q_n} \bar{R}_q, \qquad \sum_{q=0}^{Q_n-1} R_q(\eta) \leq \sum_{k=0}^{n-1} \delta_{k+1}\ind{\zeta_k < \eta} \leq \sum_{q=0}^{Q_n} R_q(\eta).
\end{equation}

Corollary~\ref{cor:SqTq} allows us to prove the following estimates, whose proof is detailed in Subsection~\ref{ss:pf-compR-LLN}.

\begin{lem}[Comparison for $\bar{R}_q$ and $R_q(\eta)$]\label{lem:compR}
  Let us define
  \begin{equation*}
    \bar{R}'_q := S_{q+1}-S_q, \qquad R'_q(\eta) := \int_{r=S_q}^{T_q} \ind{Z^+_{q,r} < \eta}\dd r, \qquad R''_q(\eta,\epsilon) := R'_q(\eta+\epsilon)-R'_q(\eta-\epsilon),
  \end{equation*}
  where $0 < \epsilon < \eta$. We have
  \begin{equation*}
    \lim_{q \to +\infty} \left|\bar{R}_q - \bar{R}'_q\right| = 0, \qquad \text{almost surely,}
  \end{equation*}
  and, for any $\eta \in (0, \eta_0/3]$, for any $0 < \epsilon < \eta$,
  \begin{equation*}
    \limsup_{q \to +\infty} \left(\left|R_q(\eta) - R'_q(\eta)\right|- R''_q(\eta,\epsilon)\right) \leq 0, \qquad \text{almost surely.}
  \end{equation*}
\end{lem}

To state our last intermediary result, we denote by $\mathsf{Z}^{+,0;c}_{z_0,\bullet}$ the Brownian motion started at $z_0 \geq 0$, positively reflected at $0$, with drift $c$, so that
\begin{equation*}
  \mathsf{T}^c_{z_0,z_1} = \inf\{r \geq 0: \mathsf{Z}^{+,0;c}_{z_0,r} = z_1\}, \qquad z_1 \geq z_0,
\end{equation*}
and introduce the notation
\begin{equation*}
  \bar{\mathcal{R}}^c_{z_0,z_1} = \Exp\left[\mathsf{T}^c_{z_0,z_1}\right], \qquad \mathcal{R}^c_{z_0,z_1}(\eta) = \Exp\left[\int_{r=0}^{\mathsf{T}^c_{z_0,z_1}}\ind{\mathsf{Z}^{+,0;c}_{z_0,r}<\eta}\dd r\right].
\end{equation*}

\begin{lem}[Laws of Large Numbers]\label{lem:LLN}
  With the notation introduced above, 
  \begin{equation*}
    \lim_{Q \to +\infty} \frac{1}{Q}\sum_{q=0}^{Q-1} \bar{R}'_q = \bar{\mathcal{R}}^{c_1}_{\eta_0/3, 2\eta_0/3} + \bar{\mathcal{R}}^{-c_1}_{\eta_0/3, 2\eta_0/3}, \qquad \text{almost surely,}
  \end{equation*}
  and, for any $\eta>0$,
  \begin{equation*}
    \lim_{Q \to +\infty} \frac{1}{Q}\sum_{q=0}^{Q-1} R'_q(\eta) = \mathcal{R}^{-c_1}_{\eta_0/3, 2\eta_0/3}(\eta), \qquad \text{almost surely.}
  \end{equation*}
\end{lem}
The proof of Lemma~\ref{lem:LLN} is detailed in Subsection~\ref{ss:pf-compR-LLN}.

We are now ready to complete the proof of the estimate~\eqref{eq:limsup-zeta}. By the Ces\`aro Lemma, we deduce from Lemmas~\ref{lem:compR} and~\ref{lem:LLN} that 
\begin{equation}\label{eq:finpflimsupzeta:1}
  \lim_{Q \to +\infty} \frac{1}{Q}\sum_{q=0}^{Q-1} \bar{R}_q = \lim_{Q \to +\infty} \frac{1}{Q}\sum_{q=0}^{Q-1} \bar{R}'_q = \bar{\mathcal{R}}^{c_1}_{\eta_0/3, 2\eta_0/3} + \bar{\mathcal{R}}^{-c_1}_{\eta_0/3, 2\eta_0/3}, \qquad \text{almost surely.}
\end{equation}
Similarly, we get, for any $0 < \epsilon < \eta \leq \eta_0/3$,
\begin{equation*}
  \limsup_{Q \to +\infty} \left|\frac{1}{Q}\sum_{q=0}^{Q-1} R_q(\eta) - \mathcal{R}^{-c_1}_{\eta_0/3, 2\eta_0/3}(\eta)\right| \leq \limsup_{Q \to +\infty}\frac{1}{Q}\sum_{q=0}^{Q-1} R''_q(\eta,\epsilon), \qquad \text{almost surely.}
\end{equation*}
On the other hand, using Lemma~\ref{lem:LLN} again yields
\begin{equation*}
  \lim_{Q \to +\infty} \frac{1}{Q}\sum_{q=0}^{Q-1}R''_q(\eta,\epsilon) = \mathcal{R}^{-c_1}_{\eta_0/3, 2\eta_0/3}(\eta+\epsilon)-\mathcal{R}^{-c_1}_{\eta_0/3, 2\eta_0/3}(\eta-\epsilon), \qquad \text{almost surely,}
\end{equation*}
and by Corollary~\ref{cor:estimT}, the right-hand vanishes when $\epsilon \dto 0$. Therefore we finally get
\begin{equation}\label{eq:finpflimsupzeta:2}
  \lim_{Q \to +\infty} \frac{1}{Q}\sum_{q=0}^{Q-1} R_q(\eta) = \mathcal{R}^{-c_1}_{\eta_0/3, 2\eta_0/3}(\eta), \qquad \text{almost surely.}
\end{equation}

By~\eqref{eq:two-sided} and since $Q_n \to +\infty$ with $n$, we conclude from~\eqref{eq:finpflimsupzeta:1} and~\eqref{eq:finpflimsupzeta:2} that for $\eta \in (0,\eta_0/3]$,
\begin{equation*}
  \lim_{n \to +\infty} \frac{1}{\Delta_n} \sum_{k=0}^{n-1} \delta_{k+1}\ind{\zeta_k < \eta} = \frac{\mathcal{R}^{-c_1}_{\eta_0/3, 2\eta_0/3}(\eta)}{\bar{\mathcal{R}}^{c_1}_{\eta_0/3, 2\eta_0/3} + \bar{\mathcal{R}}^{-c_1}_{\eta_0/3, 2\eta_0/3}}, \qquad \text{almost surely,}
\end{equation*}
which by Corollary~\ref{cor:estimT} yields~\eqref{eq:limsup-zeta} and completes the proof of Proposition~\ref{prop:tightness}.

\section{Proof of Theorems~\ref{theo:main-mun} and~\ref{theo:main}}\label{sec:prooftheomain}
Theorem~\ref{theo:main-mun} being a particular case of~Theorem~\ref{theo:main}, we only prove~Theorem~\ref{theo:main}. To proceed, we introduce the discrete and continuous renewal times
\begin{equation*}
  \mathfrak{n}_0 := 0, \qquad \mathfrak{n}_j := \min\{n > \mathfrak{n}_{j-1}: \theta_n = 1\}, \qquad \tauel_j := \Gamma_{\mathfrak{n}_j},
\end{equation*}
with $\theta_n$ defined in~\eqref{eq:theta}. It follows from Lemma~\ref{lem:controlmomentstempsarret} and the strong Markov property that the set $\{n \geq 1: \theta_n=1\}$ is almost surely unbounded, so that $\mathfrak{n}_j$ and $\tauel_j$ are well defined for any $j \geq 0$, and under Assumption~\ref{it:steps-1}, $\tauel_j \to +\infty$ when $j \to +\infty$.

We then introduce the empirical measure of the Euler scheme at renewal times
\begin{equation*}
  \forall \ell \geq 1, \qquad \vta_\ell := \frac{1}{\ell} \sum_{j=1}^\ell \delta_{\bar{X}_{\tauel_j}},
\end{equation*}
and first show that the almost sure tightness of $(\pn_n)_{n \geq 1}$ implies that of $(\vta_\ell)_{\ell \geq 1}$.

\begin{lem}[Tightness of $(\vta_\ell)_{\ell \geq 1}$]\label{lem:tightness-vartheta}
  Under the assumptions of Corollary~\ref{cor:tightness-pn}, the sequence $(\vta_\ell)_{\ell \geq 1}$ is almost surely tight.
\end{lem}
\begin{proof}
  Since ${\bf p}_n$ denotes the redistribution measure, it follows from the strong Markov property that for any bounded and measurable function $f:D \to \ER$ and $j \geq 1$,
  \begin{equation*}
    \Exp\left[f(\bar{X}_{\tauel_j})-\pn_{\mathfrak{n}_j} (f) | \mathcal{F}_{\tauel_{j}^{-}}\right] = 0.
  \end{equation*}
  Therefore, the sequence $(f(\bar{X}_{\tauel_j})-\pn_{\mathfrak{n}_j} (f))_{j \geq 1}$ is a martingale difference sequence for the filtration $(\mathcal{F}_{\tauel_{j+1}^{-}})_{j \geq 1}$. Since it is bounded, the strong Law of Large Numbers for martingale difference sequences yields
  \begin{equation*}
    \lim_{\ell \to +\infty} \frac{1}{\ell}\sum_{j=1}^\ell \left(f(\bar{X}_{\tauel_j})-\pn_{\mathfrak{n}_j} (f)\right) = \lim_{\ell \to +\infty} \left(\vta_\ell(f) - \frac{1}{\ell}\sum_{j=1}^\ell\pn_{\mathfrak{n}_j} (f)\right) = 0, \qquad \text{almost surely.}
  \end{equation*}
  The almost sure tightness of $(\vta_\ell)_{\ell \geq 1}$ then easily follows from Corollary~\ref{cor:tightness-pn}.
\end{proof}

The sequel of the argument uses the same strategy as in~\cite[Sections~4.2--4.4]{BenChaVil22} and relies on the introduction of two operators $A$ and $\Pi$. For any measurable and bounded function $f: D \to \R$, we let $Af$ be defined by
\begin{equation*}
  \forall x \in D, \qquad Af(x) = \Exp_x\left[\int_0^{\tau_D} f(Y_t)\dd t\right].
\end{equation*}
The operator $A$ is regularizing: by~\cite{BenChaVil22}, under Assumptions~\ref{cond:D} and~\ref{cond:coeffs}, for any measurable and bounded function $f : D \to \R$, the function $Af$ is bounded and Lipschitz continuous on $D$. Besides, for any $\mu \in \mathcal{M}_1(D)$, we have $\mu A \iind{D}>0$, which allows us to define 
\begin{equation*}
  \Pi_\mu = \frac{\mu A}{\mu A \iind{D}} \in \mathcal{M}_1(D).
\end{equation*}

The remainder of this section is organized as follows. As a preliminary, we prove the convergence to $A$ of its natural discretization in Subsection~\ref{ss:estim-tau}. Then, the proof of Theorem~\ref{theo:main} works in two main steps. First, we prove in Subsection~\ref{ss:mu-vartheta} that if $\vta_\ell$ converges to some measure $\nu$, then $\mubar_n$ converges to $\Pi_\nu$. Second, we use the notion of \emph{asymptotic pseudo-trajectory} to show in Subsection~\ref{ss:cv-vartheta} that $\vta_\ell$ converges to $\mu^\star$. Using the remark that $\Pi_{\mu^\star}=\mu^\star$, we complete the proof of Theorem~\ref{theo:main} in Subsection~\ref{ss:pf-main}.

\subsection{Discretization of the operator $A$}\label{ss:estim-tau}

Let Assumptions~\ref{cond:D} and~\ref{cond:coeffs} be in force. Moreover, let $\bfeta = (\eta_n)_{n \geq 1}$ be a sequence of positive time steps, denote
\begin{equation*}
  |\bfeta| := \sup_{n \geq 1} \eta_n, \qquad \dten_n := \sum_{k=0}^{n-1}\eta_{k+1},
\end{equation*}
{and assume that $\lim_{n \to +\infty} \dten_n = +\infty$.} {We denote by $(\bar{Y}^\bfeta_t)_{t \geq 0}$ the continuous-time Euler scheme associated with the SDE~\eqref{eq:SDE}, with the same initial initial condition $\bar{Y}_0^\bfeta = Y_0 \in D$ and driven by the same Brownian motion $(B_t)_{t \geq 0}$. We next set
\begin{equation}\label{eq:taubareta}
  \bar{\tau}^\bfeta_D := \inf\{\dten_n : \bar{Y}^\bfeta_{\dten_n} \not\in D\} = \inf\{t \geq 0: \bar{Y}^\bfeta_{\un{t}} \not\in D\},
\end{equation}
with $\un{t} := \dten_n$ for any $t \in [\dten_n, \dten_{n+1})$.} This allows us to define, for any measurable and bounded function $f : D \to \R$,
\begin{equation}\label{def:Aeta}
  \forall x \in D, \qquad \bar{A}^\bfeta f(x) := \Exp_x\left[\int_0^{\bar{\tau}^\bfeta_D} f(\bar{Y}^\bfeta_{\un{t}})\dd t\right].
\end{equation}
It follows from Lemma~\ref{lem:controlmomentstempsarret} that the function $\bar{A}^\bfeta f$ is bounded on $D$. The main result of this subsection is the following statement.

\begin{prop}[Convergence of $\bar{A}^\bfeta f$]\label{prop:weakA}
  Assume~\ref{cond:D} and~\ref{cond:coeffs}. Let $\bfeta$ be a step sequence {such that
  \begin{equation*}
    \lim_{n \to +\infty} \dten_n = +\infty \qquad \text{and} \qquad \sup_{n \geq 1} \frac{\eta_{n+1}}{\eta_n} < +\infty.
  \end{equation*}}
  Then, for any measurable and bounded function $f : D \to \R$,
  \begin{equation*}
    \lim_{|\bfeta| \to 0} \sup_{x \in D}|\bar{A}^\bfeta f(x) - Af(x)| = 0.
  \end{equation*}
\end{prop}
The proof of Proposition~\ref{prop:weakA} relies on various discretization estimates, which are gathered in Section~\ref{app:discr} and may be of independent interest. 

\begin{proof}[Proof of Proposition~\ref{prop:weakA}]
  Let $f : D \to \R$ be a measurable and bounded function, which we extend by $0$ on $\R^d \setminus D$. It is well-known that there exists a sequence $(f_p)_{p \geq 1}$ of Lipschitz continuous functions $f_p : \R^d \to \R$ which are such that $\|f_p\|_\infty \leq \|f\|_\infty$ and $f_p(x) \to f(x)$ when $p \to +\infty$, $\dd x$-almost everywhere. Writing
  \begin{equation*}
    |\bar{A}^\bfeta f(x) - Af(x)| \leq |\bar{A}^\bfeta (f-f_p)(x)| + |\bar{A}^\bfeta f_p(x) - Af_p(x)| + |A(f-f_p)(x)| 
  \end{equation*}
  and applying Lemma~\ref{lem:weakerror} to $f_p$, we observe that to prove the proposition it suffices to show that
  \begin{equation}\label{eq:weakA-1}
    \lim_{p \to +\infty} \left(\sup_{x \in D} |A(f-f_p)(x)| + \limsup_{|\bfeta| \to 0} \sup_{x \in D} |\bar{A}^\bfeta (f-f_p)(x)|\right) = 0.
  \end{equation}
  
  We first address the term $|\bar{A}^\bfeta (f-f_p)(x)|$. We fix $0 < t_1 < t_2 < +\infty$, and write
  \begin{align*}
    \left|\bar{A}^\bfeta (f-f_p)(x)\right| &= \left|\int_0^{+\infty} \Exp_x\left[(f-f_p)(\bar{Y}^\bfeta_{\un{t}})\ind{t < \bar{\tau}^\bfeta_D}\right] \dd t\right|\\
    &\leq 2 \|f\|_\infty t_1 + \int_{t_1}^{t_2} \Exp_x\left[|f-f_p|(\bar{Y}^\bfeta_{\un{t}})\ind{\bar{Y}^\bfeta_{\un{t}} \in D}\right] \dd t + 2 \|f\|_\infty \int_{t_2}^{+\infty} \Pr_x\left(t < \bar{\tau}^\bfeta_D\right)\dd t.
  \end{align*}
  For $\delta > 0$, one may fix $t_1$ small enough for the inequality $2 \|f\|_\infty t_1 \leq \delta/3$ to hold, and by Lemma~\ref{lem:controlmomentstempsarret}, one may fix $t_2$ large enough for the inequality 
  \begin{equation*}
    \limsup_{|\bfeta| \to 0} \sup_{x \in D} 2 \|f\|_\infty \int_{t_2}^{+\infty} \Pr_x\left(t < \bar{\tau}^\bfeta_D\right)\dd t \leq \frac{\delta}{3}
  \end{equation*}
  to hold. Last, with the notation of Lemma~\ref{lem:gauss-density}, as soon as $\eta_1 \leq t_1$ we have
  \begin{align*}
    \int_{t_1}^{t_2} \Exp_x\left[|f-f_p|(\bar{Y}^\bfeta_{\un{t}})\ind{\bar{Y}^\bfeta_{\un{t}} \in D}\right] \dd t \leq \int_{y \in D} |f(y)-f_p(y)| \int_{t_1}^{t_2} \bar{p}^\bfeta_{\un{t}}(x,y) \dd t \dd y.
  \end{align*}
  Now, by the Gaussian upper bound from Lemma~\ref{lem:gauss-density}, it follows that some positive  $C_{t_1}$ and $\bar{\eta}$ exist such that  if $|\bfeta|\le \bar{\eta}$, then
  \begin{align*}
    \int_{t_1}^{t_2} \Exp_x\left[|f-f_p|(\bar{Y}^\bfeta_{\un{t}})\ind{\bar{Y}^\bfeta_{\un{t}} \in D}\right] \dd t \leq C_{t_1} \int_{y \in D} |f(y)-f_p(y)| \dd y.
  \end{align*}
  By the Dominated Convergence Theorem, $\int_{y \in D} |f(y)-f_p(y)| \dd y\xrightarrow{p\rightarrow+\infty}0$ so that,
  \begin{equation*}
  \lim_{p \to +\infty}\limsup_{|\bfeta| \to 0} \sup_{x \in D} \int_{t_1}^{t_2} \Exp_x\left[|f-f_p|(\bar{Y}^\bfeta_{\un{t}})\ind{\bar{Y}^\bfeta_{\un{t}} \in D}\right] \dd t =0.
  \end{equation*}
 From what precedes, we get
  \begin{equation*}
    \lim_{p \to +\infty} \limsup_{|\bfeta| \to 0} \sup_{x \in D} \left|\bar{A}^\bfeta (f-f_p)(x)\right| = 0.
  \end{equation*}
  The term $|A(f-f_p)(x)|$ follows from the same arguments. This gives~\eqref{eq:weakA-1} and thus completes the proof.
\end{proof}

\begin{rk}[On the role of Assumption~\ref{it:steps-3}]\label{rk:H3c}
  Assumption~\ref{it:steps-3} is only used in the proof of~\eqref{eq:gauss-pHn}, which in turn is only used for the regularization argument in the proof of~Proposition~\ref{prop:weakA}, under a rather weak form since we actually only require that
  \begin{equation}\label{eq:bsup-barp}
    \forall 0 < t_1 < t_2 < +\infty, \qquad \limsup_{|\bfeta| \to 0} \sup_{x,y \in D}\sup_{t \in [t_1,t_2]} \bar{p}^\bfeta_{\un{t}}(x,y) < +\infty.
  \end{equation}
  Therefore, the following two remarks are in order.
  \begin{enumerate}
    \item If one is able to show that~\eqref{eq:bsup-barp} holds true without using Assumption~\ref{it:steps-3}, then Proposition~\ref{prop:weakA} and all subsequent results of the article remain in force. For instance, if $\sigma$ is constant, this property is true without \ref{it:steps-3} (with the help of Girsanov arguments). Similarly, Malliavin calculus may also lead to~\eqref{eq:bsup-barp} when the coefficients are smooth enough.
    \item In any case, Lemma~\ref{lem:weakerror} shows that the uniform convergence of $\bar{A}^\bfeta f$ to $Af$ holds true for any continuous and bounded function $f : D \to \R$, independently from the bound~\eqref{eq:bsup-barp}. Therefore, without Assumption~\ref{it:steps-3}, Proposition~\ref{prop:weakA} remains true for continuous and bounded functions, and so do our main results Theorems~\ref{theo:main-mun} and~\ref{theo:main}.
  \end{enumerate}
\end{rk}
\subsection{Comparison between $(\mubar_n)_{n \geq 1}$ and $(\vta_\ell)_{\ell \geq 1}$}\label{ss:mu-vartheta}

For any $\ell \geq 1$ and $f : D \to \R$ measurable and bounded, let us define
\begin{equation*}
  \Delta N_\ell(f) = \sum_{k=\mathfrak{n}_{\ell-1}}^{\mathfrak{n}_\ell-1} \gamma_{k+1} f(\bar{X}_{\Gamma_k})=\int_{s=\mathfrak{n}_{\ell-1}}^{\mathfrak{n}_{\ell}} f(\bar{X}_{\un{s}})\dd s.
\end{equation*}
Note that for all $s\ge \mathfrak{n}_{\ell-1}$, $\bar{X}_{\un{s}}=\bar{X}_{\un{s}_{\bfeta}}^{\bfeta}$ with $\bfeta=\boldsymbol{\gamma}^{(\ell-1)}$ where 
 $\boldsymbol{\gamma}^{(\ell)} = (\gamma^{(\ell)}_k)_{k \geq 1}$ is defined by 
\begin{equation}\label{eq:defgammaparl}
  \forall k \geq 1, \qquad \gamma^{(\ell)}_k := \gamma_{ \mathfrak{n}_\ell+ k}.
\end{equation}
 Thus, by the strong Markov property, we deduce that, for any measurable and bounded $f: D \to \R$,
\begin{align}
  \forall \ell \geq 1, \qquad \Exp[\Delta N_\ell(f)|\mathcal{F}_{\tauel_{\ell-1}}] &= \bar{A}^{\boldsymbol{\gamma}^{(\ell-1)}}f(\bar{X}_{\tauel_{\ell-1}}),\label{eq:ECtauel}\\
  \forall \ell \geq 2, \qquad \Exp[\Delta N_\ell(f)|\mathcal{F}_{\tauel_{\ell-1}^-}] &= \pn_{\mathfrak{n}_{\ell-1}}\bar{A}^{\boldsymbol{\gamma}^{(\ell-1)}}f,\label{eq:ECtauel-}
\end{align}
where $\bar{A}^{\bfeta}$ is defined by \eqref{def:Aeta} and where, in the second line, we used that $\pn_{\mathfrak{n}_{\ell-1}}$ is  the redistribution measure (related to the  $(\ell-1)$-th jump).
\begin{lem}[Comparison between $(\mubar_n)_{n \geq 1}$ and $(\vta_\ell)_{\ell \geq 1}$]\label{lem:compardiscconti}
  Let the assumptions of Theorem~\ref{theo:main} hold.
  \begin{enumerate}[label=(\roman*),ref=\roman*]
    \item\label{it:compardiscconti-1} For any bounded and measurable function $f : D \to \R$,
    \begin{equation*}
      \lim_{\ell \to +\infty} \frac{1}{\ell}\sum_{j=1}^{\ell} \left(\bar{A}^{\boldsymbol{\gamma}^{(j)}}f-Af\right)(\bar{X}_{\tauel_j}) = 0, \qquad \text{almost surely.}
    \end{equation*}
    \item\label{it:compardiscconti-2} For any bounded and measurable function $f : D \to \R$,
    \begin{equation*}
      \lim_{\ell \to +\infty} \frac{1}{\ell}\sum_{j=1}^{\ell}\Big(\Delta N_{j+1}(f) - \vta_\ell Af\Big) = 0, \qquad \text{almost surely.}
    \end{equation*}
    \item\label{it:compardiscconti-3} Almost surely,
    \begin{equation*}
      \liminf_{\ell\to+\infty}\vta_\ell A \iind{D}>0, \quad \liminf_{\ell\to+\infty}\frac{\tauel_{\ell}}{\ell}>0, \quad \liminf_{\ell\to+\infty}\frac{1}{\ell}\sum_{j=1}^\ell \ES[\Delta \tauel_{j+1}|{\cal F}_{\tauel_j}]>0,
    \end{equation*}
    where $\Delta \tauel_{j+1}=\tauel_{j+1}-\tauel_j=\Delta N_{j+1} (\iind{D})$.
    \item\label{it:compardiscconti-4} If $(\vta_\ell)_{\ell\ge1}$ is almost surely convergent to a probability $\nu\in{\cal M}_1(D)$ for the weak topology, then we have
    \begin{equation*}
      \lim_{\ell \to +\infty} \frac{\tauel_\ell}{\ell} = \nu A \iind{D}, \qquad \text{almost surely;}
    \end{equation*}
    and, for any bounded and measurable function $f : D \to \R$,
    \begin{equation*}
      \lim_{n \to +\infty} \mubar_n(f) = \Pi_\nu(f), \qquad \text{almost surely.}
    \end{equation*}
  \end{enumerate}
\end{lem}
\begin{proof} 
  Since $\mathfrak{n}_j \to +\infty$ when $j \to +\infty$, Assumption~\ref{it:steps-1} ensures that $|\boldsymbol{\gamma}^{(j)}| \to 0$. Therefore, by Proposition~\ref{prop:weakA} and the Ces\`aro Lemma, for any bounded and measurable function $f : D \to \R$ we have
  \begin{equation*}
    \left|\frac{1}{\ell}\sum_{j=1}^{\ell} \left(\bar{A}^{\boldsymbol{\gamma}^{(j)}}f-Af\right)(\bar{X}_{\tauel_j})\right| \leq \frac{1}{\ell}\sum_{j=1}^{\ell} \sup_{x \in D} |\bar{A}^{\boldsymbol{\gamma}^{(j)}}f(x)-Af(x)| \xrightarrow{\ell\to+\infty} 0,
  \end{equation*}
  which proves~\eqref{it:compardiscconti-1}. Now, since the paths of $\bar{X}_t$ are continuous on each interval $[\tauel_\ell, \tauel_{\ell+1})$, one can check that for any $\ell \ge 1$, $\tauel_\ell$ is ${\cal F}_{\tauel_\ell^-}$-measurable. It follows that for any measurable and bounded $f : D \to \R$, the sequence $(\Delta N_\ell(f))_{\ell\ge1}$ is $({\cal F}_{\tauel_{\ell}^{-}})_{\ell \geq 1}$-adapted and by~\eqref{eq:ECtauel-}, 
$$\ES[\Delta N_{\ell+1}(f)|{\cal F}_{\tauel_{\ell}^{-}}]=\pn_{\mathfrak{n}_\ell} (\bar{A}^{\boldsymbol{\gamma}^{(\ell)}} f).$$
But for any $\ell\ge 1$, we also have
$$\ES[\bar{A}^{\boldsymbol{\gamma}^{(\ell)}} f(\bar{X}_{\tauel_{\ell}})|{\cal F}_{\tauel_{\ell}^{-}}]=\pn_{\mathfrak{n}_\ell} (\bar{A}^{\boldsymbol{\gamma}^{(\ell)}} f),$$
so that the sequence $(\Delta N_{\ell+1}(f)-\bar{A}^{\boldsymbol{\gamma}^{(\ell)}} f(\bar{X}_{\tauel_{\ell}}))_{\ell \geq 1}$ is a martingale difference sequence for the filtration $({\cal F}_{\tauel_{\ell}^{-}})_{\ell \geq 1}$. Moreover, for any $\ell \geq 1$,
\begin{align*}
  \Exp\left[\left(\Delta N_{\ell+1}(f)-\bar{A}^{\boldsymbol{\gamma}^{(\ell)}} f(\bar{X}_{\tauel_{\ell}})\right)^2 | {\cal F}_{\tauel_{\ell}^{-}}\right] & \leq 2 \Exp\left[\left(\Delta N_{\ell+1}(f)\right)^2+\left(\bar{A}^{\boldsymbol{\gamma}^{(\ell)}} f(\bar{X}_{\tauel_{\ell}})\right)^2 | {\cal F}_{\tauel_{\ell}^{-}}\right]\\
  & \leq 2 \|f\|_\infty^2 \Exp\left[\left(\tauel_{\ell+1}-\tauel_\ell\right)^2+\left(\Exp_{\bar{X}_{\tauel_\ell}}[\bar{\tau}^\bfeta_D]_{|\bfeta=\boldsymbol{\gamma}^{(\ell)}}\right)^2 | {\cal F}_{\tauel_{\ell}^{-}}\right]\\
  & \leq 4 \|f\|_\infty^2 \sup_{x \in D, |\bfeta| \leq |\boldsymbol{\gamma}^{(\ell)}|} \Exp_x\left[(\bar{\tau}^\bfeta_D)^2\right].
\end{align*}
Since $|\boldsymbol{\gamma}^{(\ell)}|$ is bounded by $|\boldsymbol{\gamma}|$ uniformly in $\ell$, we deduce from Lemma~\ref{lem:controlmomentstempsarret} that 
\begin{equation*}
  \sup_{\ell \geq 1} \Exp\left[\left(\Delta N_{\ell+1}(f)-\bar{A}^{\boldsymbol{\gamma}^{(\ell)}} f(\bar{X}_{\tauel_{\ell}})\right)^2\right] < +\infty,
\end{equation*}
which by the strong Law of Large Numbers for martingale difference sequences leads to
\begin{equation*}
  \lim_{\ell \to +\infty} \frac{1}{\ell}\sum_{j=1}^\ell \left(\Delta N_{j+1}(f)-\bar{A}^{\boldsymbol{\gamma}^{(j)}} f(\bar{X}_{\tauel_j})\right) = 0, \qquad \text{almost surely.}
\end{equation*}
Combining this statement with~\eqref{it:compardiscconti-1} proves~\eqref{it:compardiscconti-2}.

By Lemma~\ref{lem:tightness-vartheta}, almost surely, there exists a compact set $K \subset D$ for which $\vta_\ell(K) \geq 1/2$. Hence, since $A\iind{D}$ is continuous and positive on $D$, for any $\ell \geq 1$,
\begin{equation*}
  \vta_\ell A \iind{D} \geq \frac{1}{2}\inf_{x \in K} A \iind{D}(x) > 0, 
\end{equation*}
which proves that almost surely, $\liminf_{\ell\to+\infty}\vta_\ell A \iind{D}>0$. By~\eqref{it:compardiscconti-2}, and since
\begin{equation*}
  \frac{1}{\ell}\sum_{j=1}^\ell \Delta N_{j+1}(\iind{D}) = \frac{\tauel_{\ell+1}-\tauel_1}{\ell},
\end{equation*}
we then deduce that almost surely, $\liminf_{\ell\to+\infty}\tauel_\ell/\ell>0$. Last, by~\eqref{eq:ECtauel},
\begin{equation*}
  \ES[\Delta \tauel_{j+1}|{\cal F}_{\tauel_j}] = \bar{A}^{\boldsymbol{\gamma}^{(j)}}\iind{D}(\bar{X}_{\tauel_j}),
\end{equation*}
so that by~\eqref{it:compardiscconti-1}, 
\begin{equation*}
  \lim_{\ell \to +\infty} \frac{1}{\ell}\sum_{j=1}^\ell (\ES[\Delta \tauel_{j+1}|{\cal F}_{\tauel_j}] - A \iind{D}(\bar{X}_{\tauel_j}))=0.
\end{equation*}
Since, on the other hand,
\begin{equation*}
  \liminf_{\ell \to +\infty} \frac{1}{\ell} \sum_{j=1}^\ell A \iind{D}(\bar{X}_{\tauel_j}) = \liminf_{\ell \to +\infty} \vta_\ell A \iind{D} > 0,
\end{equation*}
this completes the proof of~\eqref{it:compardiscconti-3}.

We finally assume that almost surely, $\vta_\ell$ converges weakly to some probability measure $\nu$ on $D$. Then, for any measurable and bounded function $f : D \to \R$, $Af$ is continuous and bounded and therefore $\vta_\ell Af$ converges to $\nu Af$, almost surely. By~\eqref{it:compardiscconti-2} we next deduce that
\begin{equation*}
  \lim_{\ell \to +\infty} \frac{1}{\ell}\sum_{j=1}^\ell \Delta N_j(f) = \nu A f, \qquad \text{almost surely.}
\end{equation*}
Applying this result to $f=\iind{D}$, we first get that $\tauel_\ell/\ell \to \nu A \iind{D}$, almost surely. Then, for any measurable and bounded function $f : D \to \R$, we deduce that
\begin{equation*}
  \lim_{\ell \to +\infty} \mubar_{\mathfrak{n}_\ell}(f) = \lim_{\ell \to +\infty} \frac{\ell}{\tauel_\ell} \frac{1}{\ell}\sum_{j=1}^\ell \Delta N_j(f) = \frac{\nu A f}{\nu A \iind{D}} = \Pi_\nu f, \qquad \text{almost surely.}
\end{equation*}
To complete the proof, we show that this convergence holds along the whole sequence $(\mubar_n)_{n \geq 1}$. To this aim, we write, for any $n \in \{\mathfrak{n}_\ell, \ldots, \mathfrak{n}_{\ell+1}-1\}$,
\begin{equation*}
  |\mubar_n(f) - \mubar_{\mathfrak{n}_\ell}(f)| \leq 2\|f\|_\infty \frac{\Gamma_n - \Gamma_{\mathfrak{n}_\ell}}{\Gamma_n} \leq 2\|f\|_\infty \frac{\tauel_{\ell+1}-\tauel_\ell}{\tauel_\ell},
\end{equation*}
and since $\tauel_\ell/\ell$ converges almost surely to the positive number $\nu A \iind{D}$ then the right-hand side vanishes. This shows that $\mubar_n$ converges to $\Pi_\nu f$ and completes the proof of~\eqref{it:compardiscconti-4}.
\end{proof}

\subsection{Convergence of $(\vta_\ell)_{\ell \geq 1}$}\label{ss:cv-vartheta} 

The purpose of this subsection is to establish the following statement.

\begin{prop}[Convergence of $(\vta_\ell)_{\ell \geq 1}$]\label{prop:cv-vartheta}
  Under the assumptions of Theorem~\ref{theo:main}, $(\vta_\ell)_{\ell \geq 1}$ converges almost surely to $\mu^\star$, weakly in $\mathcal{M}_1(D)$.
\end{prop}

The proof of Proposition~\ref{prop:cv-vartheta} is based on the idea that, by construction,
\begin{align}
  \vta_{\ell+1}&=\vta_\ell\left(1-\frac{1}{\ell+1}\right)+\frac{1}{\ell+1} \delta_{\bar{X}_{\tauel_{\ell+1}}}\nonumber\\
  &=\vta_\ell+\frac{1}{\ell+1} \left( \Pi_{\vta_\ell}-\vta_\ell\right)+ \frac{1}{\ell+1}\left( \delta_{\bar{X}_{\tauel_{\ell+1}}}-\Pi_{\vta_\ell}\right)\nonumber\\
  &=\vta_\ell+{h}_{\ell} F(\vta_\ell)+{h}_{\ell} \varepsilon_{\ell+1}\label{eq:recursivvta}
\end{align}
where
\begin{align}\label{def:F}
  h_\ell=\frac{1}{(\ell+1) (\vta_{\ell} A  \iind{D})},\quad F(\mu)=\mu A-(\mu A  \iind{D}) \mu,\quad \varepsilon_{\ell+1}= (\vta_{\ell} A  \iind{D})\delta_{\bar{X}_{\tauel_{\ell+1}}}-\vta_\ell A.
\end{align}
Written in this way, $(\vta_\ell)_{\ell\ge1}$ can be viewed as a discretization of the ordinary differential equation 
\begin{equation}\label{eq:ODE}
  \dot{\varphi}_t = F(\varphi_t),
\end{equation}
with a sequence of random steps $(h_\ell)_{\ell\ge1}$ and a perturbation sequence denoted by $(\varepsilon_\ell)_{\ell\ge 2}$. It is clear that $F(\mu^\star)=0$, and to prove that $\vta_\ell$ converges to $\mu^\star$, we shall show that $(\vta_\ell)_{\ell\ge1}$ is an \emph{asymptotic pseudo-trajectory} for~\eqref{eq:ODE} on the one hand, and that $\mu^\star$ is the unique, global attractor of this ODE on the other hand. To proceed, we first show that the noise component in~\eqref{eq:recursivvta} vanishes asymptotically. Before stating it, we first remark that, by~\eqref{def:F} and Lemma~\ref{lem:controlmomentstempsarret},
  \begin{equation}\label{eq:assumption3i}
    \liminf_{\ell\to+\infty} (\ell+1) h_\ell \ge \frac{1}{\sup_{x\in D}\ES_x[\tau_D]}>0, \qquad \text{almost surely,}
  \end{equation}
  and by Lemma~\ref{lem:compardiscconti}~\eqref{it:compardiscconti-3},
  \begin{equation}\label{eq:assumption3ii}
    \limsup_{\ell\to+\infty} (\ell+1) h_\ell \le \frac{1}{\liminf_{\ell\to+\infty} \vta_\ell A\iind{D}} < +\infty, \qquad \text{almost surely.}
  \end{equation}
For  $t>0$, we also set 
  \begin{equation}\label{def:Ndelt}
    N(t)=\inf\left\{N \ge 1, \sum_{j=1}^{N+1}h_j>t\right\}.
  \end{equation}
  Note that~\eqref{eq:assumption3i} and~\eqref{eq:assumption3ii} guarantee that $N(t)$ is almost surely finite, and that $N(t) \to +\infty$ when $t \to +\infty$.
  
  The main auxiliary result toward the proof of Proposition~\ref{prop:cv-vartheta} is the following statement.

\begin{lem}[Control of the noise component]\label{lem:control-error}
  Let the assumptions of Theorem~\ref{theo:main} hold. For any bounded and Lipschitz continuous function $f : D \to \R$, we have for any $s>0$,
  \begin{equation*}
    \lim_{t\to+\infty}\sum_{j=N(t)+1}^{N(t+s)} {h}_j \varepsilon_{j+1}(f)=0, \qquad \text{almost surely.}
  \end{equation*}
\end{lem}
\begin{proof}
  Let $f: D \to \R$ be bounded and Lipschitz continuous. The sequence
  \begin{equation*}
    \left(\sum_{j=1}^\ell \left({h}_j \varepsilon_{j+1}(f)-\ES\left[{h}_j \varepsilon_{j+1}(f)|{{\cal F}}_{\tauel_{j+1}^{-}}\right]\right)\right)_{\ell\ge1}
  \end{equation*}
  is a $({\cal F}_{\tauel_{\ell+1}})_{\ell \geq 1}$-martingale, whose infinite bracket is bounded by
  \begin{equation*}
    2\|f\|_\infty^2\sum_{j=1}^{+\infty} \frac{1}{(j+1)^2} <+\infty.
  \end{equation*}
  This implies that this martingale converges almost surely. Hence, since $N(t)\rightarrow+\infty$ as $t\rightarrow+\infty$, we are reduced to proving that 
  \begin{equation}\label{eq:criterionpseudo}
    \lim_{t\to+\infty}\sum_{j=N(t)+1}^{N(t+s)}  \ES\left[{h}_j \varepsilon_{j+1}(f)|{{\cal F}}_{\tauel_{j+1}^{-}}\right]=0, \qquad \text{almost surely.}
  \end{equation}
  One remarks that 
  \begin{equation*}
    \mathbb{E}[\varepsilon_{j+1}(f)|{{\cal F}}_{\tauel_{j+1}^{-}}]= \vta_j A \iind{D} \pn_{\mathfrak{n}_{j+1}}(f)- \vta_j A f = \vta_j A \iind{D} \left(\pn_{\mathfrak{n}_{j+1}}- \Pi_{\vta_j} \right)(f).
  \end{equation*}
  Using Assumption~\ref{cond:pn}, the definition of $N(t)$ and Lemma~\ref{lem:controlmomentstempsarret}, we get
  \begin{align*}
    &\sum_{j=N(t)+1}^{N(t+s)} h_j \vta_j A \iind{D} \left|(\pn_{\mathfrak{n}_{j+1}} - \mubar_{\mathfrak{n}_{j+1}}) (f)\right|\\
    &\le  \left(\frac{1}{N(t+s)}+s \sup_{x\in D}\ES_x[\tau_D]\right) \sup_{n\ge \mathfrak{n}_{N(t)+2}} \left|({\pn}_n-\mubar_n)(f)\right|\xrightarrow{t\to+\infty}0.
  \end{align*}
  In view of the previous convergence and of~\eqref{eq:criterionpseudo}, we are thus reduced to proving that
  \begin{equation}\label{eq:criterionpseudo2}
    \lim_{t\to+\infty} \left|\sum_{j=N(t)+1}^{N(t+s)} \frac{1}{j+1}\left(\mubar_{\mathfrak{n}_{j+1}} - \Pi_{\vta_j}\right)(f)\right|=0, \qquad \text{almost surely.}
  \end{equation}
  
  We remark that $\mubar_{\mathfrak{n}_{j+1}}(f)$ decomposes as follows:
  \begin{align*}
    \mubar_{\mathfrak{n}_{j+1}}(f)&= \left(\frac{1}{\tauel_{j+1}}-\frac{1}{\sum_{k=1}^j   \ES[\Delta \tauel_{k+1}|{\cal F}_{\tauel_k}]}\right)\tauel_{j+1} \mubar_{\mathfrak{n}_{j+1}}(f)\\
    &\quad +\frac{\tauel_j  \mubar_{\mathfrak{n}_j}(f)}{ \sum_{k=1}^j   \ES[\Delta \tauel_{k+1}|{\cal F}_{\tauel_k}]} +\frac{\Delta N_{j+1}(f)}{ \sum_{k=1}^j   \ES[\Delta \tauel_{k+1}|{\cal F}_{\tauel_k}]},
  \end{align*}
  where we recall from Subsection~\ref{ss:mu-vartheta} that $\Delta N_{j+1}(f) = \sum_{k=\mathfrak{n}_j}^{\mathfrak{n}_{j+1}-1}\gamma_{k+1}f(\bar{X}_{\Gamma_k})$. Hence, 
  \begin{equation*}
    \frac{1}{j+1}\left(\mubar_{\mathfrak{n}_{j+1}} - \Pi_{\vta_j}\right)(f) = \Delta R_{j,1}+\Delta R_{j,2}+\Delta R_{j,3}+\Delta R_{j,4},
  \end{equation*}
  where,
  \begin{align*}
    \Delta R_{j,1}&=\frac{1}{j+1}\left(\frac{1}{\tauel_{j+1}}-\frac{1}{\sum_{k=1}^j  \ES[\Delta \tauel_{k+1}|{\cal F}_{\tauel_k}]}\right)\tauel_{j+1} \mubar_{\mathfrak{n}_{j+1}}(f),\\
    \Delta R_{j,2}&=\frac{\tauel_{j}  \mubar_{\mathfrak{n}_j}(f)-\sum_{k=0}^{j-1} \bar{A}^{\boldsymbol{\gamma}^{(k)}} f(\bar{X}_{\tauel_k})}{(j+1) \sum_{k=1}^j   \ES[\Delta \tauel_{k+1}|{\cal F}_{\tauel_k}]},\\
    \Delta R_{j,3}&=\frac{1}{j+1}\left(\frac{\sum_{k=0}^{j-1} \bar{A}^{\boldsymbol{\gamma}^{(k)}} f(\bar{X}_{\tauel_k})}{ \sum_{k=1}^j   \ES[\Delta \tauel_{k+1}|{\cal F}_{\tauel_k}]}- \frac{\vta_j A f}{\vta_{j} A \iind{D}}\right),\\
    \Delta R_{j,4}&=\frac{\Delta N_{j+1}(f)}{ (j+1)\sum_{k=1}^j   \ES[\Delta \tauel_{k+1}|{\cal F}_{\tauel_k}]},
  \end{align*}
  and we recall that for $\ell \geq 1$, the sequence of time steps $\boldsymbol{\gamma}^{(\ell)}$ is defined in~\eqref{eq:defgammaparl}. In view of~\eqref{eq:criterionpseudo2}, it remains to prove that for $i=1,2,3,4$,
  \begin{equation}\label{eq:criterionpseudo2:1}
    \lim_{\ell\to+\infty} \left(\sum_{j=N(t)+1}^{N(t+s)} \Delta R_{j,i}\right)=0, \qquad \text{almost surely.}
  \end{equation}
  
  We first address the cases $i=1,2,4$. We introduce, for any $N\in\mathbb{N}$ and $c>0$, 
  \begin{equation*}
    \sigma_{N,c}:=\inf\left\{\ell \ge N, \frac{1}{\ell+1}\sum_{k=1}^{\ell+1} \ES[\Delta \tauel_{k+1}|{\cal F}_{\tauel_k}]\le c\right\},
  \end{equation*}
  with the convention that $\sigma_{N,c} = +\infty$  if  the set above is empty. We shall prove that, for $i=1,2,4$, for any $N \in \N$ and $c > 0$,
  \begin{equation*}
    \ES\left[ \sum_{j=N}^{\sigma_{N,c}} \left|\Delta R_{j,i}\right|\right] < +\infty,
  \end{equation*}
  which shows that $\sum_{j=N}^{\sigma_{N,c}} |\Delta R_{j,i}|$ is almost surely convergent. Since, by~Lemma~\ref{lem:compardiscconti}~\eqref{it:compardiscconti-3}, almost surely there exist $N$ and $c$ for which $\sigma_{N,c} = +\infty$, we deduce that $\sum_{j\ge1} |\Delta R_{j,i}|$ is almost surely convergent and in turn that \eqref{eq:criterionpseudo2:1} holds true. For $i=1$, we use the definition of $\sigma_{N,c}$ to write, for any $N$ and $c$,
  \begin{align*}
    \ES\left[ \sum_{j=N}^{\sigma_{N,c}} \left|\Delta R_{j,1}\right|\right]& \leq \ES\left[\sum_{j=N}^{\sigma_{N,c}}\frac{\|f\|_\infty}{j+1}\left| \frac{\tauel_{j+1}-\sum_{k=1}^j   \ES[\Delta \tauel_{k+1}|{\cal F}_{\tauel_k}]}{\sum_{k=1}^j   \ES[\Delta \tauel_{k+1}|{\cal F}_{\tauel_k}]}\right|\right]\\
    &\le \sum_{j=N}^{+\infty} \frac{\|f\|_\infty}{cj(j+1)}\ES\left[\left(\tauel_{j+1}-\sum_{k=1}^j   \ES[\Delta \tauel_{k+1}|{\cal F}_{\tauel_k}]\right)^2\right]^{\frac{1}{2}}\\
    &\le \sum_{j=N}^{+\infty} \frac{\|f\|_\infty}{cj(j+1)} \left(\sum_{k=0}^j \ES[(\Delta \tauel_{k+1})^2]\right)^{\frac{1}{2}}\\
    &\le \sum_{j=N}^{+\infty} \frac{\|f\|_\infty}{cj(j+1)}\sqrt{j+1}\sup_{x\in D,|\bfeta| \le |\boldsymbol{\gamma}|}\ES_x[(\bar{\tau}^\bfeta_{D})^2]^{\frac{1}{2}}<+\infty,
  \end{align*}
  by~Lemma~\ref{lem:controlmomentstempsarret}. For $i=2$, we remark that since, by~\eqref{eq:ECtauel}, $\Exp[\Delta N_{k+1}(f)|\mathcal{F}_{\tauel_k}] = \bar{A}^{\boldsymbol{\gamma}^{(k)}}f(\bar{X}_{\tauel_k})$ for $k \geq 0$, we have
  \begin{align*}
    \Exp\left[\left|\tauel_j \mubar_{\mathfrak{n}_j}(f) - \sum_{k=0}^{j-1} \bar{A}^{\boldsymbol{\gamma}^{(k)}}f(\bar{X}_{\tauel_k})\right|^2\right] &= \Exp\left[\left|\sum_{k=0}^{j-1} \left(\Delta N_{k+1}(f) -  \bar{A}^{\boldsymbol{\gamma}^{(k)}}f(\bar{X}_{\tauel_k})\right)\right|^2\right]\\
    &= \sum_{k=0}^{j-1}\Exp\left[\left(\Delta N_{k+1}(f) -  \bar{A}^{\boldsymbol{\gamma}^{(k)}}f(\bar{X}_{\tauel_k})\right)^2\right]\\
    &\leq \sum_{k=0}^{j-1}\Exp\left[\left(\Delta N_{k+1}(f)\right)^2\right]\\
    &\leq \|f\|_\infty^2 \sum_{k=0}^{j-1}\Exp\left[\left(\Delta \tauel_{k+1}\right)^2\right]\\
    &\leq j \|f\|_\infty^2 \sup_{x\in D,|\bfeta| \le |\boldsymbol{\gamma}|}\ES_x[(\bar{\tau}^\bfeta_{D})^2].
  \end{align*}
  Thanks to~Lemma~\ref{lem:controlmomentstempsarret} again, it follows that for any $N$ and $c$,
  \begin{equation*}
    \Exp\left[\sum_{j=N}^{\sigma_{N,c}} \left|\Delta R_{j,2}\right|\right] \leq \sum_{j=N}^{+\infty} \frac{\sqrt{j}}{cj(j+1)}\|f\|_\infty \sup_{x\in D,|\bfeta| \le |\boldsymbol{\gamma}|}\ES_x[(\bar{\tau}^\bfeta_{D})^2]^{\frac{1}{2}} < +\infty.
  \end{equation*}
  Last, for $i=4$, with similar arguments we have
  \begin{align*}
    \Exp\left[\sum_{j=N}^{\sigma_{N,c}} \left|\Delta R_{j,4}\right|\right] &\leq \sum_{j=N}^{+\infty} \frac{1}{cj(j+1)} \Exp\left[|\Delta N_{j+1}(f)|\right] \leq \sum_{j=N}^{+\infty} \frac{1}{cj(j+1)} \|f\|_\infty \sup_{x\in D,|\bfeta| \le |\boldsymbol{\gamma}|}\ES_x[\bar{\tau}^\bfeta_{D}] < +\infty.
  \end{align*}

  We now address the case $i=3$. We first write
  \begin{equation*}
    \Delta R_{j,3} = \frac{1}{j+1}\left(\mathrm{I}_j + \mathrm{II}_j + \mathrm{III}_j\right),
  \end{equation*}
  with
  \begin{equation*}
    \mathrm{I}_j = \frac{\sum_{k=0}^{j-1} (\bar{A}^{\boldsymbol{\gamma}^{(k)}}f-Af)(\bar{X}_{\tauel_k})}{\sum_{k=1}^j \Exp[\Delta \tauel_{k+1} | \mathcal{F}_{\tauel_k}]}, \qquad \mathrm{II}_j = \frac{Af(\bar{X}_{\tauel_j}) - Af(\bar{X}_0)}{\sum_{k=1}^j \Exp[\Delta \tauel_{k+1} | \mathcal{F}_{\tauel_k}]},    
  \end{equation*}
  and
  \begin{equation*}
    \mathrm{III}_j = \frac{\vta_j A f}{\vta_j A \iind{D}} \frac{\sum_{k=1}^j (A \iind{D}(\bar{X}_{\tauel_k}) - \Exp[\Delta \tauel_{k+1} | \mathcal{F}_{\tauel_k}])}{\sum_{k=1}^j \Exp[\Delta \tauel_{k+1} | \mathcal{F}_{\tauel_k}]} = \Pi_{\vta_\ell}f \frac{\sum_{k=1}^j (A \iind{D} - \bar{A}^{\boldsymbol{\gamma}^{(k)}} \iind{D})(\bar{X}_{\tauel_k})}{\sum_{k=1}^j \Exp[\Delta \tauel_{k+1} | \mathcal{F}_{\tauel_k}]}.
  \end{equation*}
  Using~Lemma~\ref{lem:compardiscconti}~\eqref{it:compardiscconti-3} and~Proposition~\ref{prop:weakA}, we get
  \begin{equation*}
    \lim_{j \to +\infty} |\mathrm{I}_j| + |\mathrm{II}_j| + |\mathrm{III}_j| = 0, \qquad \text{almost surely.}
  \end{equation*}
  Therefore, almost surely, for any $\varepsilon > 0$, there exists $\ell_0(\epsilon) \geq 1$ such that $|\Delta R_{j,3}| \leq \varepsilon/(j+1)$ for any $j \geq \ell_0(\varepsilon)$. As a consequence, if $\ell \geq \ell_0(\epsilon)$ then
  \begin{equation*}
    \sum_{j=N(t)+1}^{N(t+s)} \left|\Delta R_{j,3}\right| \leq \varepsilon \sum_{j=N(t)+1}^{N(t+s)} \frac{1}{j+1} = \varepsilon \sum_{j=N(t)+1}^{N(t+s)} h_j \vta_j A \iind{D} \leq \varepsilon \left(1+{s} \sup_{x \in D} \Exp_x[\tau_D]\right),
  \end{equation*}
  which by~Lemma~\ref{lem:controlmomentstempsarret} shows that \eqref{eq:criterionpseudo2:1} holds true for $i=3$ and completes the proof.
\end{proof}

The completion of the proof of Proposition~\ref{prop:cv-vartheta} now follows from rather standard arguments, which are detailed in Section~\ref{proof:propcvvartheta}.

\subsection{Completion of the proof of Theorem~\ref{theo:main}}\label{ss:pf-main} We now conclude the proof by summarizing the arguments.

By~Proposition~\ref{prop:cv-vartheta}, $\vta_\ell$ converges weakly to $\mu^\star$, almost surely. Therefore, Lemma~\ref{lem:compardiscconti}~\eqref{it:compardiscconti-4} yields that for any measurable and bounded function $f: D \to \R$, $\mu_n(f) \to \Pi_{\mu^\star}(f)$, and it finally immediately follows from the definition of QSDs that $\Pi_{\mu^\star}(f)=\mu^\star(f)$. This completes the proof of the first statement in Theorem~\ref{theo:main}.

Now, to prove the almost sure convergence to $\lambda^\star$
 of ${\Gamma_n}^{-1}\sum_{k=1}^n \theta_k$,  we again use the  combination of Proposition~\ref{prop:cv-vartheta} and Lemma~\ref{lem:compardiscconti}~\eqref{it:compardiscconti-4} to obtain that almost surely,
 $$ \frac{\dten_{\mathfrak{n}_\ell}}{\ell}\xrightarrow{\ell \rightarrow+\infty} \mu^\star A \iind{D}=\ES_{\mu^{\star}}[\tau_D]=\frac{1}{\lambda^\star},$$
 since $\tau_D$ has exponential distribution with parameter $\lambda^\star$ when starting from $\mu^\star$ (see \emph{e.g.} \cite[Theorem 2.2]{CMM13}). To deduce the result, it then suffices to remark that for any $n\in \llbracket \mathfrak{n}_\ell , \mathfrak{n}_{\ell+1}-1\rrbracket$,
 $$ \frac{\ell}{\dten_{\mathfrak{n}_\ell+1}}\le \frac{1}{\Gamma_n}\sum_{k=1}^n \theta_k\le \frac{\ell}{\dten_{\mathfrak{n}_\ell}}$$
and that $\dten_{\mathfrak{n}_\ell}\rightarrow+\infty$ as $\ell \to +\infty$.

\section{Proof of Theorems~\ref{theo2} and~\ref{theo:weakerrorf}}\label{sec:converg-distrib}
In this section, we focus on the proofs of Theorem~\ref{theo2}, that is to say of the convergence in distribution of $(\bar{X}_{\Gamma_n})_{n\ge0}$ towards $\mu^\star$, and Theorem~\ref{theo:weakerrorf} which establishes the weak convergence toward $\mu$-return processes, over finite time horizons, of Euler schemes with redistribution. The main tool for both results is Proposition~\ref{propdistrib} which is a slight generalization of Theorem~\ref{theo:weakerrorf}. It is stated and proved in~Subsection~\ref{ss:weakerror}. The proof of Theorem~\ref{theo2} is then completed in Subsection~\ref{sec:prooftheo2}.

\subsection{Weak convergence of Euler schemes with redistribution}\label{ss:weakerror}

We first gather a preliminary estimate related to Euler schemes with redistribution as introduced in~Definition~\ref{defi:eulerdistrib}. For such a scheme $(\bar{X}^{\bfeta,\bfnu}_t)_{t \geq 0}$, and with the notation of~Definition~\ref{defi:eulerdistrib}, we denote by $(\bar\tauel_k)_{k\ge1}$ the associated jump sequence, that is to say the set of times $\dten_n$ for which $U^{\bfeta,\bfnu}_n = 1$. We also introduce the notation
\begin{equation}\label{eq:defNt}
  N(t) = \inf\left\{N \geq 1: \sum_{j=1}^{N+1} \eta_j > t\right\}, 
\end{equation}
so that $n \leq N(t)$ if and only if $\dten_n \leq t$. We last recall the definition~\eqref{def:wass} of the Wasserstein distance $\mathcal{W}_1$.

\begin{lem}[Estimate on Euler schemes with redistribution]\label{lem:cvdistrib} 
  Assume \ref{cond:D} and \ref{cond:coeffs}. For any $T>0$, $\varepsilon > 0$ and $\mu \in \mathcal{M}_1(D)$, there exist $\rho > 0$, $\delta > 0$ and $\eta_0 > 0$ such that, for any Euler scheme with redistribution $(\bar{X}^{\bfeta,\bfnu}_t)_{t \geq 0}$, if
  \begin{equation}\label{eq:cond-cvdistrib}
    |\bfeta| \leq \eta_0, \qquad \mathcal{W}_1\left(\mathcal{L}(\bar{X}^{\bfeta,\bfnu}_0), \mu\right) \leq \rho, \qquad {\Pr\left(\sup_{n \geq 1} \mathcal{W}_1(\nu_n,\mu) \leq \rho\right) = 1},
  \end{equation}
  then 
  \begin{equation*}
    \sup_{t\in[0,T]}\PE\left(\exists k\ge 1, \bar\tauel_k\in[t-\delta, t+\delta]\right) \le \varepsilon.
  \end{equation*}
\end{lem}
\begin{proof} 
  Let us first fix $\mu \in \mathcal{M}_1(D)$ and $\varepsilon > 0$. As a preliminary step, we claim that, defining the compact set $K_r=\{x\in D, d(x,\partial D)\ge r\}$, one can choose $r > 0$ and $\rho > 0$ such that for any $\alpha \in \mathcal{M}_1(D)$, if $\mathcal{W}_1(\alpha,\mu) \leq \rho$ then $\alpha(K_r^c) \leq \varepsilon/3$. The proof of this claim is elementary and we omit it. From now on we set $K=K_r$.
  
  We now fix $T>0$. Let $(\bar{X}^{\bfeta,\bfnu}_t)_{t \geq 0}$ be an Euler scheme with redistribution which satisfies the second and third conditions of~\eqref{eq:cond-cvdistrib} with $\rho$ fixed in the preliminary step. We first write, for any $\delta > 0$ and $t \in [0,T]$, 
  \begin{equation}\label{eq:cvdistrib:0}
    \PE\left(\exists k\ge 1, \bar\tauel_k\in[t-\delta, t+\delta]\right) \leq \sum_{k \geq 1} \PE\left(\bar\tauel_k\in[t-\delta, t+\delta]\right).
  \end{equation}
  For any $k \geq 1$, it follows from~Definition~\ref{defi:eulerdistrib} that the random variable $\bar\tauel_{k-1}$ is $\mathcal{G}_{(\bar\tauel_{k-1})^-}$-measurable, and conditionally on $\mathcal{G}_{(\bar\tauel_{k-1})^-}$, $\bar\tauel_k-\bar\tauel_{k-1}$ has the law of the exit time from $D$ for an Euler scheme (without redistribution) with initial distribution $\nu_{N(\bar\tauel_{k-1})}$ and step sequence $\bfeta^{(\bar\tauel_{k-1})} = (\eta_{N(\bar\tauel_{k-1})+n})_{n \geq 1}$, where $N(\cdot)$ is defined in~\eqref{eq:defNt}. Note that for $k=1$ we take the convention that $\mathcal{G}_{0^-}=\{\emptyset,\Omega\}$, $\bar\tauel_0=0$ and $\nu_0 = \mathcal{L}(\bar{X}^{\bfeta,\bfnu}_0)$. By the second and third conditions in~\eqref{eq:cond-cvdistrib}, we therefore deduce that, taking the conditional expectation with respect to $\mathcal{G}_{\dten_k^-}$,
  \begin{equation}\label{eq:cvdistrib:1}
    \PE\left(\bar\tauel_k\in[t-\delta, t+\delta]\right) \leq \ES\left[\Phi_\rho(t-\bar\tauel_{k-1},\delta,\bfeta^{(\bar\tauel_{k-1})})\ind{\bar\tauel_{k-1}<t+\delta}\right],
  \end{equation}
  where, for $s\ge-\delta$,
  \begin{equation}\label{eq:strategyphi}
    \Phi_\rho(s, \delta, \bfeta)=\sup_{\alpha, {\cal W}_1(\alpha,\mu)\le \rho}\PE_{\alpha}(\bar{\tau}^\bfeta_{D}\in[s-\delta,s+\delta]),
  \end{equation}
{with the notation of~\eqref{eq:taubareta} for $\bar{\tau}^\bfeta_{D}$}. For the compact subset $K$ fixed in the preliminary step,
  \begin{equation*}
    \Phi_\rho(s, \delta, \bfeta)\le \sup_{x\in K}\Pr_x(|\tau_D-\bar{\tau}^\bfeta_D|\ge\delta)+\sup_{x\in K} \PE_x(\tau_D\in[s-2\delta,s+2\delta])+\frac{\varepsilon}{3}.
  \end{equation*}
  By~Lemma~\ref{lem:cvdistrib:1}, we may fix $\delta \in [0,1)$ such that
  \begin{equation*}
    \sup_{x\in K} \PE_x(\tau_D\in[s-2\delta,s+2\delta]) \leq \frac{\varepsilon}{3},
  \end{equation*}
  uniformly in $s \leq T$. By~Lemma~\ref{lem:strong-error-tau} and the Markov inequality, for this choice of $\delta$ one may fix $\eta_0 >0$ such that if $|\bfeta| \leq \eta_0$ then 
  \begin{equation*}
    \sup_{x\in K}\Pr_x(|\tau_D-\bar{\tau}^\bfeta_D|\ge\delta) \leq \frac{\varepsilon}{3}.
  \end{equation*}
  As an intermediate conclusion, we have constructed $\delta$, $\rho$ and $\eta_0$, which depend on $T$, $\mu$ and $\varepsilon$, such that if $|\bfeta| \leq \eta_0$ then
  \begin{equation}\label{eq:boundphistra}
    \forall s \in [0,T], \qquad \Phi_\rho(s, \delta, \bfeta) \leq \varepsilon.
  \end{equation}
  As a consequence, by~\eqref{eq:cvdistrib:1}, as soon as the scheme $(\bar{X}^{\bfeta,\bfnu}_t)_{t \geq 0}$ satisfies~\eqref{eq:cond-cvdistrib}, it holds
  \begin{equation}\label{eq:cvdistrib:2}
    \forall t \in [0,T], \qquad \PE(\bar{\tauel}_k\in[t-\delta, t+\delta])\le \varepsilon\PE(\bar\tauel_{k-1}<t+\delta) \le \varepsilon\PE(\bar\tauel_{k-1}<T+1).
  \end{equation}
   
  Now let $\delta'>0$ and $\xi_{k}=\iind{\Delta \bar\tauel_k\ge {\delta'}}$, where $\Delta \bar\tauel_k = \bar\tauel_k-\bar\tauel_{k-1}$ for $k \geq 1$. For every $k\ge 1$, we have
  \begin{equation*}
    \PE(\bar\tauel_{k}<T)\le \sum_{(i_1,\ldots,i_k)\in\{0,1\}^k, i_1+\ldots+i_k\le \lfloor \frac{T}{{\delta'}}\rfloor }\PE\left(\bigcap_{\ell=1}^k \left\{\xi_{\ell}=i_\ell\right\}\right).
  \end{equation*}
  But for any $\ell\ge 1$,
  \begin{equation*}
    \PE(\xi_{\ell}=0|\mathcal{G}_{\bar\tauel_{\ell-1}})\le \tilde{\Phi}(\nu_{\bar\tauel_{\ell-1}}, {\delta'},\bfeta^{(\bar\tauel_{\ell-1})})\le \sup_{\alpha, {\cal W}_1(\alpha,\mu)\le \rho}\tilde{\Phi}(\alpha, {\delta'},\bfeta^{(\bar\tauel_{\ell-1})}),
  \end{equation*}
  with $\tilde{\Phi}(\alpha, {\delta'},{\bfeta})=\PE_{\alpha}(\bar{\tau}^\bfeta_D \le {\delta'})$. We remark that, if we set $\delta=\delta'$ in $\Phi_\rho$ defined in \eqref{eq:strategyphi}, we have
  \begin{equation*}
    \sup_{\alpha, {\cal W}_1(\alpha,\mu)\le \rho}\tilde{\Phi}(\alpha, {\delta'},{\bfeta})\le \Phi_\rho(0,\delta',\bfeta).
  \end{equation*}
  As a consequence, applying~\eqref{eq:boundphistra} with $\varepsilon = 1/2$, we get that there exist $\delta'$, $\rho'$ and $\eta_0'$ such that 
  \begin{equation*}
    \sup_{(\alpha,\bfeta),{\cal W}_1(\alpha,\mu)\le \rho, |\bfeta|\le\eta_0} \tilde{\Phi}(\alpha, {\delta'},{\bfeta})\le \frac{1}{2}.
  \end{equation*}
  Up to replacing $\rho$ and $\eta_0$ by $\rho \wedge \rho'$ and $\eta_0 \wedge \eta_0'$, we deduce that if $(\bar{X}^{\bfeta,\bfnu}_t)_{t \geq 0}$ satisfies~\eqref{eq:cond-cvdistrib}, then
  \begin{equation*}
    \PE(\xi_{\ell}=i_\ell|\mathcal{G}_{\bar\tauel_{\ell-1}})\le \begin{cases} 
      1 &\textnormal{if $i_\ell=1$,}\\
      \frac{1}{2}&\textnormal{if $i_\ell=0$.}
    \end{cases}
  \end{equation*}
  As a consequence,
  \begin{align*}
    \sum_{k\ge1}\PE(\bar\tauel_{k-1}<T)&\le 1+ \sum_{k\ge1}\sum_{\ell=0}^{\lfloor \frac{T}{{\delta'}}\rfloor} \binom{k}{\ell} \frac{1}{2^{k-\ell}} \le 1+ \sum_{k\ge1} {2^{-k}} k^{\lfloor \frac{T}{{\delta'}}\rfloor}  \sum_{\ell=0}^{\lfloor \frac{T}{{\delta'}}\rfloor}\frac{2^\ell}{\ell !}\le 1+ \ee^2\sum_{k\ge1} {2^{-k}} k^{\lfloor \frac{T}{{\delta'}}\rfloor},
  \end{align*}
  and the right-hand side is bounded by some constant $C$ which only depends on $T$ and $\delta'$, but not on $\varepsilon$. Combining the latter estimate with~\eqref{eq:cvdistrib:0} and~\eqref{eq:cvdistrib:2}, we deduce that
  \begin{equation*}
    \sup_{t \in [0,T]} \PE\left(\exists k\ge 1, \bar\tauel_k\in[t-\delta, t+\delta]\right) \leq C \varepsilon,
  \end{equation*}
  which up to renormalizing $\varepsilon$ concludes the proof.
\end{proof}

We are now able to state the main result of this section. We recall that, for any $\mu \in \mathcal{M}_1(D)$, we denote by $(P_t^\mu)_{t\ge0}$ the semigroup of the $\mu$-return process $(X^\mu_t)_{t \geq 0}$ in $D$ defined in~Subsection~\ref{ss:cvdistrib}. If $X^\mu_0 \sim \mu_0$, then a basic property of this process, which is crucial in the proof of~Proposition~\ref{propdistrib} below, is its \emph{renewal} structure encoded in the identity
\begin{equation}\label{eq:renewal}
  \forall t \geq 0, \qquad \mu_0 P^\mu_t f = \Exp_{\mu_0}[f(X^\mu_t)] = \Exp_{\mu_0}\left[f(Y_t)\ind{t < \tau_D}\right] + \Exp_{\mu_0}\left[(\mu P_{t-\tau_D}^\mu f) \ind{t \geq \tau_D}\right],
\end{equation}
with $(Y_t)_{t \geq 0}$ the solution to~\eqref{eq:SDE} with initial distribution $\mu_0$ and $\tau_D$ the associated exit time from $D$.

The Euler scheme with redistribution $(\bar{X}^{\bfeta,\bfnu}_t)_{t \geq 0}$ enjoys a similar property, but its formulation is slightly more involved. First, we set $\nu_0 = \mathcal{L}(\bar{X}^{\bfeta,\bfnu}_0)$ and introduce the notation
\begin{equation}\label{eq:sg-redistrib}
  \forall t \geq 0, \qquad \Exp[f(\bar{X}^{\bfeta,\bfnu}_t)] = \nu_0 \bar{P}^{\bfeta,\bfnu}_t f,
\end{equation}
for the time-marginal measure of this scheme. The key point is now that, with the notation introduced before the statement of~Lemma~\ref{lem:cvdistrib} and defining 
\begin{equation*}
  \mathcal{G}'_t := \mathcal{G}_{\bar{\tauel}_1+t}, \quad \bfnu' = (\nu_{N(\bar{\tauel}_1)+n})_{n \geq 1}, \quad \bfeta' = (\eta_{N(\bar{\tauel}_1)+n})_{n \geq 1}, \quad \bar{X}_t'^{\bfeta',\bfnu'} := \bar{X}^{\bfeta,\bfnu}_{\bar{\tauel}_1+t},
\end{equation*}
we have that $\Pr$-almost surely, the triple $((\mathcal{G}'_t)_{t \geq 0},\bfnu', (\bar{X}_t'^{\bfeta',\bfnu'})_{t \geq 0})$ is an Euler scheme with redistribution, with step sequence $\bfeta'$ and initial distribution $\nu'_0 := \nu_{N(\bar{\tauel}_1)}$, defined on the space $(\Omega, \mathcal{F})$ equipped with the (random) probability measure $\Pr'(\cdot) = \Pr(\cdot|\mathcal{G}_{(\bar{\tauel}_1)^-})$. The associated time-marginal measure
\begin{equation*}
  \nu'_0 \bar{P}_t'^{\bfeta',\bfnu'} f := \Exp'\left[f\left(\bar{X}_t'^{\bfeta',\bfnu'}\right)\right] = \Exp\left[f\left(\bar{X}_{\bar{\tauel}_1+t}^{\bfeta,\bfnu}\right)|\mathcal{G}_{(\bar{\tauel}_1)^-}\right]
\end{equation*}
is then a $\mathcal{G}_{(\bar{\tauel}_1)^-}$-measurable random variable, and the discretized version of the renewal identity~\eqref{eq:renewal} writes
\begin{equation}\label{eq:renewal-discr}
  \forall t \geq 0, \qquad \nu_0 \bar{P}_t^{\bfeta,\bfnu} f = \Exp[f(\bar{X}^{\bfeta,\bfnu}_t)] = \Exp\left[f(\bar{Y}^\bfeta_t)\ind{t < \bar{\tau}^\bfeta_D}\right] + \Exp\left[(\nu'_0 \bar{P}_{t-\bar{\tauel}_1}'^{\bfeta',\bfnu'}f) \ind{t \geq \bar{\tauel}_1}\right],
\end{equation}
with $(\bar{Y}^\bfeta_t)_{t \geq 0}$ the Euler scheme associated with~\eqref{eq:SDE}, with step sequence $\bfeta$ and initial distribution $\nu_0$, and $\bar{\tau}^\bfeta_D$ {is defined by~\eqref{eq:taubareta}}.

\begin{prop}[Wasserstein error estimate for Euler schemes with redistribution]\label{propdistrib} 
  Let $\mu\in{\cal M}_1(D)$, $T>0$, $\varepsilon > 0$ and a compact subset $K \subset D$. There exist $\rho>0$, $\eta_0>0$ and $C \geq 0$, such that, for any Euler scheme with redistribution $(\bar{X}^{\bfeta,\bfnu}_t)_{t \geq 0}$ which satisfies
  \begin{equation}\label{eq:conditionnunnuu}
    |\bfeta| \leq \eta_0, \qquad \Pr\left(\sup_{n\ge1} {\cal W}_1({\nu}_n, \mu)> \rho\right)\le {\varepsilon},
  \end{equation}
  and for any $\mu_0 \in \mathcal{M}_1(D)$, we have
  \begin{equation*}
    \sup_{t \in [0,T]} \mathcal{W}_1\left(\nu_0 \bar{P}^{\bfeta,\bfnu}_t, \mu_0 P^\mu_t\right) \leq C_M \left(\varepsilon + \mu_0(K^c)\right) + C \mathcal{W}_1(\mu_0,\nu_0),
  \end{equation*}
  where $C_M$ is a constant which only depends on the set $D$.
\end{prop}

Theorem~\ref{theo:weakerrorf} is a direct consequence of~Proposition~\ref{propdistrib}.

\begin{rk}\label{rk:weakerrorrf} Proposition~\ref{propdistrib} is slightly more uniform than Theorem~\ref{theo:weakerrorf} since it applies uniformly to any Euler scheme with redistribution satisfying the condition \eqref{eq:conditionnunnuu}. Note that taking $\nu_n = \mu$ and $\bar{X}^{\bfeta,\bfnu}_0=X_0$ yields the natural Euler discretization of the $\mu$-return process (denote it by $\bar{X}^{\bfeta,\mu}_t$), for which our result yields the weak consistency estimate $$\sup_{t \in [0,T]} \mathcal{W}_1(\mathcal{L}(\bar{X}^{\bfeta,\mu}_t),\mathcal{L}(X^\mu_t)) \xrightarrow{|\bfeta| \to 0}0.$$
\end{rk}

\begin{rk} It seems that the methods used in this proof may be of interest for the $\mu$-return process itself.  For instance, the reader may check that a very simple adaptation of our renewal argument leads to the  following local stability estimate: for any compact subset $K$ of $D$,
$$\sup_{t\in[0,T]} \sup_{x,y\in K,|x-y|\le \delta}{\cal W}_1(\delta_x P_t^\mu,\delta_y {P}^{\mu}_t)\xrightarrow{\delta\rightarrow0}0.$$
\end{rk}

\begin{proof}[Proof of Proposition~\ref{propdistrib}] 
  Let us fix $\mu \in \mathcal{M}_1(D)$ and a function $f : D \to \R$ such that 
  \begin{equation}\label{eq:defM}
    \|f\|_\infty \leq M := \sup\{|x-y|, x,y \in D\}, \qquad [f]_1=1.
  \end{equation}
  The combination of~\eqref{eq:renewal} with~\eqref{eq:renewal-discr} yields, for any $\mu_0 \in \mathcal{M}_1(D)$ and for any Euler scheme with redistribution $(\bar{X}^{\bfeta,\bfnu}_t)_{t \geq 0}$,
  \begin{equation*}
    \forall t \geq 0, \qquad \left|\mu_0 P_t^\mu f - \nu_0 \bar{P}^{\bfeta,\bfnu}_t f\right| \leq \mathrm{I} + \mathrm{II},
  \end{equation*}
  with
  \begin{align*}
    \mathrm{I} &:= \left|\Exp_{\mu_0}\left[f(Y_t)\ind{t < \tau_D}\right] - \Exp_{\nu_0}\left[f(\bar{Y}^\bfeta_t)\ind{t < \bar{\tau}^\bfeta_D}\right]\right|,\\
    \mathrm{II} &:= \left|\Exp_{\mu_0}\left[(\mu P_{t-\tau_D}^\mu f) \ind{t \geq \tau_D}\right] - \Exp\left[(\nu'_0 \bar{P}_{t-\bar{\tauel}_1}'^{\bfeta',\bfnu'}f) \ind{t \geq \bar{\tauel}_1}\right]\right|.
  \end{align*}
  The proof is organized in 4 steps. In Step~1, we control the term $\mathrm{I}$ while in Step~2, we show that under the condition
  \begin{equation}\label{eq:cond-cvdistrib:2}
    |\bfeta| \leq \eta_0, \qquad \Pr\left(\sup_{n \geq 1} \mathcal{W}_1(\nu_n,\mu) \leq \rho\right) = 1,
  \end{equation}
  which is stronger than~\eqref{eq:conditionnunnuu}, the term $\mathrm{II}$ can be controlled so that $|\mu_0 P_t^\mu f - \nu_0 \bar{P}^{\bfeta,\bfnu}_t f|$ satisfies a renewal inequation. This inequation is solved in Step~3, and we detail how to relax the condition~\eqref{eq:cond-cvdistrib:2} to~\eqref{eq:conditionnunnuu} in Step~4.
  
  \medskip  
  \emph{Step~1: control of $\mathrm{I}$.} Assume that $(Y_t)_{t \geq 0}$ and $(\bar{Y}^\bfeta_t)_{t \geq 0}$ are defined on the same probability space, driven by the same Brownian motion $(B_t)_{t \geq 0}$, and with initial condition $(Y_0,\bar{Y}^\bfeta_0)$ distributed according to an optimal coupling $\Pi^*$ of $(\mu_0,\nu_0)$ (see Subsection~\ref{ss:wass} for the existence of $\Pi^*$). Then, for any $t \geq 0$,
  \begin{equation*}
    \mathrm{I} \leq \left|\Exp_{\Pi^*}\left[\left(f(Y_t)-f(\bar{Y}^\bfeta_t)\right)\ind{t < \tau_D, t < \bar{\tau}^\bfeta_D}\right]\right| + M \left(\Pr_{\Pi^*}\left(\tau_D \leq t \leq \bar{\tau}^\bfeta_D\right) + \Pr_{\Pi^*}\left(\bar{\tau}^\bfeta_D \leq t \leq \tau_D\right)\right).
  \end{equation*}
  By~Lemma~\ref{lem:strong-pages} and~Lemma~\ref{lem:strong-coupling}, and since $[f]_1 = 1$, for any $T>0$ there exists $C_T \geq 0$ such that, for any $t \in [0,T]$,
  \begin{align*}
    \left|\Exp_{\Pi^*}\left[\left(f(Y_t)-f(\bar{Y}^\bfeta_t)\right)\ind{t < \tau_D, t < \bar{\tau}^\bfeta_D}\right]\right| &\le \Exp_{\Pi^*}\left[\left|Y_t-\bar{Y}^\bfeta_t\right|\right]\\
    &\leq \int_{D \times D} \Exp\left[\left|Y^x_t-Y^y_t\right|\right]\Pi^*(\dd x,\dd y) + \sup_{y \in D} \Exp_y\left[\left|Y_t-\bar{Y}^\bfeta_t\right|\right]\\
    &\leq C_T\left(\mathcal{W}_1(\mu_0,\nu_0) + \sqrt{|\bfeta|}\right).
  \end{align*}
  On the other hand, by~Lemma~\ref{lem:wass-estim}~\eqref{it:wass-estim:2}, for any $T>0$, $\varepsilon > 0$ and compact subset $K \subset D$, there exist $C \geq 0$ and $\eta_0>0$ such that if $|\bfeta| \leq \eta_0$ then for any $t \in [0,T]$,
  \begin{equation*}
    \Pr_{\Pi^*}\left(\tau_D \leq t \leq \bar{\tau}^\bfeta_D\right) + \Pr_{\Pi^*}\left(\bar{\tau}^\bfeta_D \leq t \leq \tau_D\right) \leq \varepsilon+C {\cal W}_1(\mu_0,\nu_0)+\mu_0(K^c).
  \end{equation*}
  As a conclusion, for any $T>0$, $\varepsilon > 0$ and compact subset $K \subset D$, there exist $C^{(1)}_M \geq 0$, $C^{(1)} \geq 0$ and $\eta^{(1)}_0 > 0$ such that if $|\bfeta| \leq \eta^{(1)}_0$ then for any $t \in [0,T]$,
  \begin{equation*}
    \mathrm{I} \leq C^{(1)}_M\left(\varepsilon + \mu_0(K^c)\right) + C^{(1)} \mathcal{W}_1(\mu_0,\nu_0).
  \end{equation*}
  Notice that the constant $C^{(1)}_M$ only depends on $M$ and not on $T, \varepsilon, K$.
  
  \medskip
  \emph{Step~2: control of $\mathrm{II}$.} Taking again an optimal coupling $\Pi^*$ of $(\mu_0,\nu_0)$, we now get, for any $t \geq 0$ and  $\delta > 0$,
  \begin{align*}
    \mathrm{II} &\leq \left|\Exp_{\Pi^*}\left[\left(\mu P_{t-\tau_D}^\mu f - \nu'_0 \bar{P}_{t-\bar{\tauel}_1}'^{\bfeta',\bfnu'}f\right)\ind{\tau_D \leq t, |\tau-\bar{\tauel}_1| \leq \delta}\right]\right|\\
    &\quad + M\left(\Pr_{\Pi^*}\left(\tau_D \leq t \leq \bar{\tau}^\bfeta_D\right) + \Pr_{\Pi^*}\left(\bar{\tau}^\bfeta_D \leq t \leq \tau_D\right) + 2 \Pr_{\Pi^*}(|\tau_D-\bar{\tau}^\bfeta_D| \geq \delta)\right),
  \end{align*}
  and the first term satisfies the estimate
  \begin{equation}\label{eq:pfpropdistrib:21}
    \begin{aligned}
      &\left|\Exp_{\Pi^*}\left[\left(\mu P_{t-\tau_D}^\mu f - \nu'_0 \bar{P}_{t-\bar{\tauel}_1}'^{\bfeta',\bfnu'}f\right)\ind{\tau_D \leq t, |\tau_D-\bar{\tauel}_1| \leq \delta}\right]\right|\\
      &\leq \Exp_{\Pi^*}\left[\left|\mu P_{t-\tau_D}^\mu f - \nu'_0 \bar{P}_{t-\tau_D}'^{\bfeta',\bfnu'}f\right|\ind{\tau_D \leq t - \delta}\right]\\
      & + \Exp_{\Pi^*}\left[\sup_{u \in [-\delta,\delta]} \left|\nu'_0 \bar{P}_{t-\tau_D}'^{\bfeta',\bfnu'}f-\nu'_0 \bar{P}_{t-\tau_D+u}'^{\bfeta',\bfnu'}f\right|\ind{\tau_D \leq t - \delta}\right] + 2M\Pr_{\mu_0}(\tau_D \in [t-\delta,t]).
    \end{aligned}
  \end{equation}
  
  We first address the right-hand side of~\eqref{eq:pfpropdistrib:21}. For any $\rho>0$ and $\eta_0>0$, denote by $\mathcal{C}(\rho,\eta_0)$ the set of Euler schemes with redistribution $(\bar{X}^{\bfeta,\bfnu}_t)_{t \geq 0}$ which satisfy the condition~\eqref{eq:cond-cvdistrib}. For any $t > 0$, introduce
  \begin{equation*}
    F_{\rho,\eta_0}(t) = \sup_{\mathcal{C}(\rho,\eta_0)} \left|\mu P_t^\mu f - \nu_0 \bar{P}^{\bfeta,\bfnu}_t f\right|.
  \end{equation*}
  We now remark that if $(\bar{X}^{\bfeta,\bfnu}_t)_{t \geq 0}$ satisfies the condition~\eqref{eq:cond-cvdistrib:2}, then $\Pr$-almost surely, the scheme $(\bar{X}_t'^{\bfeta',\bfnu'})_{t \geq 0}$ belongs to $\mathcal{C}(\rho,\eta_0)$, which yields the crucial estimate
  \begin{equation*}
    \left|\mu P_{t-\tau_D}^\mu f - \nu'_0 \bar{P}_{t-\tau_D}'^{\bfeta',\bfnu'}f\right| \leq F_{\rho,\eta_0}(t-\tau_D),
  \end{equation*}
  and therefore
  \begin{align*}
    \Exp_{\Pi^*}\left[\left|\mu P_{t-\tau_D}^\mu f - \nu'_0 \bar{P}_{t-\tau_D}'^{\bfeta',\bfnu'}f\right|\ind{\tau_D \leq t - \delta}\right] &\leq \int_0^{t-\delta} F_{\rho,\eta_0}(t-s) \Pr_{\mu_0}(\tau_D \in \dd s)\\ 
    &\leq \int_0^t F_{\rho,\eta_0}(t-s) \Pr_{\mu_0}(\tau_D \in \dd s).
  \end{align*}
  
  Let us now focus on the second term in the right-hand side of~\eqref{eq:pfpropdistrib:21}. We fix $T>0$ and for any $r \in [\delta,T]$, we set
  \begin{equation*}
    \Omega_{\delta,r} = \{\exists k \geq 1: \bar{\tauel}'_k \in [r-\delta,r+\delta]\},
  \end{equation*}
  where $(\bar{\tauel}'_k)_{k \geq 1}$ is the jump sequence associated with the scheme $(\bar{X}_t'^{\bfeta',\bfnu'})_{t \geq 0}$. We next write, for any $u \in [-\delta,\delta]$,
  \begin{equation*}
    \left|\nu'_0 \bar{P}_r'^{\bfeta',\bfnu'}f-\nu'_0 \bar{P}_{r+u}'^{\bfeta',\bfnu'}f\right| \leq 2 \|f\|_\infty \Pr'\left(\Omega_{\delta,r}\right) + [f]_1 \Exp'\left[|\bar{X}_r'^{\bfeta',\bfnu'}-\bar{X}_{r+u}'^{\bfeta',\bfnu'}|\iind{\Omega_{\delta,r}^c}\right].
  \end{equation*}
  On the one hand, on the event $\Omega_{\delta,r}^c$, $(\bar{X}_t'^{\bfeta',\bfnu'})_{t \in [r-\delta,r+\delta]}$ coincides with an Euler scheme without redistribution, with a step sequence $\bfeta''$ which also satisfies $|\bfeta''| \leq \eta_0$. Therefore, by standard arguments, there exists $C \geq 0$ which only depends on $b$ and $\sigma$ such that 
  \begin{equation*}
    \sup_{r \in [\delta,T], u \in [-\delta,\delta]} \Exp'\left[|\bar{X}_r'^{\bfeta',\bfnu'}-\bar{X}_{r+u}'^{\bfeta',\bfnu'}|\iind{\Omega_{\delta,r}^c}\right] \leq C \sqrt{\delta}.
  \end{equation*}
  Let us fix $\varepsilon > 0$ and set $\delta_\varepsilon$ such that $C \sqrt{\delta_\varepsilon} \leq \varepsilon$. On the other hand, by~Lemma~\ref{lem:cvdistrib}, there exist $\rho>0$, $\eta_0>0$ and $\delta>0$ such that if $(\bar{X}_t'^{\bfeta',\bfnu'})_{t \geq 0} \in \mathcal{C}(\rho,\eta_0)$, then
  \begin{equation*}
    \sup_{r \in [\delta,T]} \Pr'\left(\Omega_{\delta,r}\right) \leq \varepsilon.
  \end{equation*}
  Therefore, up to replacing $\delta$ with $\delta \wedge \delta_\varepsilon$, we get, for any $t \in [0,T]$,
  \begin{equation*}
    \Exp_{\Pi^*}\left[\sup_{u \in [-\delta,\delta]} \left|\nu'_0 \bar{P}_{t-\tau_D}'^{\bfeta',\bfnu'}f-\nu'_0 \bar{P}_{t-\tau_D+u}'^{\bfeta',\bfnu'}f\right|\ind{\tau_D \leq t - \delta}\right] \leq (2M+1)\varepsilon.
  \end{equation*}
 For any compact subset $K \subset D$, the third term in the right-hand side of~\eqref{eq:pfpropdistrib:21} is easily controlled by
  \begin{equation*}
    2M\left(\sup_{x\in K} \Pr_x(\tau_D \in[t-\delta,t])+\mu_0(K^c)\right),
  \end{equation*}
  so that, by~Lemma~\ref{lem:cvdistrib:1}, up to decreasing $\delta$, we get that for any $t \in [0,T]$,
  \begin{equation*}
    2M\Pr_{\mu_0}(\tau_D \in [t-\delta,t]) \leq 2M\left(\varepsilon+\mu_0(K^c)\right).
  \end{equation*}
  
  Finally, by~Lemma~\ref{lem:wass-estim}, there exist $C \geq 0$ and $\eta_0 > 0$ such that if $|\bfeta| \leq \eta_0$ then for any $t \in [0,T]$,
  \begin{equation*}
    \Pr_{\Pi^*}\left(\tau_D \leq t \leq \bar{\tau}^\bfeta_D\right) + \Pr_{\Pi^*}\left(\bar{\tau}^\bfeta_D \leq t \leq \tau_D\right) + 2 \Pr_{\Pi^*}(|\tau_D-\bar{\tau}^\bfeta_D| \geq \delta) \leq \varepsilon+C {\cal W}_1(\mu_0,\nu_0)+\mu_0(K^c).
  \end{equation*}
  
  Therefore, up to decreasing $\eta_0$, for any $T > 0$, $\varepsilon > 0$, and compact subset $K \subset D$, we have constructed $\rho^{(2)}>0$, $\eta^{(2)}_0>0$ and $C^{(2)}_M \geq 0$, $C^{(2)} \geq 0$ such that, for any $(\bar{X}^{\bfeta,\bfnu}_t)_{t \geq 0}$ which satisfies~\eqref{eq:cond-cvdistrib:2}, for any $t \in [0,T]$,
  \begin{equation*}
    \mathrm{II} \leq \int_0^t F_{\rho,\eta_0}(t-s)\Pr_{\mu_0}(\tau_D \in \dd s) + C^{(2)}_M(\varepsilon + \mu_0(K^c)) + C^{(2)}\mathcal{W}_1(\mu_0,\nu_0).
  \end{equation*}
  As in Step~1, the constant $C^{(2)}_M$ only depends on $M$ and not on $T, \varepsilon, K$.

  \medskip
  \emph{Step~3: renewal argument.} Let us fix $T>0$, $\varepsilon > 0$ and $K$ a compact subset of $D$. Combining the results of Steps~1 and~2, there exist $\rho>0$, $\eta_0>0$ and $C_M \geq 0$, $C \geq 0$, such that $C_M$ only depends on $M$, and for any $(\bar{X}^{\bfeta,\bfnu}_t)_{t \geq 0}$ which satisfies~\eqref{eq:cond-cvdistrib:2}, for any $\mu_0 \in \mathcal{M}_1(D)$, for any $t \in [0,T]$,
  \begin{equation}\label{eq:renewal-ineq}
    \left|\mu_0 P_t^\mu f - \nu_0 \bar{P}^{\bfeta,\bfnu}_t f\right| \leq \int_0^t F_{\rho,\eta_0}(t-s)\Pr_{\mu_0}(\tau_D \in \dd s) + C_M\left(\varepsilon + \mu_0(K^c)\right) + C\mathcal{W}_1(\mu_0,\nu_0).
  \end{equation}
  Applying this inequality with $\mu_0=\mu$ and taking the supremum of the left-hand side over $\mathcal{C}(\rho,\eta_0)$, which is included in the set of schemes satisfying~\eqref{eq:cond-cvdistrib:2}, we deduce that
  \begin{equation*}
    \forall t \in [0,T], \qquad F_{\rho,\eta_0}(t) \leq \int_0^t F_{\rho,\eta_0}(t-s)\Pr_{\mu_0}(\tau_D \in \dd s) + C_M\left(\varepsilon + \mu(K^c)\right) + C\rho.
  \end{equation*}
  Let us now take $K$ such that $\mu(K^c) \leq \varepsilon$, and then decrease $\rho$ so that $C\rho \leq C_M\varepsilon$. Up to replacing $C_M$ with $3C_M$, we get
  \begin{equation*}
    \forall t \in [0,T], \qquad F_{\rho,\eta_0}(t) \leq \int_0^t F_{\rho,\eta_0}(t-s)\Pr_{\mu_0}(\tau_D \in \dd s) + C_M\varepsilon.
  \end{equation*}
  Iterating this inequality and using the fact that $F_{\rho,\eta_0}$ is nonnegative, we get for every $N\in\mathbb{N}^*$, for every $t\in[0,T]$,
  \begin{equation*}
    F_{\rho,\eta_0}(t)\le \int_0^t F_{\rho,\eta_0}(t-s) (\PE_\tau^\mu)^{\ast N}(\dd s)+ C_M \varepsilon\sum_{k=0}^{N-1} \PE_\mu(\tau_D\le t)^k,
  \end{equation*}
  where $(\PE_\tau^\mu)^{\ast k}=\PE_\tau^\mu\ast\ldots \ast\PE_\tau^\mu$ with $\ast$ denoting the convolution and $\PE_\tau^\mu$ the probability distribution of $\tau_D$ under $\PE_\mu$. For every $0\le t\le T$, $\PE_\mu(\tau_D\le t)\le\PE_\mu(\tau_D\le T)<1$. Then, since $F_{\rho,\eta_0}$ is bounded by $2 M$, it follows that for every $t\in[0,T]$,
  \begin{equation*}
    F_{\rho,\eta_0}(t)\le 2M\limsup_{N\to+\infty}\PE_\mu(\tau_D\le T)^N+ \frac{C_M \varepsilon}{\PE_\mu(\tau_D> T)}\le \frac{ C_M \varepsilon}{\PE_\mu(\tau_D> T)}.
  \end{equation*}
  Plugging this estimate in the right-hand side of~\eqref{eq:renewal-ineq}, and up to increasing the value of $C_M$, we conclude that for any $(\bar{X}^{\bfeta,\bfnu}_t)_{t \geq 0}$ which satisfies~\eqref{eq:cond-cvdistrib:2}, for any $\mu_0 \in \mathcal{M}_1(D)$,
  \begin{equation*}
    \sup_{t \in [0,T]} \left|\mu_0 P_t^\mu f - \nu_0 \bar{P}^{\bfeta,\bfnu}_t f\right| \leq C_M \left(\varepsilon + \mu_0(K^c)\right) + C \mathcal{W}_1(\mu_0,\nu_0).
  \end{equation*}

  \medskip
  \emph{Step~4: conclusion of the proof.} Since the result of Step~3 is valid for any function $f$ satisfying~\eqref{eq:defM}, by the Kantorovitch--Rubinstein duality for the Wasserstein distance~\cite[Lemma~11.8.3]{dudley89}, we easily deduce the estimate
  \begin{equation*}
    \sup_{t \in [0,T]} \mathcal{W}_1\left(\mu_0 P_t^\mu, \nu_0 \bar{P}^{\bfeta,\bfnu}_t\right) \leq C_M \left(\varepsilon + \mu_0(K^c)\right) + C \mathcal{W}_1(\mu_0,\nu_0).
  \end{equation*}
  Let us now assume that the scheme $(\bar{X}^{\bfeta,\bfnu}_t)_{t \geq 0}$ does not satisfies~\eqref{eq:cond-cvdistrib:2}, but merely~\eqref{eq:conditionnunnuu}. On the same probability space, endowed with the same filtration $(\mathcal{G}_t)_{t \geq 0}$, define the sequence $\tilde{\bfnu} = (\tilde{\nu}_n)_{n \geq 1}$ by
  \begin{equation*}
    \tilde{\nu}_n = \begin{cases}
      \nu_n & \text{if $\mathcal{W}_1(\nu_n,\mu) \leq \rho$,}\\
      \mu & \text{otherwise,}
    \end{cases}
  \end{equation*}
  and then construct the scheme $(\bar{X}^{\bfeta,\tilde{\bfnu}}_t)_{t \geq 0}$ with the same initial condition as $(\bar{X}^{\bfeta,\bfnu}_t)_{t \geq 0}$, the same Brownian motion $(B_t)_{t \geq 0}$, and, as long as $\tilde{\nu}_n=\nu_n$, the same restarting points $U^{\bfeta,\bfnu}_n$. Then on the one hand, it is clear that $(\bar{X}^{\bfeta,\tilde{\bfnu}}_t)_{t \geq 0}$ satisfies~\eqref{eq:cond-cvdistrib:2}, and therefore by the argument above,
  \begin{equation*}
    \sup_{t \in [0,T]} \mathcal{W}_1\left(\mu_0 P_t^\mu, \nu_0 \bar{P}^{\bfeta,\tilde{\bfnu}}_t\right) \leq C_M \left(\varepsilon + \mu_0(K^c)\right) + C \mathcal{W}_1(\mu_0,\nu_0).
  \end{equation*}
  On the other hand, for any function $f : D \to \R$ satisfying~\eqref{eq:defM}, 
  \begin{equation*}
    \sup_{t\in[0,T]} \left| \ES[f(\bar{X}^{\bfeta,\bfnu}_t)]-\ES[f(\bar{X}^{\bfeta,\tilde{\bfnu}}_t)\right|\le
2\|f\|_\infty\PE\left(\sup_{n\ge1} {\cal W}_1({\nu}_n, \mu)> \rho\right)\le 2 M \varepsilon.
  \end{equation*}
  Therefore, by the Kantorovitch--Rubinstein duality again and the triangle inequality for the Wasserstein distance,
  \begin{equation*}
    \sup_{t \in [0,T]} \mathcal{W}_1\left(\mu_0 P_t^\mu, \nu_0 \bar{P}^{\bfeta,\bfnu}_t\right) \leq (C_M+2M) \left(\varepsilon + \mu_0(K^c)\right) + C \mathcal{W}_1(\mu_0,\nu_0),
  \end{equation*}
  and the proof is completed at the price of replacing $C_M$ by $C_M+2M$.
\end{proof}

\subsection{Proof of Theorem~\ref{theo2}}\label{sec:prooftheo2} 

The starting point of the proof of~Theorem~\ref{theo2} is the transcription of the almost sure convergence of $\pn_n$ toward $\mu^\star$ given by~Theorem~\ref{theo:main} in Wasserstein distance.

\begin{lem}[Convergence of $\pn_n$ to $\mu^\star$ in Wasserstein distance]\label{lem:cv-wass-pn}
  Under the assumptions of~Theorem~\ref{theo:main},
  \begin{equation*}
    \mathcal{W}_1(\pn_n,\mu^\star) \xrightarrow{n\rightarrow+\infty}0, \qquad \text{almost surely.}
  \end{equation*}
\end{lem}
\begin{proof}
By Corollary~\ref{cor:pnconverges}, $(\pn_n)_{n\ge1}$ converges almost surely to $\mu^\star$ for the weak topology on ${\cal M}_1(D)$. But, by Remark~\ref{rk:topoPolish} and \cite[Corollary 6.13]{villani_book}, one knows that the Wasserstein distance metrizes the weak convergence on ${\cal M}_1(D)$. The result follows.
\end{proof}

We will also use the following property of the $\mu^\star$-return process. The argument is certainly standard, for the sake of completeness we give a sketch of its proof under our assumptions.
\begin{lem}[Uniform ergodicity of the $\mu^\star$-return process]\label{lem:ergo-return}
  Under Assumptions~\ref{cond:D} and~\ref{cond:coeffs}, for any continuous and bounded function $f : D \to \R$,
  \begin{equation*}
    \lim_{t \to +\infty} \sup_{\alpha \in \mathcal{M}_1(D)} |\alpha P^{\mu^\star}_t f - \mu^\star(f)| = 0.
  \end{equation*}
\end{lem}
\begin{proof}
  From the renewal identity~\eqref{eq:renewal} we get, for any $t \geq 0$ and $\alpha \in \mathcal{M}_1(D)$,
  \begin{equation*}
    \alpha P^{\mu^\star}_t f = \Exp_\alpha\left[f(Y_t)|t<\tau_D\right]\Pr_\alpha(t<\tau_D) + \int_0^t (\mu^\star P^{\mu^\star}_{t-s} f)\Pr_\alpha(\tau_D \in \dd s).
  \end{equation*}
  Applying this identity with $\alpha=\mu^\star$ and using the basic properties of $\mu^\star$ yields, for any $t \geq 0$,
  \begin{equation*}
    \mu^\star P^{\mu^\star}_t f = \mu^\star(f)\ee^{-\lambda^\star t} + \int_0^t (\mu^\star P^{\mu^\star}_{t-s}f) \lambda^\star \ee^{-\lambda^\star s}\dd s,
  \end{equation*}
  from which one easily deduces the standard statement that $\mu^\star P^{\mu^\star}_t f = \mu^\star(f)$. As a consequence
  \begin{equation*}
    \alpha P^{\mu^\star}_t f = \Exp_\alpha\left[f(Y_t)|t<\tau_D\right]\Pr_\alpha(t<\tau_D) + \mu^\star(f)\Pr_\alpha(\tau_D \leq t)
  \end{equation*}
  and the result follows from the fact that, by~\cite[Eq.~(3.3)]{BenChaVil22}, 
  \begin{equation*}
    \lim_{t \to +\infty} \sup_{\alpha \in \mathcal{M}_1(D)} \left|\Exp_\alpha\left[f(Y_t)|t<\tau_D\right] - \mu^\star(f)\right| = 0.\qedhere
  \end{equation*}
\end{proof}

For any $n \geq 0$, we now denote $\alpha_n = \mathcal{L}(\bar{X}_{\Gamma_n})$ and prove the tightness of the sequence $(\alpha_n)_{n \geq 0}$.

\begin{prop}[Tightness]\label{prop:theo2-tight}
  Under the assumptions of~Theorem~\ref{theo2}, the sequence $(\alpha_n)_{n \geq 0}$ is tight.
\end{prop}
\begin{proof}
  We first introduce some notation. For $n$ large enough, by Assumption~\ref{it:steps-1} we may find $\varphi(n) < n$ such that $\Gamma_n - 2 \leq \Gamma_{\varphi(n)} \leq \Gamma_n - 1$. For such $n$, we set $\tauel^n_0 = \Gamma_{\varphi(n)}$ and $\{\tauel^n_k, k \geq 1\} = \{\Gamma_{n'}, n' > \varphi(n), \theta_{n'}=1\}$, similarly to the introduction of Section~\ref{sec:prooftheomain}. With these notations, we have
  \begin{equation*}
    \PE(\bar{X}_{\Gamma_n}\in A)=\sum_{k\ge0} \PE(\bar{X}_{\Gamma_n}\in A, \tauel_k^n\le \Gamma_n<\tauel_{k+1}^n).
  \end{equation*}
  In order to take advantage of the almost sure convergence of $\pn_n$ to $\mu^\star$, for $\rho>0$ we introduce the events
  \begin{equation*}
    \Upsilon^{n,\rho}:=\left\{\sup_{\varphi(n) \le \ell}{\cal W}_1(\pn_{\ell},\mu^\star)\le \rho\right\}, \qquad \Upsilon^{n,k,\rho}:=\left\{ \sup_{\varphi(n)\le \ell \le N(\bar\tauel_k^n) }{\cal W}_1(\pn_{\ell},\mu^\star)\le \rho\right\},
  \end{equation*}
  where, as in~Subsection~\ref{ss:weakerror}, we have defined $N(t) = \inf\left\{N \geq 1: \Gamma_N > t\right\}$, so that $n \leq N(t)$ if and only if $\Gamma_n \leq t$. By construction, $\Upsilon^{n,\rho}\subset \Upsilon^{n,k,\rho}$, so that
  \begin{equation}\label{eq:arg1tensionnun}
    \PE(\bar{X}_{\Gamma_n}\in A)\le \PE((\Upsilon^{n,\rho})^c)+\sum_{k\ge0} \PE(\bar{X}_{\Gamma_n}\in A, \tauel_k^n\le \Gamma_n<\tauel_{k+1}^n, \Upsilon^{n,k,\rho}).
  \end{equation}
  
  By~Lemma~\ref{lem:cv-wass-pn}, $\PE((\Upsilon^{n,\rho})^c)\to 0$ as $n\to+\infty$. We thus focus on the second term. For $k=0$, we deduce conditioning on $\mathcal{F}_{\Gamma_{\varphi(n)}}$ that
  \begin{equation*}
    \PE(\bar{X}_{\Gamma_n}\in A, \tauel_0^n\le \Gamma_n<\tauel_1^n, \Upsilon^{n,0,\rho})\le\ES[\Psi_A(\bar{X}_{\Gamma_{\varphi(n)}},\Gamma_n-\Gamma_{\varphi(n)},\boldsymbol{\gamma}^n)]\le \sup_{x\in D,t\in[1,2]} \Psi_A(x,t,\boldsymbol{\gamma}^n)\label{eq:argumentihopethelast}
  \end{equation*}
  where we have set $\boldsymbol{\gamma}^n = (\gamma_{\varphi(n)+k})_{k \geq 1}$ and defined $\Psi_A(x,t,\bfeta)=\PE_x(\bar{Y}^\bfeta_t\in A,\bar{\tau}_{D}^\bfeta>t)$. Now, let $k\ge1$. Conditioning with respect to ${\cal F}_{\tauel_k^n}$, we get
  \begin{equation*}
    \PE(\bar{X}_{\Gamma_n}\in A, \tauel_k^n\le \Gamma_n<\tauel_{k+1}^n, \Upsilon^{n,k,\rho}) = \ES\left[\Psi_A\left(\bar{X}_{\tauel_k^n},\Gamma_n-\tauel_k^n,\boldsymbol{\gamma}^{N(\tauel_k^n)}\right) \ind{\tauel_k^n \le \Gamma_n,\Upsilon^{n,k,\rho}}\right].
  \end{equation*}
  But, since the event $\{\tauel_k^n \le \Gamma_n,\Upsilon^{n,k,\rho}\}$ is in ${\cal F}_{(\tauel_k^n)^-}$, we can still condition with respect to ${\cal F}_{(\tauel_k^n)^-}$ in order to obtain
  \begin{equation*}
    \PE(\bar{X}_{\Gamma_n}\in A, \tauel_k^n\le \Gamma_n<\tauel_{k+1}^n, \Upsilon^{n,k,\rho}) = \ES\left[\int_D \Psi_A\left(x,\Gamma_n-\tauel_k^n,\boldsymbol{\gamma}^{N(\tauel_k^n)}\right)\pn_{N(\tauel_k^n)}(\dd x) \ind{\tauel_k^n \le \Gamma_n,\Upsilon^{n,k,\rho}}\right].
  \end{equation*}
  Then, for any subset $B$ of $D$, we have:
  \begin{equation*}
    \int_D \Psi_A\left(x,\Gamma_n-\tauel_k^n,\boldsymbol{\gamma}^{N(\tauel_k^n)}\right)\pn_{N(\tauel_k^n)}(\dd x) \leq \pn_{N(\tauel_k^n)}(B^c) + \sup_{x\in B,t \le 2,|\bfeta|\le |\boldsymbol{\gamma}^{N(\tauel_k^n)}|} \Psi_A(x,t,\bfeta).
  \end{equation*}
  Therefore,
  \begin{equation*}
    \iind{\Upsilon^{n,k,\rho}} \int_D \Psi_A\left(x,\Gamma_n-\tauel_k^n,\boldsymbol{\gamma}^{N(\tauel_k^n)}\right)\pn_{N(\tauel_k^n)}(\dd x) \leq \tilde{\Psi}_{A,B}(|\boldsymbol{\gamma}^{N(\tauel_k^n)}|, \rho),
  \end{equation*}
  where
  \begin{equation*}
    \tilde{\Psi}_{A,B}(\eta,\rho) := \sup_{\mu, {\cal W}_1(\mu,\mu^\star)\le \rho}\mu(B^c)+\sup_{x\in B,t\le 2,|\bfeta|\le \eta}  \Psi_A(x,t,\bfeta).
  \end{equation*}
  Hence,
  \begin{equation*}
    \PE(\bar{X}_{\Gamma_n}\in A, \tauel_k^n\le \Gamma_n<\tauel_{k+1}^n, \Upsilon^{n,k,\rho}) \leq \tilde{\Psi}_{A,B}(|\boldsymbol{\gamma}^{N(\tauel_k^n)}|, \rho) \PE(\tauel_k^n \le \Gamma_n,\Upsilon^{n,k,\rho}).
  \end{equation*}
  We deduce from~\eqref{eq:arg1tensionnun}, what precedes and the fact that $|\boldsymbol{\gamma}^n| \to 0$ that for any positive $\rho$, for any subsets $A$ and $B$ of $D$,
  \begin{equation}\label{eq:refdsks}
    \begin{aligned}
      \limsup_{n\to+\infty} \PE(\bar{X}_{\Gamma_n} \in A) &\leq \limsup_{\eta \to 0} \sup_{x\in D,t\in[1,2],|\bfeta| \leq \eta} \Psi_A(x,t,\bfeta)\\
      &\quad + \limsup_{\eta \to 0}\tilde{\Psi}_{A,B}(\eta, \rho) \limsup_{n \to +\infty} \sum_{k \geq 1} \PE(\tauel_k^n \le \Gamma_n,\Upsilon^{n,k,\rho}).
    \end{aligned}
  \end{equation}
  
  Let $\varepsilon>0$. As in the proof of~Lemma~\ref{lem:cvdistrib}, set $K_r=\{x\in D, d(x,\partial D)\ge r\}$ and choose $r_1>0$, $\rho_\varepsilon>0$ such that $\sup_{\mu, {\cal W}_1(\mu,\mu^\star)\le \rho_\varepsilon}\mu((K_{r_1})^c)\le \varepsilon$. We thus fix $B=K_{r_1}$. Now, we set $A=K_{r_2}^c$ for some $r_2 \leq r_1/2$ to be fixed below. On the one hand, since $d(B,A)=r_1-r_2 \geq r_1/2$, we have for all $x\in B$, 
  \begin{equation*}
    \PE_x(\bar{Y}_t^\bfeta\in A)\le \PE_x(|\bar{Y}_t^\bfeta-x|\ge r_1-r_2)\le\frac{1}{r_1-r_2}\ES_x\left[|\bar{Y}_t^\bfeta-x|\right]\le \frac{2C \sqrt{t}}{r_1},
  \end{equation*}
  for some constant $C$ which only depends on $b$ and $\sigma$. Hence, we can find some $t_0 \in (0,2]$, which depends on $\varepsilon$ and $r_1$ but not on $r_2$, such that, for any $\eta$,
  \begin{equation*}
    \sup_{x\in B,t\le t_0,|\bfeta|\le \eta} \Psi_A(x,t,\bfeta)\le \varepsilon.
  \end{equation*}
  Now, for $t\ge t_0$, for any sequence $\bfeta$,
  \begin{align*}
    \PE_x(\bar{Y}_{t}^\bfeta\in A)\le \PE_x({Y}_{t}\in K_{2 r_2}^c)+\PE_x(|\bar{Y}_{t}^\bfeta-{Y}_{t}|\ge r_2).
  \end{align*}
  By~Lemma~\ref{lem:gauss-density},
  \begin{equation*}
    \sup_{t\in[ t_0,2], (x,y)\in D\times D} p_{t}(x,y)\le C<+\infty,
  \end{equation*}
  where $C$ depends on $t_0$. Thus, by~Lemma~\ref{lem:strong-pages},
  \begin{align*}
    \PE_x(\bar{Y}_{t}^\bfeta\in A)\le C \lambda_d(K_{2 r_2}^c)+\frac{1}{r_2}\ES_x \left[|\bar{Y}_{t}^\bfeta-{Y}_{t}|\right]\le C \left(\lambda_d(K_{2 r_2}^c)+\frac{\sqrt{|\bfeta|}}{r_2}\right).
  \end{align*}
  Since $\lambda_d(K_{2 r_2}^c)\xrightarrow{r_2\to0}0$, we may now fix $r_2$ small enough such that
  \begin{equation*}
    \limsup_{\eta \to 0} \sup_{x\in B,t \in [t_0,2], |\bfeta|\le \eta} \Psi_A(x,t,\bfeta) \leq \varepsilon.
  \end{equation*}
  Up to renormalizing $\varepsilon$, it follows that
  \begin{equation}\label{eq:refdsks:1}
    \limsup_{\eta \to 0} \sup_{x\in D,t\in[1,2],|\bfeta| \leq \eta} \Psi_A(x,t,\bfeta) \leq \varepsilon, \qquad \limsup_{\eta \to 0}\tilde{\Psi}_{A,B}(\eta, \rho) \leq \varepsilon.
  \end{equation}
  It remains now to control the series which appears in \eqref{eq:refdsks}. Noting that $\Gamma_n-\Gamma_{\varphi(n)}\le 2$, the same arguments as in the last part of the proof of~Lemma~\ref{lem:cvdistrib} lead to the following statement: there exist $n_0\ge1$ and a positive constant $C$ such that for every $n\ge n_0$,
  \begin{equation*}
    \sum_{k\ge1} \PE( \tauel_k^n \le \Gamma_n,\Upsilon^{n,k,\rho})\le C<+\infty.
  \end{equation*}
  Combining the latter estimate with~\eqref{eq:refdsks} and~\eqref{eq:refdsks:1}, and recalling that $A$ is the complement of the compact subset $K_{r_2}$, we complete the proof.
\end{proof}

We may finally complete the proof of~Theorem~\ref{theo2}.

\begin{proof}[Proof of~Theorem~\ref{theo2}]
  Since, by~Proposition~\ref{prop:theo2-tight}, the sequence $(\alpha_n)_{n \geq 0}$ is tight, by Assumption~\ref{it:steps-1}, Lemma~\ref{lem:cv-wass-pn} and Remark~\ref{rk:informationavantzero}, one may apply~Theorem~\ref{theo:weakerrorf} to the scheme $(\bar{X}_{\Gamma_n + t})_{t \geq 0}$ with $\mu_0^n=\nu_0^n=\alpha_n$ and $\mu=\mu^\star$ to get, for any $T>0$,
  \begin{equation}\label{eq:pf-theo2:ident}
    \lim_{n \to +\infty} \sup_{t \in [0,T]} \mathcal{W}_1\left(\alpha_n P^{\mu^\star}_t, \mathcal{L}(\bar{X}_{\Gamma_n+t})\right) = 0.
  \end{equation}
  Let $f : D \to \R$ be Lipschitz continuous and let $\varepsilon > 0$. By~Lemma~\ref{lem:ergo-return}, there exists $T>0$ such that
  \begin{equation*}
    \sup_{t \in [T/2,T]} \sup_{n \geq 1} |\alpha_n P^{\mu^\star}_t f - \mu^\star(f)| \leq \varepsilon.
  \end{equation*}
  Now, by Assumption~\ref{it:steps-1}, for $n$ large enough there exists $\varphi(n) < n$ such that $\Gamma_n \in [\Gamma_{\varphi(n)}+T/2,\Gamma_{\varphi(n)}+T]$, and obviously $\varphi(n) \to +\infty$. Thus we have
  \begin{align*}
    \left|\Exp\left[f(\bar{X}_{\Gamma_n})\right] - \mu^\star(f)\right| &\leq \sup_{t \in [T/2,T]} \left|\Exp\left[f(\bar{X}_{\Gamma_{\varphi(n)}+t})\right] - \mu^\star(f)\right|\\
    &\leq \sup_{t \in [T/2,T]} \left|\Exp\left[f(\bar{X}_{\Gamma_{\varphi(n)}+t})\right] - \alpha_{\varphi(n)} P^{\mu^\star}_t f\right| + \sup_{t \in [T/2,T]} \left|\alpha_{\varphi(n)} P^{\mu^\star}_t f - \mu^\star(f)\right|.
  \end{align*}
  By~\eqref{eq:pf-theo2:ident} applied along the subsequence $\varphi(n)$, and using the Lipschitz continuity of $f$, the first term in the right-hand side vanishes when $n \to +\infty$, while the second term is bounded by $\varepsilon$ uniformly in $n$. Therefore $\alpha_n(f) \to \mu^\star(f)$ and the proof is completed.
\end{proof}
{\begin{rk} In the above proof, Assumption~\ref{it:steps-3} is only used in the assumptions of Lemma~\ref{lem:cv-wass-pn}. But, by Remark~\ref{withoutH3c}, one easily checks that  Lemma~\ref{lem:cv-wass-pn} holds without Assumption~\ref{it:steps-3}. This means that Theorem~\ref{theo2} holds true without Assumption~\ref{it:steps-3}.
\end{rk}}

\begin{center}
	\vskip 1cm
	{\Large Supplementary material}
	\vskip 1cm
\end{center}

This Supplementary Material document contains proofs of results stated in the main part of the article, which are either technical developments or close to standard arguments.

\section{Useful facts on the one-dimensional reflected Brownian motion with drift}\label{app:RBM}

In this section, we collect various definitions and estimates related with one-dimensional reflected Brownian motions with drift. These results are used in the proofs of Section~\ref{sec:tightness} of the article.

\subsection{Positive and negative reflection maps} 

For some interval $[r_0,r_1)$ with $-\infty < r_0 < r_1 \leq +\infty$, we denote by $D([r_0,r_1),A)$ the set of cadlag paths indexed by $[r_0,r_1)$ with values in $A$. Given $z \in \R$, we define the \emph{positive reflection map at $z$} 
\begin{equation*}
  \boldsymbol{\Gamma}^{+,z} : \left\{\begin{array}{ccc}
    D([r_0,r_1),\R) & \to & D([r_0,r_1),[z,+\infty))\\
    \beta_\bullet & \mapsto & Z^+_\bullet
  \end{array}\right.
\end{equation*}
by
\begin{equation*}
  \forall r \in [r_0, r_1), \qquad Z^+_r = \beta_r + \max_{r_0 \leq u \leq r} [\beta_u-z]_-,
\end{equation*}
and the \emph{negative reflection map at $z$} 
\begin{equation*}
  \boldsymbol{\Gamma}^{-,z} : \left\{\begin{array}{ccc}
    D([r_0,r_1),\R) & \to & D([r_0,r_1),(-\infty,z])\\
    \beta_\bullet & \mapsto & Z^-_\bullet
  \end{array}\right.
\end{equation*}
by
\begin{equation*}
  \forall r \in [r_0, r_1), \qquad Z^-_r = \beta_r - \max_{r_0 \leq u \leq r} [\beta_u-z]_+.
\end{equation*}

\begin{prop}[Properties of the reflection map]\label{prop:refmap}
  With the notation introduced above, we have the following statements.
  \begin{enumerate}[label=(\roman*),ref=\roman*]
    \item\label{it:prop-refmap:accr} For any $r_0 \leq r' \leq r < r_1$, we have $Z^-_r-Z^-_{r'} \leq \beta_r - \beta_{r'} \leq Z^+_r-Z^+_{r'}$.
    \item\label{it:prop-refmap:flow} For any $r_0 \leq r' < r_1$, we have $Z^\pm_\bullet = \boldsymbol{\Gamma}^{\pm,z}(Z^\pm_{r'} + \beta_\bullet - \beta_{r'})$ on $[r',r_1)$.
  \end{enumerate}
\end{prop}
\begin{proof}
  The first point is immediate. We show the second point for $\boldsymbol{\Gamma}^{+,0}$, the argument carries over to general reflection maps without any difficulty. For any $r_0 \leq r' \leq r < r_1$, the claimed identity rewrites
  \begin{equation*}
    Z^+_r = Z^+_{r'} + \beta_r - \beta_{r'} + \max_{r' \leq v \leq r} [Z^+_{r'} + \beta_v - \beta_{r'}]_-,
  \end{equation*}
  and it is easily seen to hold if and only if
  \begin{equation*}
    \max_{r_0 \leq u \leq r} [\beta_u]_- = \max_{r_0 \leq u \leq r'} [\beta_u]_- + \max_{r' \leq v \leq r} \left[\beta_v + \max_{r_0 \leq u \leq r'} [\beta_u]_-\right]_-.
  \end{equation*}
  To prove the latter identity we distinguish between two cases, namely whether the maximum of $[\beta_u]_-$ is reached on $[r_0,r']$ or on $[r',r]$. In the first case, we have
  \begin{equation*}
    \max_{r_0 \leq u \leq r} [\beta_u]_- = \max_{r_0 \leq u \leq r'} [\beta_u]_-,
  \end{equation*}
  while $\beta_v + \max_{r_0 \leq u \leq r'} [\beta_u]_- \geq 0$ for any $v \in [r',r]$ and thus
  \begin{equation*}
    \max_{r' \leq v \leq r} \left[\beta_v + \max_{r_0 \leq u \leq r'} [\beta_u]_-\right]_- = 0,
  \end{equation*}
  which proves the identity. In the second case, we have
  \begin{equation*}
    \max_{r' \leq v \leq r} \left[\beta_v + \max_{r_0 \leq u \leq r'} [\beta_u]_-\right]_- = \max_{r' \leq v \leq r} -\left(\beta_v + \max_{r_0 \leq u \leq r'} [\beta_u]_-\right) = \max_{r' \leq v \leq r} [\beta_v]_- - \max_{r_0 \leq u \leq r'} [\beta_u]_-,
  \end{equation*}
  and therefore
  \begin{equation*}
    \max_{r_0 \leq u \leq r} [\beta_u]_- = \max_{r' \leq v \leq r} [\beta_v]_- = \max_{r_0 \leq u \leq r'} [\beta_u]_- + \max_{r' \leq v \leq r} \left[\beta_v + \max_{r_0 \leq u \leq r'} [\beta_u]_-\right]_-,
  \end{equation*}
  which completes the proof.
\end{proof}

\subsection{Hitting times of reflected Brownian motions}\label{ss:T-mbr}

In this section, we let $(\mathcal{G}_r)_{r \geq 0}$ be a filtration and $(\mathsf{W}_r)_{r \geq 0}$ be a one-dimensional $(\mathcal{G}_r)_{r \geq 0}$-Brownian motion.

For $c \in \R$ and $z_0 \geq z$, we denote by $(\mathsf{Z}^{+,z;c}_{z_0,r})_{r \geq 0}$ the Brownian motion with constant drift $c$, positively reflected at the level $z$, and started from $z_0$ at $r=0$. In explicit terms,
\begin{equation*}
  \mathsf{Z}^{+,z;c}_{z_0, \bullet} = \boldsymbol{\Gamma}^{+,z} \upbeta^c_{z_0, \bullet}, \qquad \upbeta^c_{z_0,r} := z_0 + cr + \mathsf{W}_r.
\end{equation*}
Similarly, for $c \in \R$ and $z_0 \leq z$, we define $(\mathsf{Z}^{-,z;c}_{z_0,r})_{r \geq 0}$ by
\begin{equation*}
  \mathsf{Z}^{-,z;c}_{z_0, \bullet} = \boldsymbol{\Gamma}^{-,z} \upbeta^c_{z_0, \bullet}.
\end{equation*}

\begin{prop}[Strong Markov property]\label{prop:strongMarkovZ}
  For $c \in \R$ and $z_0 \geq z$, the process $(\mathsf{Z}^{+,z;c}_{z_0,r})_{r \geq 0}$ is adapted to $(\mathcal{G}_r)_{r \geq 0}$ and has the strong Markov property, in the sense that if $\varrho$ is a stopping time, then conditionally on the event $\{\varrho < +\infty, \mathsf{Z}^{+,z;c}_{z_0,\varrho} = z'\}$, the process $(\mathsf{Z}^{+,z;c}_{z_0,r+\varrho})_{r \geq 0}$ has the same law as $(\mathsf{Z}^{+,z;c}_{z',r})_{r \geq 0}$. Similar statements hold for the negatively reflected process $(\mathsf{Z}^{-,z;c}_{z_0,r})_{r \geq 0}$.
\end{prop}

Proposition~\ref{prop:strongMarkovZ} directly follows from Proposition~\ref{prop:refmap}~\eqref{it:prop-refmap:flow} combined with the strong Markov property for the Brownian motion $(\mathsf{W}_r)_{r \geq 0}$.

\medskip
The following invariances in law are immediate:
\begin{itemize}
  \item translation: for any $u \in \R$, $\mathsf{Z}^{\pm,z+u;c}_{z_0+u, \bullet}$ has the same law as $\mathsf{Z}^{\pm,z;c}_{z_0, \bullet}$;
  \item symmetry: $\mathsf{Z}^{+,z;c}_{z_0, \bullet}$ has the same law as $\mathsf{Z}^{-,-z;-c}_{-z_0, \bullet}$.
\end{itemize}

These invariance properties allow to reduce the study of hitting times for reflected drifted Brownian motions to the random variable
\begin{equation*}
  \mathsf{T}^c_{z_0, z_1} := \inf\{r \geq 0: \mathsf{Z}^{+,0;c}_{z_0, r} = z_1\}, \qquad z_1 \geq z_0 \geq 0,
\end{equation*}
defined from the drifted Brownian motion with positive reflection at $0$.

\begin{prop}[Moments of $\mathsf{T}^c_{z_0, z_1}$]\label{prop:estimT}
  Let $0 \leq z_0 \leq z_1$. For any $c \in \R$, for any $k \geq 1$,
  \begin{equation*}
    \Exp\left[\left(\mathsf{T}^c_{z_0, z_1}\right)^k\right] < +\infty.
  \end{equation*}
\end{prop}
\begin{proof}  From the definition of $(\mathsf{Z}^{+,0;c}_{z_0, r})_{r \geq 0}$ and Proposition~\ref{prop:refmap}~\eqref{it:prop-refmap:accr}, it is easy to show that 
  \begin{equation*}
    \inf_{z_0 \in [0,z_1]} \Pr\left(\mathsf{T}^c_{z_0, z_1} \leq 1\right) \geq \inf_{z_0 \in [0,z_1]} \Pr\left(\sup_{r \in [0,1]} \upbeta^c_{z_0,r} \geq z_1\right) > 0.
  \end{equation*}
  We therefore deduce from the Markov property that $\mathsf{T}^c_{z_0, z_1}$ has geometric tails and thus finite moments of all orders.  
\end{proof}

Proposition~\ref{prop:estimT} allows us to introduce the notation
\begin{equation*}
  \bar{\mathcal{R}}^c_{z_0, z_1} := \Exp\left[\mathsf{T}^c_{z_0, z_1}\right],
\end{equation*}
and for $\eta \in (0, z_0)$, 
\begin{equation*}
  \mathcal{R}^c_{z_0, z_1}(\eta) := \Exp\left[\int_{r=0}^{\mathsf{T}^c_{z_0, z_1}} \ind{\mathsf{Z}^{+,0;c}_{z_0, r} < \eta}\dd r\right].
\end{equation*}

\begin{cor}[Continuity of $\mathcal{R}^c_{z_0, z_1}$]\label{cor:estimT}
  The function $\mathcal{R}^c_{z_0, z_1}$ is continuous on $(0,z_0)$ and satisfies
  \begin{equation*}
    \lim_{\eta \dto 0} \mathcal{R}^c_{z_0, z_1}(\eta) = 0.
  \end{equation*}
\end{cor}
\begin{proof}
  We first rewrite, for any $\eta \in (0,z_0)$,
  \begin{equation*}
    \mathcal{R}^c_{z_0, z_1}(\eta) = \int_{r=0}^{+\infty} \Pr\left(\mathsf{T}^c_{z_0, z_1} > r, \mathsf{Z}^{+,0;c}_{z_0, r} < \eta\right)\dd r,
  \end{equation*}
  so that by Proposition~\ref{prop:estimT} and the dominated convergence theorem, to prove the corollary it suffices to show that for any $r \geq 0$, the nonnegative random variable $\mathsf{Z}^{+,0;c}_{z_0, r}$ has a density with respect to the Lebesgue measure.   { For $c=0$, this follows for instance from the fact that the process $(\mathsf{Z}^{+,0;0}_{z_0, r})_{r \geq 0}$ has the same law as $(|z_0 + \mathsf{W}_r|)_{r \geq 0}$~\cite[Corollary~2.2, p.~240]{RevYor99}.}
  By the Girsanov theorem, the absolute continuity of $\mathsf{Z}^{+,0;c}_{z_0, r}$ is then preserved for any value of $c \in \R$.
\end{proof}

\begin{rk}
  The density of $\mathsf{Z}^{+,0;c}_{z_0, r}$ is explicitly computed in~\cite[Section~4.2]{Lin05}.
\end{rk}

Finally, since $\mathsf{T}^c_{z_0, z_1}$ is almost surely finite for arbitrarily large $z_1$, the invariance properties combined with Proposition~\ref{prop:estimT} also entail the following immediate corollary. 

\begin{cor}[Unboundedness of trajectories]\label{cor:unboundZ}
  For any $c \in \R$ and $z \leq z_0$,
  \begin{equation*}
    \limsup_{r \to +\infty} \mathsf{Z}^{+,z;c}_{z_0, r} = +\infty, \qquad \text{almost surely;}
  \end{equation*}
  and similarly, for $z \geq z_0$,
  \begin{equation*}
    \liminf_{r \to +\infty} \mathsf{Z}^{-,z;c}_{z_0, r} = -\infty, \qquad \text{almost surely.}
  \end{equation*}
\end{cor}

\begin{rk}
  The proofs of Proposition~\ref{prop:estimT} and Corollaries~\ref{cor:estimT} and~\ref{cor:unboundZ} could also follow from the remark that, by the Tanaka formula, the process $(\mathsf{Z}^{+,0;c}_{z_0, r})_{r \geq 0}$ has the same law as the process $(|\mathsf{X}_r|)_{r \geq 0}$, where $(\mathsf{X}_r)_{r \geq 0}$ is the solution to the stochastic differential equation
  \begin{equation*}
    \dd \mathsf{X}_r = c \sgn(\mathsf{X}_r)\dd r + \dd \mathsf{W}_r, \qquad \mathsf{X}_0 = z_0.
  \end{equation*}
  This allows to use general results on the integrability of exit times from bounded sets, the existence of a density, and the ergodic behavior of one-dimensional diffusions with bounded measurable drift and additive noise.
\end{rk}

\section{Postponed proofs from Subsection~\ref{ss:zeta}}\label{app:pfs-tension}

This section contains the proofs of Proposition~\ref{prop:ZZpq}, Lemma~\ref{lem:compR} and Lemma~\ref{lem:LLN}.

\subsection{Proof of Proposition~\ref{prop:ZZpq}}\label{ss:pf-ZZpq}

As a preliminary for the proof of Proposition~\ref{prop:ZZpq}, we clarify the relation between $(Z_r)_{r \in [S_q,T_q)}$ and $(\beta_r)_{r \in [S_q,T_q)}$ in Lemma~\ref{lem:baromega} below. In this statement, we define $(\bar{\beta}_r)_{r \in [S_q,T_q)}$ by
\begin{equation*}
  \bar{\beta}_r = \begin{cases}
    \beta_{S_q} = \frac{\eta_0}{3} & \text{on $[S_q, \Delta_{n_q} \wedge T_q)$,}\\
    \beta_{\Delta_n} = \frac{\eta_0}{3} + \omega_{\Delta_n} - \omega_{S_q} & \text{on $[\Delta_n, \Delta_{n+1}\wedge T_q)$, for any $n \in \{n_q, \ldots, m_q-1\}$.}
  \end{cases}
\end{equation*}

\begin{lem}[$Z_{\Delta_n}$ as a reflected scheme]\label{lem:baromega}
  For any $n \in \{n_q, \ldots, m_q-1\}$,
  \begin{equation*}
    Z_{\Delta_n} = \left(\boldsymbol{\Gamma}^{+,0}\bar{\beta}_\bullet\right)_{\Delta_n}.
  \end{equation*}
\end{lem}
\begin{proof}
  We first note that if $\Delta_{n_q} \geq T_q$, that is to say $n_q=m_q$, then the statement is empty. We therefore assume that $\Delta_{n_q} < T_q$ and first remark that, by construction of the process $(\bar{\beta}_r)_{r \in [S_q,T_q)}$ and since $\bar{\beta}_{S_q} = \eta_0/3>0$, for any $n \in \{n_q, \ldots, m_q-1\}$, we have the identity
  \begin{equation}\label{eq:pf:baromega}
    \left(\boldsymbol{\Gamma}^{+,0}\bar{\beta}_\bullet\right)_{\Delta_n} = \beta_{\Delta_n} + \max_{n_q \leq n' \leq n} [\beta_{\Delta_{n'}}]_-.
  \end{equation}
  In particular, for $n=n_q$, we deduce that
  \begin{equation*}
    \left(\boldsymbol{\Gamma}^{+,0}\bar{\beta}_\bullet\right)_{\Delta_{n_q}} = \beta_{\Delta_{n_q}} + [\beta_{\Delta_{n_q}}]_- = [\beta_{\Delta_{n_q}}]_+,
  \end{equation*}
  which clearly coincides with $Z_{\Delta_{n_q}}$.
  
  Let us now assume that $n \in \{n_q, \ldots, m_q-1\}$ is such that $Z_{\Delta_n} = \left(\boldsymbol{\Gamma}^{+,0}\bar{\beta}_\bullet\right)_{\Delta_n}$. Then, by Proposition~\ref{prop:refmap}~\eqref{it:prop-refmap:flow}, we have
  \begin{equation*}
    \left(\boldsymbol{\Gamma}^{+,0}\bar{\beta}_\bullet\right)_{\Delta_{n+1}} = \left(\boldsymbol{\Gamma}^{+,0}\left(Z_{\Delta_n} + \bar{\beta}_\bullet - \bar{\beta}_{\Delta_n}\right)\right)_{\Delta_{n+1}}.
  \end{equation*}
  But it follows from the definition of $\bar{\beta}_\bullet$, and the fact that $Z_{\Delta_n} \geq 0$, that
  \begin{equation*}
    \max_{\Delta_n \leq v \leq \Delta_{n+1}} \left[Z_{\Delta_n} + \bar{\beta}_v - \bar{\beta}_{\Delta_n}\right]_- = \left[Z_{\Delta_n} + \beta_{\Delta_{n+1}} - \beta_{\Delta_n}\right]_-,
  \end{equation*}
  so that
  \begin{equation*}
    \left(\boldsymbol{\Gamma}^{+,0}\bar{\beta}_\bullet\right)_{\Delta_{n+1}} = \left[Z_{\Delta_n} + \beta_{\Delta_{n+1}} - \beta_{\Delta_n}\right]_+ = \left[Z_{\Delta_n} + \omega_{\Delta_{n+1}} - \omega_{\Delta_n}\right]_+.
  \end{equation*}
  Since $\Delta_{n+1} < T_q$, $Z_\bullet$ does not reflect at $\eta_0$ on $[\Delta_n, \Delta_{n+1}]$ and therefore the right-hand side above coincides with $Z_{\Delta_{n+1}}$. The end of the proof follows by induction.
\end{proof}

We are now ready to complete the proof of Proposition~\ref{prop:ZZpq}.

\begin{proof}[Proof of~Proposition~\ref{prop:ZZpq}]
  We deduce from Lemma~\ref{lem:baromega} and~\eqref{eq:pf:baromega} that, for any $n \in \{n_q, \ldots, m_q-1\}$,
  \begin{equation*}
    Z^+_{q,\Delta_n} - Z_{\Delta_n} = \left(\boldsymbol{\Gamma}^{+,0}\beta_\bullet\right)_{\Delta_n} - \left(\boldsymbol{\Gamma}^{+,0}\bar{\beta}_\bullet\right)_{\Delta_n} = \max_{S_q \leq u \leq \Delta_n} [\beta_u]_- - \max_{n_q \leq n' \leq n} [\beta_{\Delta_{n'}}]_-,
  \end{equation*}
  from which it immediately follows that
  \begin{equation}\label{eq:pfZZpq:1}
    0 \leq Z^+_{q,\Delta_n} - Z_{\Delta_n} \leq \sup_{S_q \leq u \leq \Delta_n} \left|\beta_u - \bar{\beta}_u\right| \leq \max_{n_q \leq n' \leq n} \iota_{n'} \leq \epsilon_q.
  \end{equation}
  
  To complete the proof of Proposition~\ref{prop:ZZpq}, it remains to estimate $Z^+_{q,r}-Z_r$ for values of $r$ outside the grid $(\Delta_n)_{n_q \leq n \leq m_q-1}$. To proceed, we note that for $n \in \{n_q, \ldots, m_q-1\}$, Proposition~\ref{prop:refmap}~\eqref{it:prop-refmap:flow} yields
  \begin{equation*}
    Z^+_{q,\bullet} = \boldsymbol{\Gamma}^{+,0}\left(Z^+_{q,\Delta_n} + \beta_\bullet - \beta_{\Delta_n}\right), \qquad \text{on $[\Delta_n, \Delta_{n+1}\wedge T_q)$,}
  \end{equation*}
  while
  \begin{equation*}
    Z_\bullet = Z_{\Delta_n} + \beta_\bullet - \beta_{\Delta_n}, \qquad \text{on $[\Delta_n, \Delta_{n+1}\wedge T_q)$.}
  \end{equation*}
  Therefore, for any $r \in [\Delta_n, \Delta_{n+1}\wedge T_q)$,
  \begin{equation}\label{eq:pfZZpq:2}
    Z^+_{q,r} - Z_r = Z^+_{q,\Delta_n} - Z_{\Delta_n} + \max_{\Delta_n \leq u \leq r} \left[Z^+_{q,\Delta_n} + \beta_u - \beta_{\Delta_n}\right]_-,
  \end{equation}
  which already shows that $Z^+_{q,r} - Z_r \geq 0$ thanks to the lower bound in~\eqref{eq:pfZZpq:1}. Besides, letting 
  \begin{equation*}
    \rho_n := \inf\{u \geq \Delta_n: Z^+_{q,\Delta_n} + \beta_u - \beta_{\Delta_n} \leq 0\},
  \end{equation*}
  we get
  \begin{equation*}
    \max_{\Delta_n \leq u \leq r} \left[Z^+_{q,\Delta_n} + \beta_u - \beta_{\Delta_n}\right]_- = \begin{cases}
      0 & \text{if $\rho_n \geq r$,}\\
      \max_{\Delta_n \leq u \leq r} [\beta_u-\beta_{\rho_n}]_+ & \text{otherwise,}
    \end{cases}
  \end{equation*}
  which shows that
  \begin{equation}\label{eq:pfZZpq:3}
    \max_{\Delta_n \leq u \leq r} \left[Z^+_{q,\Delta_n} + \beta_u - \beta_{\Delta_n}\right]_- \leq \sup_{u,v \in [\Delta_n, \Delta_{n+1}]} |\beta_u-\beta_v| = \iota_{n+1} \leq \epsilon_q.
  \end{equation}
  Combining~\eqref{eq:pfZZpq:2} with the upper bound of~\eqref{eq:pfZZpq:1}, we deduce that $Z^+_{q,r} - Z_r \leq 2 \epsilon_q$. The same argument applies on the interval $[S_q, \Delta_{n_q}\wedge T_q)$, and this completes the proof.
\end{proof}

\subsection{Proofs of Lemma~\ref{lem:compR} and Lemma~\ref{lem:LLN}}\label{ss:pf-compR-LLN}

\begin{proof}[Proof of Lemma~\ref{lem:compR}]
  For any $q \geq 0$, we have
  \begin{equation*}
    \left|\bar{R}_q - \bar{R}'_q\right| = \left|\sum_{n=n_q}^{n_{q+1}-1} \delta_{n+1} - (S_{q+1}-S_q)\right| \leq \left(\Delta_{n_q}-S_q\right) + \left(\Delta_{n_{q+1}}-S_{q+1}\right) \leq \delta_{n_q} + \delta_{n_{q+1}}.
  \end{equation*}
  By Corollary~\ref{cor:SqTq}~\eqref{it:cor-SqTq:1}, $n_q \to +\infty$ when $q \to +\infty$ and therefore~\eqref{eq:gamma-delta} ensures that the right-hand side above goes to $0$ when $q \to +\infty$.
  
  On the other hand, let $\eta \in (0,\eta_0/3]$. For any $q \geq 0$, we have
  \begin{align*}
    R_q(\eta) - R'_q(\eta) &= \sum_{n=n_q}^{n_{q+1}-1} \delta_{n+1} \ind{\zeta_n < \eta} - \int_{r=S_q}^{T_q} \ind{Z^+_{q,r}<\eta}\dd r\\
    &= \sum_{n=n_q}^{m_q-1} \delta_{n+1} \ind{\zeta_n < \eta} - \int_{r=S_q}^{T_q} \ind{Z^+_{q,r}<\eta}\dd r\\
    &= \sum_{n=n_q}^{m_q-1} \int_{r=\Delta_n}^{\Delta_{n+1} \wedge T_q} \left(\ind{Z_{\Delta_n} < \eta} - \ind{Z^+_{q,r} < \eta}\right) \dd r\\
    &\qquad - \int_{r=S_q}^{\Delta_{n_q}} \ind{Z^+_{q,r} < \eta} \dd r + \int_{r=T_q}^{\Delta_{m_q}} \ind{Z_{\Delta_{m_q-1}} < \eta} \dd r,
  \end{align*}
  where we recall that $m_q$ is defined in Proposition~\ref{prop:ZZpq} and we have used the fact that, since $\eta \leq \eta_0/3$, then $\zeta_n = Z_{\Delta_n} \geq \eta$ for any $n \in \{m_q, \ldots, n_{q+1}-1\}$. By the same arguments as in the beginning of the proof,
  \begin{equation*}
    \left|- \int_{r=S_q}^{\Delta_{n_q}} \ind{Z^+_{q,r} < \eta} \dd r + \int_{r=T_q}^{\Delta_{m_q}} \ind{Z_{\Delta_{m_q-1}} < \eta} \dd r\right| \leq \delta_{n_q} + \delta_{m_q},
  \end{equation*}
  which vanishes when $q \to +\infty$. 
  
  To complete the proof we first establish that, for any $q \geq 0$,
  \begin{equation*}
    \forall n \in \{n_q, \ldots, m_q-1\}, \quad \forall r \in [\Delta_n, \Delta_{n+1} \wedge T_q), \qquad |Z^+_{q,r} - Z_{\Delta_n}| \leq 3 \epsilon_q.
  \end{equation*}
  Indeed, we have
  \begin{equation*}
    |Z^+_{q,r} - Z_{\Delta_n}| \leq |Z^+_{q,r} - Z^+_{q,\Delta_n}| + |Z^+_{q,\Delta_n} - Z_{\Delta_n}|,
  \end{equation*}
  and by~\eqref{eq:pfZZpq:1}, the second term is bounded by $\epsilon_q$, while by Proposition~\ref{prop:refmap}~\eqref{it:prop-refmap:flow} and the definition of $Z^+_{q,\bullet}$,
  \begin{align*}
    Z^+_{q,r} - Z^+_{q,\Delta_n} = \beta_r - \beta_{\Delta_n} + \max_{\Delta_n \leq u \leq r} \left[Z^+_{q,\Delta_n} + \beta_u - \beta_{\Delta_n}\right]_-,
  \end{align*}
  and by~\eqref{eq:pfZZpq:3}, the modulus of the right-hand side is bounded by $2\epsilon_q$. We deduce that for $\Delta_n$ and $r$ as above, if $\ind{Z_{\Delta_n} < \eta} \not= \ind{Z^+_{q,r} < \eta}$ then necessarily $\eta-3\epsilon_q \leq Z^+_{q,r} < \eta+3\epsilon_q$. Therefore, using Corollary~\ref{cor:SqTq}~\eqref{it:cor-SqTq:4}, we deduce that for any $\epsilon \in (0,\eta)$, for $q$ large enough we have $3\epsilon_q \leq \epsilon$ and thus
  \begin{equation*}
    \left|\sum_{n=n_q}^{m_q-1} \int_{r=\Delta_n}^{\Delta_{n+1} \wedge T_q} \left(\ind{Z_{\Delta_n} < \eta} - \ind{Z^+_{q,r} < \eta}\right) \dd r\right| \leq \int_{r=S_q}^{T_q} \ind{\eta-\epsilon \leq Z^+_{q,r} < \eta+\epsilon}\dd r = R''_q(\eta,\epsilon),
  \end{equation*}
  which completes the proof.
\end{proof}

\begin{proof}[Proof of Lemma~\ref{lem:LLN}]
  We start by writing
  \begin{equation*}
    \frac{1}{Q}\sum_{q=0}^{Q-1} \bar{R}'_q = \frac{1}{Q}\sum_{q=0}^{Q-1} (S_{q+1}-T_q) + \frac{1}{Q}\sum_{q=0}^{Q-1} (T_q-S_q)
  \end{equation*}
  and first observe that by Corollary~\ref{cor:SqTq}~\eqref{it:cor-SqTq:2}, the strong Law of Large Numbers and Proposition~\ref{prop:estimT},
  \begin{equation*}
    \lim_{Q \to +\infty} \frac{1}{Q}\sum_{q=0}^{Q-1} (S_{q+1}-T_q) = \bar{\mathcal{R}}^{c_1}_{\eta_0/3, 2\eta_0/3}, \qquad \text{almost surely.}
  \end{equation*}
  We now address simultaneously $T_q-S_q$ and $R'_q(\eta)$ by showing that for any measurable function $0 \leq g \leq 1$,
  \begin{equation*}
    \lim_{Q \to +\infty} \frac{1}{Q}\sum_{q=0}^{Q-1} \int_{r=S_q}^{T_q} g(Z^+_{q,r}) \dd r = \Exp\left[\int_{r=0}^{\mathsf{T}^{-c_1}_{\eta_0/3, 2\eta_0/3}} g(\mathsf{Z}^{+,0;-c_1}_{\eta_0/3, r}) \dd r\right], \qquad \text{almost surely.}
  \end{equation*}
  The difficulty here is that in the prelimit, the variable $T_q$ depends on $(Z_r)_{r \in [S_q,T_q)}$, and thus on the randomness induced by the stopping times $(\Delta_n)_{n_q \leq n \leq m_q-1}$, therefore the summands need not be independent.
  
  By Corollary~\ref{cor:SqTq}~\eqref{it:cor-SqTq:3}, we have, for any $q \geq 0$,
  \begin{equation*}
    \left|\int_{r=S_q}^{T_q} g(Z^+_{q,r}) \dd r-\int_{r=S_q}^{T^{+;0}_q} g(Z^+_{q,r}) \dd r\right| \leq \int_{r=T^{+;0}_q}^{T^{+;\epsilon_q}_q} g(Z^+_{q,r}) \dd r.
  \end{equation*}
  We deduce from the argument in Remark~\ref{rk:Tp0q} that the second term in the left-hand side forms a collection of independent copies of the random variable $\int_{r=0}^{\mathsf{T}^{-c_1}_{\eta_0/3, 2\eta_0/3}} g(\mathsf{Z}^{+,0;-c_1}_{\eta_0/3, r}) \dd r$, and therefore by the strong Law of Large Numbers and Proposition~\ref{prop:estimT} we have
  \begin{equation*}
    \lim_{Q \to +\infty} \frac{1}{Q}\sum_{q=0}^{Q-1}\int_{r=S_q}^{T^{+;0}_q} g(Z^+_{q,r}) \dd r = \Exp\left[\int_{r=0}^{\mathsf{T}^{-c_1}_{\eta_0/3, 2\eta_0/3}} g(\mathsf{Z}^{+,0;-c_1}_{\eta_0/3, r}) \dd r\right], \qquad \text{almost surely.}
  \end{equation*}
  To complete the proof, it therefore suffices to show that
  \begin{equation}\label{eq:pf-LLN:1}
    \lim_{Q \to +\infty} \frac{1}{Q}\sum_{q=0}^{Q-1} \int_{r=T^{+;0}_q}^{T^{+;\epsilon_q}_q} g(Z^+_{q,r}) \dd r = 0.
  \end{equation}
  To proceed, we fix $\epsilon>0$ and notice that by Corollary~\ref{cor:SqTq}~\eqref{it:cor-SqTq:4}, we have
  \begin{equation*}
    \limsup_{Q \to +\infty} \frac{1}{Q}\sum_{q=0}^{Q-1} \int_{r=T^{+;0}_q}^{T^{+;\epsilon_q}_q} g(Z^+_{q,r}) \dd r \leq \limsup_{Q \to +\infty} \frac{1}{Q}\sum_{q=0}^{Q-1} \int_{r=T^{+;0}_q}^{T^{+;\epsilon}_q} g(Z^+_{q,r}) \dd r, \qquad \text{almost surely.}
  \end{equation*}
  For any $q \geq 0$, we now define
  \begin{equation*}
    \Delta M_q := \int_{r=T^{+;0}_q}^{T^{+;\epsilon}_q} g(Z^+_{q,r}) \dd r - \Exp\left[\int_{r=T^{+;0}_q}^{T^{+;\epsilon}_q} g(Z^+_{q,r}) \dd r\bigg| \mathcal{G}_{T^{+;0}_q}\right].
  \end{equation*} 
  Since $T^{+;0}_q \leq T_q \leq S_{q+1} \leq T^{+;0}_{q+1}$, the family $(\mathcal{G}_{T^{+;0}_q})_{q \geq 0}$ is a filtration, with respect to which $(\Delta M_q)_{q \geq 0}$ is a martingale difference sequence. Besides, by the strong Markov property for the reflected Brownian motion $(Z^+_{q,r})_{r \geq S_q}$ (see Proposition~\ref{prop:strongMarkovZ}), for any $q \geq 0$ we have
  \begin{equation*}
    \Exp\left[\int_{r=T^{+;0}_q}^{T^{+;\epsilon}_q} g(Z^+_{q,r}) \dd r\bigg| \mathcal{G}_{T^{+;0}_q}\right] = \Exp\left[\int_{r=0}^{\mathsf{T}^{-c_1}_{2\eta_0/3, 2\eta_0/3+2\epsilon}} g(\mathsf{Z}^{+,0;-c_1}_{2\eta_0/3, r}) \dd r\right] \leq \bar{\mathcal{R}}^{-c_1}_{2\eta_0/3, 2\eta_0/3+2\epsilon},
  \end{equation*}
  and
  \begin{align*}
    \Exp\left[(\Delta M_q)^2\right] &= \Exp\left[\Var\left(\int_{r=T^{+;0}_q}^{T^{+;\epsilon}_q} g(Z^+_{q,r}) \dd r\bigg| \mathcal{G}_{T^{+;0}_q}\right)\right]\\
    &\leq \Exp\left[\Exp\left[\left(\int_{r=T^{+;0}_q}^{T^{+;\epsilon}_q} g(Z^+_{q,r}) \dd r\right)^2\bigg| \mathcal{G}_{T^{+;0}_q}\right]\right] \leq \Exp\left[\left(\mathsf{T}^{-c_1}_{2\eta_0/3, 2\eta_0/3+2\epsilon}\right)^2\right].
  \end{align*}
  Since, by Proposition~\ref{prop:estimT}, the right-hand side is finite and does not depend on $q$, we deduce from the strong Law of Large Numbers for martingale difference sequences that
  \begin{equation*}
    \lim_{Q \to +\infty} \frac{1}{Q}\sum_{q=0}^{Q-1} \Delta M_q = 0, \qquad \text{almost surely.}
  \end{equation*}
  Therefore, we have
  \begin{equation*}
    \limsup_{Q \to +\infty} \frac{1}{Q}\sum_{q=0}^{Q-1} \int_{r=T^{+;0}_q}^{T^{+;\epsilon}_q} g(Z^+_{q,r}) \dd r \leq \bar{\mathcal{R}}^{-c_1}_{2\eta_0/3, 2\eta_0/3+2\epsilon}, \qquad \text{almost surely.}
  \end{equation*}
  By Corollary~\ref{cor:estimT}, the right-hand side vanishes with $\epsilon$, which shows~\eqref{eq:pf-LLN:1} and completes the proof. 
\end{proof}

\section{Discretization estimates}\label{app:discr}

In this section, we gather various estimates regarding the Euler scheme associated with~\eqref{eq:SDE} and the (discrete) time at which it exits the set $D$. These estimates are employed in Sections~\ref{sec:prooftheomain} and~\ref{sec:converg-distrib}.

Throughout the section, we let Assumptions~\ref{cond:D} and~\ref{cond:coeffs} be in force, and let $\bfeta = (\eta_n)_{n \geq 1}$ be a sequence of positive time steps. We define and assume
\begin{equation*}
  |\bfeta| := \sup_{n \geq 1} \eta_n, \qquad \dten_n := \sum_{k=0}^{n-1}\eta_{k+1}, \qquad \lim_{n \to +\infty} \dten_n = +\infty.
\end{equation*}

As in Sections~\ref{sec:prooftheomain} and~\ref{sec:converg-distrib}, we denote by $(\bar{Y}^\bfeta_t)_{t \geq 0}$ the associated continuous-time Euler scheme for the SDE~\eqref{eq:SDE}, with the same initial initial condition $\bar{Y}_0^\bfeta = Y_0 \in D$ and driven by the same Brownian motion $(B_t)_{t \geq 0}$. We also recall that we denote
\begin{equation*}
  \bar{\tau}^\bfeta_D := \inf\{\dten_n : \bar{Y}^\bfeta_{\dten_n} \not\in D\} = \inf\{t \geq 0: \bar{Y}^\bfeta_{\un{t}} \not\in D\},
\end{equation*}
with $\un{t} := \dten_n$ for any $t \in [\dten_n, \dten_{n+1})$.

\begin{rk}
  Since we do not require $\eta_n$ to tend to $0$, all results in this section remain valid for Euler schemes with constant step sizes.
\end{rk}

\subsection{Standard estimates}

In this subsection we gather two estimates on $(Y_t)_{t \geq 0}$ and $(\bar{Y}^\bfeta_t)_{t \geq 0}$ which are standard when the step sequence $\bfeta$ is constant. The first one is a strong error estimate. In the constant step size case, its proof can be found in~\cite[Theorem 7.2]{pages_book}, the adaptation to the present setting is straightforward.

\begin{lem}[Strong error estimate]\label{lem:strong-pages}
  For any $T>0$, there exists $C_T \geq 0$ such that
  \begin{equation*}
    \sup_{y \in D} \Exp_y\left[|Y_t-\bar{Y}^\bfeta_t|\right] \leq C_T \sqrt{|\bfeta|},
  \end{equation*}
  where the notation $\Exp_y[\cdot]$ means that $Y_0=\bar{Y}^\bfeta_0=y$.
\end{lem}

The second estimate is a Gaussian bound on the transition densities of $(Y_t)_{t \geq 0}$ and $(\bar{Y}^\bfeta_t)_{t \geq 0}$.

\begin{lem}[Gaussian upper bound on the transition densities]\label{lem:gauss-density}  For any $x \in D$ and $t>0$, the random variable $Y_t$ admits a density $p_t(x,y)$ on $\R^d$ under $\Pr_x$. Besides, for any $T>0$, there exist $C>0$ such that for any $t \in (0,T]$, $x \in D$, $y \in \R^d$,
  \begin{equation*}
    p_t(x,y) \leq \frac{C}{t^{d/2}}\exp\left(-\frac{|x-y|^2}{2Ct}\right).
  \end{equation*}
  
  Likewise, for any $x \in D$ and $n \geq 1$, the random variable $\bar{Y}^\bfeta_{\dten_n}$ admits a density $\bar{p}^\bfeta_{\dten_n}(x,y)$ on $\R^d$ under $\Pr_x$. Besides, for any $T>0$, $\bar{\eta}>0$ and $\kappa>0$, there exists $C>0$ such that if
  \begin{equation}\label{eq:cond-gauss}
    |\bfeta| \leq \bar{\eta}, \qquad \sup_{n \geq 1} \frac{\eta_{n+1}}{\eta_n}  \leq \kappa,
  \end{equation}
  then for any $n \geq 1$ for which $\dten_n \leq T$, for any $x \in D$, $y \in \R^d$,
  \begin{equation}\label{eq:gauss-pHn}
    \bar{p}^\bfeta_{\dten_n}(x,y) \leq \frac{C}{\dten_n^{d/2}}\exp\left(-\frac{|x-y|^2}{2C\dten_n}\right).
  \end{equation}
\end{lem}

On the one hand, the existence and upper bound for the density of $Y_t$ is a classical result which follows from Aronson's estimates~\cite{Aro67}. On the other hand, the existence of the density $\bar{p}^\bfeta_{\dten_n}(x,y)$ for $n \geq 1$ is obvious from the construction of the Euler scheme $(\bar{Y}^\bfeta_t)_{t \geq 0}$. The Gaussian upper bound~\eqref{eq:gauss-pHn} is proved by Lemaire and Menozzi in~\cite{LemMen10}, for a constant step size $\eta$. Their argument can be adapted line-to-line to the nonconstant step size case, at the price of bounding by $\kappa$ quantities of the form $\eta_{k+1}/(\dten_{n'}-\dten_{k+1})$, for $0 \leq k < n'$, which appear in the estimates of~\cite[Section~4.3.1]{LemMen10}. In the constant step size case, such quantities are obviously bounded by $1$.

\subsection{Estimates on exit times}

In this subsection, we gather various useful estimates on the exit times $\tau_D$ and $\bar{\tau}^\bfeta_D$.

\begin{lem}[Tails of hitting times]\label{lem:controlmomentstempsarret}
  For any $\bar{\eta}>0$, there exist positive $\rho$ and $\beta$ such that if $|\bfeta| \leq \bar{\eta}$, then for every $x\in D$ and $t\ge0$,
  \begin{equation*}
    \mathbb{P}_x(\bar{{\tau}}^\bfeta_{D}>t) + \mathbb{P}_x({\tau}_{D}>t) \le \rho \ee^{-\beta t}.
  \end{equation*}
  As a consequence, for any $p>0$,
  \begin{equation*}
    \sup_{x\in D, |\bfeta| \le \bar{\eta}} \ES_x\left[(\bar{\tau}^\bfeta_D)^p\right] + \sup_{x\in D} \ES_x\left[\tau_D^p\right]  <+\infty.
  \end{equation*}
\end{lem}
\begin{proof} 
  We only detail the proof of the tail estimate for $\bar{{\tau}}^\bfeta_{D}$, the arguments are similar but simpler for $\tau_D$. By Assumption~\ref{cond:D}, $D$ is relatively compact, therefore ${\bar D}\subset \bar{B}(0,R)$ where $R$ denotes a positive number. Then, for any $x\in D$ and $T>0$,
  \begin{align*}
    \mathbb{P}_x(\bar{\tau}^\bfeta_{D}>T)&\le \mathbb{P}_x\left(\sup_{0\le t\le T} |\bar{Y}_{\un{t}}^{\bfeta}|\le R\right)\\
    &\le \mathbb{P}_x\left(\sup_{0\le t\le T} \left|\int_0^{\un{t}} \sigma(\bar{Y}^\bfeta_{\un{s}}) \dd B_s\right|\le R+|x|+T \|b\|_\infty\right)\\
    &\le 1- \mathbb{P}_x\left(\sup_{0\le t\le T} \left|\int_0^{\un{t}} \sigma(\bar{Y}^\bfeta_{\un{s}}) \dd B_s\right|>M\right),
  \end{align*}
  where $M= 2R+T \|b\|_\infty$. For any $\bar{\eta}>0$, one may now choose $T$ large enough to ensure that if $|\bfeta| \leq \bar{\eta}$, then the set $\{\dten_n, n \geq 0\} \cap [T/2,T]$ is nonempty. In this case,
  \begin{equation*}
    \sup_{0\le t\le T} \left|\int_0^{\un{t}} \sigma(\bar{Y}^\bfeta_{\un{s}}) \dd B_s\right| \geq \inf_{T/2\le t\le T} \left|\int_0^t \sigma(\bar{Y}^\bfeta_{\un{s}}) \dd B_s\right|,
  \end{equation*}
  so that
  \begin{equation*}
    \mathbb{P}_x(\bar{\tau}^\bfeta_{D}>T) \leq 1-\mathbb{P}_x\left(\inf_{T/2\le t\le T} \left|\int_0^t \sigma(\bar{Y}^\bfeta_{\un{s}}) \dd B_s\right|>M\right).
  \end{equation*}
  
  By Assumption~\ref{it:coeffs-1}, there exists $c \in (0,1]$ such that $c |y|^2 \leq \langle \sigma\sigma^\top(x) y, y\rangle \leq c^{-1} |y|^2 $ for any $x,y \in\ER^d$, so that with $y=(1,0\ldots,0)$, this implies that $c \leq |\sigma_{1,\cdot}(x)|^2 \leq c^{-1}$. Thus, using the Dambis--Dubins--Schwarz Theorem~\cite[Theorem~1.6, p.~181]{RevYor99} and denoting by $(W_r)_{r \geq 0}$ a standard one-dimensional Brownian motion, we have
\begin{align*}
  \mathbb{P}_x\left(\inf_{T/2\le t\le T} \left|\int_0^t \sigma(\bar{Y}^\bfeta_{\un{s}}) \dd B_s\right|>M\right) &\ge \mathbb{P}_x\left(\inf_{T/2\le t\le T}\left|\int_0^t \sigma_{1,\cdot}(\bar{Y}^\bfeta_{\un{s}}) \dd B_s\right|>M\right)\\
  &\ge \mathbb{P}_x\left(\inf_{T/2\le t\le T} |W_{\int_0^{t} |\sigma_{1,\cdot}(\bar{Y}^\bfeta_{\un{s}})|^2 \dd s}|>M\right)\\
  &\ge \mathbb{P}\left(\inf_{c T/2\le r\le c^{-1} T} |W_r|>M\right)=:\rho_T>0.
\end{align*}
As a consequence, $\sup_{x \in D}  \mathbb{P}_x(\bar{\tau}^\bfeta_{D}>T)\le 1-\rho_T$ and a standard Markovian induction yields: 
\begin{equation*}
  \forall k\ge0, \qquad \sup_{x\in D} \mathbb{P}_x(\bar{\tau}^\bfeta_{D}>kT) \le (1-\rho_T)^k,
\end{equation*}
which in turn easily implies that a positive $\rho$ exists such that
\begin{equation*}
  \forall t \geq 0, \qquad \sup_{x\in D} \mathbb{P}_x(\bar{\tau}^\bfeta_{D}>t) \le \rho \exp(-\beta t), \quad \textnormal{with $\beta=-\frac{\log(1-\rho_T)}{T}$}.\qedhere
\end{equation*}
\end{proof}

\begin{lem}[Estimates on $\tau_D$]\label{lem:cvdistrib:1}
  For any $T>0$, for any compact subset $K$ of $D$,
  \begin{equation*}
    \sup_{x\in K,t\in[0,T]}\PE_x(\tau_{D}\in[t-\delta,t+\delta])\xrightarrow{\delta\to0}0.
  \end{equation*}
  
  Moreover, for any $\delta>0$,
  \begin{equation*}
    \sup_{x,y\in D,|x-y|\le \rho}\PE(|\tau_{D}^x-\tau_{D}^y|\ge\delta)\xrightarrow{\rho\to0}0,
  \end{equation*}
  where $\tau^x_D$ and $\tau^y_D$ respectively refer to the exit times from $D$ for the strong solutions to~\eqref{eq:SDE} with initial conditions $x$ and $y$, and driven by the same Brownian motion $(B_t)_{t \geq 0}$.
\end{lem}

The proof of the second statement relies on the following classical result, which is also used in the proof of~Proposition~\ref{propdistrib}.

\begin{lem}[Strong coupling estimate]\label{lem:strong-coupling}
  For any $T>0$, there exists $C_T>0$ such that
  \begin{equation*}
    \forall x,y \in D, \qquad \Exp\left[\sup_{t \in [0,T]} |Y^x_t-Y^y_t|\right] \leq C_T|x-y|,
  \end{equation*}
  where $(Y^x_t)_{t \geq 0}$ and $(Y^y_t)_{t \geq 0}$ refer to the strong solutions to~\eqref{eq:SDE} with initial conditions $x$ and $y$, and driven by the same Brownian motion $(B_t)_{t \geq 0}$. 
\end{lem}

\begin{proof}[Proof of~Lemma~\ref{lem:cvdistrib:1}]
Let $T > 0$ and $\delta\le 1$. Let $K$ be a compact subset of $D$ and $x \in K$. For a positive $M$, we have, for any $t \in [0,T]$,
\begin{align*}
& \PE_x(\tau_D\in[t-\delta,t+\delta])\\
&\le \PE_x\left(\tau_D\in[t-\delta,t+\delta], \sup_{s,v\le T+1}\frac{|Y_v-Y_s|}{|v-s|^{1/4}}\le M\right) +\PE_x\left( \sup_{s,v\le T+1}\frac{|Y_v-Y_s|}{|v-s|^{1/4}}> M\right).
\end{align*}
Since $b$ and $\sigma$ are bounded, it is classical background that
$$\sup_{x\in K} \PE_x\left(\sup_{s,v\le T+1}\frac{|Y_v-Y_s|}{|v-s|^{1/4}}> M\right)\xrightarrow{M\to+\infty}0.$$
We fix $M$ large enough in such a way that this term is small enough and now consider the first term. On the event
$\{\sup_{s,v\le T+1}|Y_v-Y_s|/|v-s|^{1/4}\le M\}$, we have $\Pr_x$-almost surely, for any $x\in K$,
$$\tau_D \ge t_0:=(M^{-1}d(K,\partial D))^{4}\wedge (T+1)\quad \textnormal{and} 
\quad d(Y_t,\partial D)\ind{\tau_D\le T+1}\le |Y_t-Y_{\tau_D}|\le M |t-\tau_D|^{1/4}.$$
Hence, if $\delta< t_0/2$, we have
$$\forall t\in[0,{t_0}/{2}),\quad \forall x\in K, \qquad \PE_x\left(\tau_D\leq t+\delta, \sup_{s,v\le T+1}\frac{|Y_v-Y_s|}{|v-s|^{1/4}}\le M\right)=0.$$
Thus,
\begin{align*}
\sup_{t\in[0,T]}\PE_x\left(\tau_D\in[t-\delta,t+\delta], \sup_{s,v\le T+1}\frac{|Y_v-Y_s|}{|v-s|^{1/4}}\le M\right)\le 
\sup_{t\in[t_0/2,T]}\PE_x\left(d(Y_t,\partial D)\le M\delta^{1/4}\right). 
\end{align*}
But, by Lemma~\ref{lem:gauss-density}, there exist $C>0$ such that for any $t \in (0,T]$, $x \in D$, the density $p_t(x,\cdot)$ of $Y_t$ under $\Pr_x$ satisfies, for any $y \in \R^d$,
\begin{equation*}
  p_t(x,y) \leq \frac{C}{t^{d/2}}\exp\left(-\frac{|x-y|^2}{2Ct}\right).
\end{equation*}
Hence
$$\sup_{x\in K} \sup_{t\in[t_0/2,T]}\PE_x(d(Y_t,\partial D)\le M\delta^{1/4})\le \frac{C}{(t_0/2)^{d/2}} \lambda_d\left(\{y\in D, d(y,\partial D)\le M\delta^{1/4}\}\right) \xrightarrow{\delta\to0}0,$$
where $\lambda_d$ denotes the Lebesgue measure on $\R^d$.

For the second statement, let us set, with the notations of Lemma~\ref{lem:tildepsi}, $\tau^x_{r}=\inf\{t\ge0, \widetilde{\psi}_D(Y^x_t)\le -\eta\}$ with $\eta\in [0,\eta_0)$. Remark that $\tau^x_0=\tau^x_D$ and that $\tau^x_\eta>\tau^x_D$ if $\eta>0$. Define $\tau^y_\eta$ similarly. For any $T>0$, $\delta > 0$ and $\eta \in (0,\eta_0)$, we have
\begin{align*}
\PE(\tau_D^y-\tau_D^x\ge \delta,\tau_D^x\le T)&\le \PE(\tau_\eta^x-\tau_D^x\ge\delta)+ \PE(\tau_\eta^x-\tau_D^x < \delta, \tau_D^x\le T, Y^y_{\tau_\eta^x}\in D)\\
&\le \sup_{z\in\partial D}\PE\left(\inf_{t\le \delta} \widetilde{\psi}_D(Y_t^z)> - \eta\right)+\PE\left(\sup_{t\in[0,T+\delta]} |Y_{t}^y-Y_{t}^x|>\eta\right),
\end{align*}
where in the second line, we used the strong Markov property for the first term and the fact that $d(Y_{\tau_\eta^x}^x,D)\ge \eta$ for the second one. 

Let $z \in \partial D$. Since $\widetilde{\psi}_D(Y_0^z)=0$, we deduce from the same arguments as in Subsection~\ref{ss:coupling} that there exists a time-change $\tau^z(r)$, a process $(\tilde{K}^z_t)_{t \geq 0}$ and a Brownian motion $(W^z_r)_{r \geq 0}$ such that
\begin{equation*}
  \forall r \leq \varrho^z, \qquad \tilde{\psi}_D(Y^z_{\tau(r)}) = \int_{s=0}^{\tau^z(r)} \tilde{K}^z_s\dd s + W^z_r,
\end{equation*}
with $\varrho^z := \inf\{r \geq 0: |\tilde{\psi}_D(Y^z_{\tau(r)})| = \eta_0\}$. Furthermore, there exist $c_0, c_1 > 0$ which do not depend on $z$ and such that
\begin{equation*}
  c_0 \leq (\tau^z)'(r) \leq \frac{1}{c_0}, \qquad \left|\int_{s=\tau^z(r')}^{\tau^z(r)} \tilde{K}^z_s\dd s\right| \leq c_1(r-r').
\end{equation*}
As a consequence, for any $z \in \partial D$ we get 
\begin{equation*}
  \forall r \leq \varrho^z, \qquad \tilde{\psi}_D(Y^z_{\tau(r)}) \leq c_1 r + W^z_r,
\end{equation*}
which yields
\begin{align*}
  \Pr\left(\inf_{t \leq \delta} \tilde{\psi}_D(Y^z_t) > -\eta\right) &\leq \Pr\left(\inf_{t \leq \delta \wedge \tau^z(\varrho^z)} \tilde{\psi}_D(Y^z_t) > -\eta\right)\\
  &= \Pr\left(\inf_{r \leq (\tau^z)^{-1}(\delta) \wedge \varrho^z} \tilde{\psi}_D(Y^z_{\tau(r)}) > -\eta\right)\\
  &\leq \Pr\left(\inf_{r \leq (\tau^z)^{-1}(\delta) \wedge \varrho^z} c_1 r + W^z_r > -\eta\right).
\end{align*}
We now deduce from the fact that 
\begin{equation*}
  \forall r \leq \varrho^z, \qquad |\tilde{\psi}_D(Y^z_{\tau(r)})| \leq c_1 r + |W^z_r|,
\end{equation*}
that $\varrho^z \geq \bar{\varrho}^z := \inf\{r \geq 0: c_1 r + |W^z_r| = \eta_0\}$. Hence,
\begin{align*}
  \Pr\left(\inf_{r \leq (\tau^z)^{-1}(\delta) \wedge \varrho^z} c_1 r + W^z_r > -\eta\right) &\leq \Pr\left(\inf_{r \leq (\tau^z)^{-1}(\delta) \wedge \bar{\varrho}^z} c_1 r + W^z_r > -\eta\right)\\
  &\leq \Pr\left(\inf_{r \leq c_0 \delta \wedge \bar{\varrho}^z} c_1 r + W^z_r > -\eta\right).
\end{align*}
The random variable $\inf_{r \leq c_0 \delta \wedge \bar{\varrho}^z} c_1 r + W^z_r$ now only depends on the trajectory of the Brownian motion $(W^z_r)_{r \geq 0}$, and therefore its law does not depend on $z$, so that in the end
\begin{equation*}
  \sup_{z \in \partial D} \Pr\left(\inf_{t \leq \delta} \tilde{\psi}_D(Y^z_t) > -\eta\right) \leq \Pr\left(\inf_{r \leq c_0 \delta \wedge \bar{\varrho}} c_1 r + W_r > -\eta\right)
\end{equation*}
with obvious notation for $(W_r, \bar{\varrho})$. By the $0-1$ law, $\inf_{r \leq c_0 \delta \wedge \bar{\varrho}} c_1 r + W_r < 0$, almost surely, and therefore the right-hand side above vanishes when $\eta \to 0$.

Finally, by~Lemma~\ref{lem:strong-coupling} and the Markov inequality, for any $\eta>0$,
\begin{align*}
  \sup_{x,y\in D,|x-y|\le \rho} \PE\left(\sup_{t\in[0,T+\delta]} |Y_{t}^y-Y_{t}^x|>\eta\right) \xrightarrow{\rho\rightarrow0}0.
\end{align*}
From what precedes, we deduce that for any positive $T$, 
$$\sup_{x,y\in D,|x-y|\le \rho}\PE(\tau_{D}^y-\tau_{D}^x\ge\delta,\tau^x\le T)\xrightarrow{\rho\to0}0.$$ To conclude, it is now enough to use that, by Lemma~\ref{lem:controlmomentstempsarret}, $\sup_{x\in D} \PE(\tau_D^x>T)\rightarrow 0$ as $T\rightarrow+\infty$ and a symmetry argument to obtain the same property for $\tau_{D}^x-\tau_{D}^y$. The second statement follows.
\end{proof}

\begin{lem}[Strong discretization error on $\bar{\tau}_D$]\label{lem:strong-error-tau}
  For every $T>0$, there exist some positive $c_{T}$ and $\varepsilon$ such that if $|\bfeta| \le \varepsilon$, then 
  \begin{equation*}
    \sup_{x \in D} \ES_x[|\bar{\tau}_{D}^{\bfeta}-\tau_D|]\le c_{T}\sqrt{|\bfeta|}+\tilde{\rho} \ee^{-\tilde{\beta} T},
  \end{equation*}
  where $\tilde{\rho} $ and $\tilde{\beta}$ are positive numbers independent of $T$.
\end{lem}
\begin{proof}
  By Lemma~\ref{lem:controlmomentstempsarret} and the Cauchy--Schwarz inequality, one easily checks that there exist some positive $\tilde{\rho} $ and $\tilde{\beta}$ such that if $|\bfeta| \leq 1$, then for every $T > 0$,
  \begin{equation*}
    \ES_x[|\bar{\tau}_{D}^{\bfeta}-\tau_D|\ind{\bar{\tau}^\bfeta_D \vee \tau_D >T}] \leq \sqrt{2(\Exp_x[(\bar{\tau}^\bfeta_D)^2] + \Exp_x[\tau_D^2])}\sqrt{\Pr_x(\bar{\tau}^\bfeta_D>T)+\Pr_x(\tau_D>T)}\le \tilde{\rho} \ee^{-\tilde{\beta} T}.
  \end{equation*}
  On the other hand, 
  \begin{equation*}
    \ES_x[|\bar{\tau}_{D}^{\bfeta}-\tau_D|\ind{\bar{\tau}^\bfeta_D \vee \tau_D \leq T}]\le \ES_x[|\bar{\tau}_{D}^{\bfeta}\wedge T-\tau_D\wedge T|].
  \end{equation*}
  To control the above right-hand quantity, we apply \cite[Theorem~3.11]{bouchard-geiss-gobet}. With the numbering of this paper, one checks that the assumptions of this theorem are fulfilled in our setting: Assumption 3.5 holds true by our Assumptions~\ref{cond:D} and~\ref{cond:coeffs}, and Assumption 3.8 also holds with  $\delta=\widetilde{\psi}_D$ defined in Lemma~\ref{lem:tildepsi}. By this theorem, we get the existence of some positive $\varepsilon$, which can be taken smaller than $1$ without loss of generality, and  $c_{T}$ such that if $|\bfeta| \le \varepsilon$, then for any $x\in D$, 
  \begin{equation*}
    \ES_x[|\bar{\tau}_{D}^{\bfeta}\wedge T-\tau_D\wedge T|]\le c_{T}\sqrt{|\bfeta|}.\qedhere
  \end{equation*}
\end{proof}

\subsection{Discretization of the operator $A$}
 
In the next statement, the operators $A$ and $\bar{A}^\bfeta$ are defined in Section~\ref{sec:prooftheomain}.

\begin{lem}[Quantitative weak error on $Af$ for Lipschitz test functions]\label{lem:weakerror} 
  For any $T>0$, there exist some positive $c_{T}$ and $\varepsilon$ such that if $|\bfeta| \le \varepsilon$, then for any bounded and Lipschitz continuous function $f : \R^d \to \R$,
  \begin{equation*}
    \sup_{x \in D} \left|\bar{A}^\bfeta f(x) - Af(x)\right| \leq (\|f\|_\infty \vee [f]_1) \left(c_{T}\sqrt{|\bfeta|}+\tilde{\rho} \ee^{-\tilde{\beta} T}\right),
  \end{equation*}
  where $\tilde{\rho}$ and $\tilde{\beta}$ are again positive numbers independent of $T$. 
\end{lem}
\begin{proof} 
  Let $f : \R^d \to \R$ be bounded and Lipschitz continuous. For any $x \in D$,
  \begin{align*}
    \left|\bar{A}^\bfeta f(x) - Af(x)\right| &= \left|\Exp_x\left[\int_0^{\bar{\tau}^\bfeta_D \vee \tau_D} \left(f(\bar{Y}^\bfeta_{\un{t}})-f(Y_t)\right)\dd t - \int_{\bar{\tau}^\bfeta_D}^{\bar{\tau}^\bfeta_D \vee \tau_D}f(\bar{Y}^\bfeta_{\un{t}})\dd t - \int_{\tau_D}^{\bar{\tau}^\bfeta_D \vee \tau_D}f(Y_t)\dd t\right]\right|\\
   & \leq \Exp_x\left[\int_0^{\bar{\tau}^\bfeta_D \vee \tau_D} \left|f(\bar{Y}^\bfeta_{\un{t}})-f(Y_t)\right|\dd t\right] + \|f\|_\infty \Exp_x\left[|\bar{\tau}^\bfeta_D - \tau_D|\right].
  \end{align*}
  For any $T>0$, 
  \begin{equation*}
    \Exp_x\left[\int_0^{\bar{\tau}^\bfeta_D \vee \tau_D} \left|f(\bar{Y}^\bfeta_{\un{t}})-f(Y_t)\right|\dd t\ind{\bar{\tau}^\bfeta_D \vee \tau_D \leq T}\right] \leq T [f]_1 \sup_{t \in [0,T]} \Exp_x\left[|\bar{Y}^\bfeta_{\un{t}}-Y_t|\right],
  \end{equation*}
  while, with the same arguments as in the proof of Lemma~\ref{lem:strong-error-tau},
  \begin{align*}
    &\Exp_x\left[\int_0^{\bar{\tau}^\bfeta_D \vee \tau_D} \left|f(\bar{Y}^\bfeta_{\un{t}})-f(Y_t)\right|\dd t\ind{\bar{\tau}^\bfeta_D \vee \tau_D > T}\right]\\
    &\leq 2\|f\|_\infty \Exp_x\left[(\bar{\tau}^\bfeta_D \vee \tau_D) \ind{\bar{\tau}^\bfeta_D \vee \tau_D > T}\right]\\
    &\leq 2\|f\|_\infty \sqrt{\Exp_x\left[(\bar{\tau}^\bfeta_D)^2\right] + \Exp_x\left[\tau_D^2\right]}\sqrt{\Pr_x\left(\bar{\tau}^\bfeta_D > T\right) + \Pr_x\left(\tau_D > T\right)}.
  \end{align*}
  Therefore, the proof follows from the application of~Lemma~\ref{lem:strong-pages}.
\end{proof}

\subsection{Wasserstein estimates}\label{ss:wass}

In this subsection, we no longer assume that $Y_0 = \bar{Y}^\bfeta_0$, but rather that the pair $(Y_0, \bar{Y}^\bfeta_0)$ is random and distributed according to some \emph{optimal} coupling in the following sense. For any $\mu, \nu \in \mathcal{M}_1(D)$, we recall the definition~\eqref{def:wass} of the Wasserstein distance $\mathcal{W}_1(\mu,\nu)$ and the definition~\eqref{def:coupl-alpha-beta} of the set ${\cal C}(\mu,\nu)$ of couplings of $\mu$ and $\nu$. By Remark~\ref{rk:topoPolish} and \cite[Theorem 4.1]{villani_book}, the set ${\cal C}^*(\mu,\nu)$ of \emph{optimal} couplings $\Pi^*$, namely for which ${\cal W}_1(\mu,\nu) = \int |x-y|\Pi^*(\dd x,\dd y)$, is always nonempty.

\begin{lem}[Wasserstein estimates]\label{lem:wass-estim}
  \begin{enumerate}[label=(\roman*),ref=\roman*]
    \item\label{it:wass-estim:1} For any $\varepsilon>0$ and $\delta>0$, there exist $C>0$ and $\eta_0 > 0$ such that, for any $\mu, \nu \in \mathcal{M}_1(D)$, for any optimal coupling $\Pi^*$ in $\mathcal{C}^*(\mu,\nu)$, if $|\bfeta|\le \eta_0$ then
    \begin{equation*}
      \PE_{\Pi^*}(|\tau-\bar{\tau}|\ge \delta)\le \varepsilon + C{\cal W}_1(\mu,\nu),
    \end{equation*}
    where in the left-hand side, the notation $\Pr_{\Pi^*}$ means that the initial condition $(Y_0,\bar{Y}^\bfeta_0)$ is random and distributed according to $\Pi^*$.
    \item\label{it:wass-estim:2} For any $T>0$ and $\varepsilon>0$, for any compact subset $K$ of $D$, there exist $C>0$ and $\eta_0 > 0$ such that, for any $\mu, \nu \in \mathcal{M}_1(D)$, for any optimal coupling $\Pi^*$ in $\mathcal{C}^*(\mu,\nu)$, if $|\bfeta|\le \eta_0$ then
  \begin{equation*}
    \sup_{t\in[0,T]}\left(\PE_{\Pi^*}(\tau_{D}\le t\le \bar{\tau}_{D}^{\bfeta})+\PE_{\Pi^*}(\bar{\tau}_{D}^{\bfeta}\le t\le \tau_{D})\right)\le \varepsilon+C {\cal W}_1(\mu,\nu)+\mu(K^c).
  \end{equation*}
  \end{enumerate}
\end{lem}
\begin{proof}
  We shall write $\tau$ instead of $\tau_{D}$ and $\bar{\tau}$ instead of $\bar{\tau}_{D}^{\bfeta}$ when no confusion holds, and $\tau^x$, $\bar{\tau}^y$ when we want to emphasize the fact that $Y_0=x$ and $\bar{Y}^\bfeta_0=y$. We first prove~\eqref{it:wass-estim:1} and thus fix $\varepsilon > 0$ and $\delta > 0$. For any $\rho > 0$, if $\Pi^* \in \mathcal{C}^*(\mu,\nu)$, then
\begin{align*}
\PE_{\Pi^*}(|\tau-\bar{\tau}|\ge \delta)&\leq \int_{x,y \in D} \ind{|x-y| \leq \rho} \PE(|\tau^x-\bar{\tau}^y|\ge \delta)\Pi^*(\dd x, \dd y) + \int_{x,y \in D} \ind{|x-y| > \rho} \Pi^*(\dd x, \dd y)\\
&\le\sup_{(x,y)\in D^2, |x-y|\le \rho} \PE(|\tau^x-\bar{\tau}^y|\ge \delta)+ \frac{1}{\rho} {\cal W}_1(\mu,\nu)\\
&\le \sup_{(x,y)\in D^2, |x-y|\le \rho} \PE\left(|\tau^x-{\tau}^y|\ge \frac{\delta}{2}\right)+ \sup_{y\in D} \PE\left(|\tau^y-\bar{\tau}^y|\ge \frac{\delta}{2}\right)+\frac{1}{\rho} {\cal W}_1(\mu,\nu),
\end{align*}
where in the second and third lines, we respectively used the Markov and the triangle inequalities. By~Lemma~\ref{lem:cvdistrib:1}, we can fix $\rho$ small enough in such a way that the first term of the right-hand side is lower than $\varepsilon/2$. On the other hand, by Lemma~\ref{lem:strong-error-tau} (applied with $T$ large enough) and the Markov inequality, there exists $\eta_0>0$ such that for all step sequence $\bfeta$ with $|\bfeta|\le\eta_0$,
\begin{equation*}
  \sup_{y\in D} \PE\left(|\tau^y-\bar{\tau}^y|\ge \frac{\delta}{2}\right)\le \frac{\varepsilon}{2}.
\end{equation*}
This completes the proof of~\eqref{it:wass-estim:1} with $C=1/\rho$.

To prove~\eqref{it:wass-estim:2}, let us now fix $T>0$. One easily checks that for any $\mu, \nu \in \mathcal{M}_1(D)$ and $\Pi^* \in \mathcal{C}^*(\mu,\nu)$, for any $\delta>0$ and $t \in [0,T]$,
\begin{align*}
\PE_{\Pi^*}({\tau}\le t \le \bar{\tau})+\PE_{\Pi^*}(\bar{\tau}\le t\le {\tau})\le \PE_{\Pi^*}(|\tau-\bar{\tau}|\ge \delta)+ \PE_{\mu}(\tau\in[t-\delta,t+\delta]).
\end{align*}
For any compact $K \subset D$,
$$\sup_{t \in [0,T]} \PE_{\mu}(\tau\in[t-\delta,t+\delta])\le\sup_{t\in[0,T], x\in K}\PE_x(\tau\in [t-\delta,t+\delta])+\mu(K^c),$$
and for any $\varepsilon>0$, by~Lemma~\ref{lem:cvdistrib:1}, one can fix $\delta_{\varepsilon}$ small enough in such a way that 
\begin{equation*}
\sup_{t\in[0,T]}\PE_{\mu}(\tau\in[t-\delta_\varepsilon,t+\delta_{\varepsilon}])\le\varepsilon+\mu(K^c).
\end{equation*}
Applying~\eqref{it:wass-estim:1} with $\delta = \delta_\varepsilon$, we conclude that for $C$ and $\eta_0$ given by~\eqref{it:wass-estim:1}, if $|\bfeta| \leq \eta_0$ then
\begin{equation*}
  \sup_{t\in[0,T]}\left(\PE_{\Pi^*}(\tau_{D}\le t\le \bar{\tau}_{D}^{\bfeta})+\PE_{\Pi^*}(\bar{\tau}_{D}^{\bfeta}\le t\le \tau_{D})\right)\le 2\varepsilon+C{\cal W}_1(\mu,\nu)+\mu(K^c),
\end{equation*}
which completes the proof of~\eqref{it:wass-estim:2} at the price of replacing $\varepsilon$ by $\varepsilon/2$.
\end{proof}

\section{Proof of Proposition~\ref{prop:cv-vartheta}}\label{proof:propcvvartheta}

This section contains the end of the proof of Proposition~\ref{prop:cv-vartheta}, once the technical result of Lemma~\ref{lem:control-error} has been established. It is included here for the sake of self-containedness, but essentially follows from standard arguments in stochastic approximation, where the specific time-discretized setting of the article does not play a major role.

We first gather useful properties on the ODE~\eqref{eq:ODE}.

\begin{prop}[Solutions to~\eqref{eq:ODE}]\label{prop:ODE}
  Assume~\ref{cond:D} and~\ref{cond:coeffs}. 
  \begin{enumerate}[label=(\roman*),ref=\roman*]
    \item For each probability $\nu$ on $D$, there is a unique $(\varphi^\nu_t)_{t \geq 0}$ in ${\cal C}(\ER_+,{\cal M}_1(D))$ such that for any bounded continuous function $f:D\to\ER$,
    \begin{equation*}
      \forall t \geq 0, \qquad \varphi^\nu_t(f)=\nu(f)+\int_0^t F(\varphi^\nu_s)(f) \dd s,
    \end{equation*}
    where $F$ is defined by~\eqref{def:F}. 
        \item For any compact subset ${\cal K}$ of ${\cal M}_1(D)$,
    \begin{equation*}
      \lim_{t \to +\infty} \sup_{\nu \in {\cal K}} \|\varphi_t^\nu-\mu^\star\|_\mathrm{TV} = 0.
    \end{equation*}
  \end{enumerate}
\end{prop}
\begin{proof}
  The first statement is a transcription of~\cite[Proposition~3.3]{BenChaVil22}, and combining the latter result with Proposition~3.1 in the same reference we get that there exist $C \geq 0$ and $\lambda > 0$ such that, for any $\nu \in \mathcal{M}_1(D)$,
  \begin{equation*}
    \forall t \geq 0, \qquad \|\varphi_t^\nu-\mu^\star\|_\mathrm{TV} \leq \frac{C}{\nu(\eta)} \ee^{-\lambda t},
  \end{equation*}
  where $\eta$ is a measurable and positive function on $D$ which is the right eigenvector associated with the principal Dirichlet eigenvalue $\lambda^\star>0$ on $D$ of the infinitesimal generator $L$ of~\eqref{eq:SDE}. Combining Eq.~(3.4) in~\cite{BenChaVil22} with the Priola--Wang gradient estimate used in the proof of~\cite[Proposition~3.2]{BenChaVil22}, $\eta$ is observed to be the uniform $t \to +\infty$ limit of the family of continuous functions $x \mapsto \ee^{\lambda^\star t}\Pr_x(\tau_D>t)$, and therefore it is continuous on $D$. Thus, for any compact set ${\cal K}$ of ${\cal M}_1(D)$, $\inf_{\nu \in{\cal K}} \nu(\eta)>0$, which completes the proof of the second statement.
\end{proof}

We may now complete the proof of Proposition~\ref{prop:cv-vartheta}, by using the estimate from Lemma~\ref{lem:control-error} to show that $(\vta_\ell)_{\ell \geq 1}$ is an asymptotic pseudo-trajectory for the ODE~\eqref{eq:ODE} and then deducing from Proposition~\ref{prop:ODE} that it converges almost surely to $\mu^\star$. Before proceeding, we clarify a few topological aspects which will be needed in the proof. Since $\bar{D}$ is compact, the set $\mathcal{C}_b(\bar{D})$ of continuous (and necessarily bounded) real-valued functions on $\bar{D}$, endowed with the sup-norm, admits a dense countable subset $(f_k)_{k \geq 1}$. For any $\nu^1, \nu^2 \in \mathcal{M}_1(D)$, we define
\begin{equation}\label{eq:distM1}
  d(\nu^1,\nu^2) = \sum_{k=1}^{+\infty} 2^{-k}\frac{|\nu^1(f_k)-\nu^2(f_k)|}{\|f_k\|_\infty}.
\end{equation}
Then $d$ defines a distance on $\mathcal{M}_1(D)$, and a sequence $(\nu_n)_{n \geq 1}$ of elements of $\mathcal{M}_1(D)$ converges weakly to $\nu \in \mathcal{M}_1(D)$ if and only if $d(\nu_n,\nu) \to 0$. This statement follows from the Portmanteau Theorem, which states that the weak convergence in $\mathcal{M}_1(D)$ is equivalent to the convergence against uniformly continuous functions on $D$; the fact that uniformly continuous functions on $D$ extend to uniformly continuous functions on $\bar{D}$; and the fact that any such function is uniformly approximated by elements of $(f_k)_{k \geq 1}$. Notice however that $d$ does not make $\mathcal{M}_1(D)$ complete.

\begin{proof}[Proof of Proposition~\ref{prop:cv-vartheta}]
  The proof is divided in three steps.

  \medskip\emph{Step~1. Definition and relative compactness of $(\tilde{\vta}^{(\ell,T)}_t)_{t \geq 0}$.}  We build a family of continuous-time processes {$(\tilde{\vta}^{(\ell,T)}_t)_{t \geq 0}$} as follows. We begin by considering a continuous-time process $(\tilde{\vta}_t)_{t \geq 0}$ with values in $\mathcal{M}_1(D)$ built by interpolation of the sequence $(\vta_\ell)_{\ell\ge1}$: we set $\dten_0 = 0$ and for any $\ell \geq 1$, $\dten_\ell = \sum_{j=1}^\ell h_j$. Then we set, for any $\ell \geq 1$,
  \begin{equation}\label{eq:hnpluss}
    \forall s \in [0,h_\ell], \qquad \tilde{\vta}_{\dten_{\ell-1}+s} = \vta_\ell + \frac{s}{h_\ell}(\vta_{\ell+1}-\vta_\ell),
  \end{equation}
  which defines $(\tilde{\vta}_t)_{t \geq 0} \in \mathcal{C}([0,+\infty),\mathcal{M}_1(D))$ such that $\tilde{\vta}_{\dten_{\ell-1}} = \vta_\ell$ for any $\ell \geq 1$. We first show that, almost surely, this function is uniformly continuous with respect to the distance $d$ defined in~\eqref{eq:distM1}. By~\eqref{eq:hnpluss} and the fact that, for any bounded and measurable function $f: D \to \R$,  
  \begin{equation}\label{eq:controlaccroissvta}
    |\vta_{\ell+1}(f)-\vta_\ell(f)|\le \frac{2\|f\|_\infty}{(\ell+1)},
  \end{equation}
  we have for every $s,t \in[\dten_{\ell-1},\dten_\ell]$,
  \begin{equation*}
    |\tilde{\vta}_t(f)-\tilde{\vta}_s(f)|\le \frac{2\|f\|_\infty}{(\ell+1)h_\ell}|t-s|,
  \end{equation*}
  so that by \eqref{eq:assumption3i}, for every  $s, t \geq 0$, almost surely, there exists a (random) positive constant $C$ such that,
  \begin{equation}\label{eq:vartheta-lip}
    |\tilde{\vta}_t(f)-\tilde{\vta}_s(f)|\le C \|f\|_\infty |t-s|.
  \end{equation}
  Then, for every  $s, t \geq 0$,  
  \begin{equation}\label{eq:vartheta-lip:2}
    d(\tilde{\vta}_t, \tilde{\vta}_s)\le 2C |t-s|,
  \end{equation}
  and the uniform continuity of  $(\tilde{\vta}_t)_{t\ge0}$ follows.
  
  From now on, we fix $T>0$ and for any $\ell \geq 1$, we define $(\tilde{\vta}^{(\ell,T)}_t)_{t \geq 0} \in \mathcal{C}([0,+\infty),\mathcal{M}_1(D))$ by
  \begin{equation}\label{eq:nutildelt}
    \forall t \geq 0, \qquad \tilde{\vta}^{(\ell,T)}_t = \tilde{\vta}_{ t+[\dten_{\ell-1}-T]_+}.
  \end{equation}
  The uniform continuity of $(\tilde{\vta}_t)_{t\ge0}$ implies that the sequence $\{(\tilde{\vta}^{(\ell,T)}_t)_{t \geq 0};\ell \geq 1\}$ is uniformly equicontinuous. On the other hand, for any $\ell \geq 1$, $\tilde{\vta}^{(\ell,T)}_0 = \tilde{\vta}_{[\dten_{\ell-1}-T]_+}$ is a convex combination of elements of the tight sequence $(\vta_\ell)_{\ell \geq 1}$ (by Lemma~\ref{lem:tightness-vartheta}). Therefore,  the sequence $(\tilde{\vta}^{(\ell,T)}_0)_{\ell \geq 1}$ is also tight and thus, by the Prohorov Theorem, it is relatively compact. We deduce from Ascoli's Theorem that the sequence $\{(\tilde{\vta}^{(\ell,T)}_t)_{t \geq 0}; \ell \geq 1\}$ is relatively compact in ${\cal C}([0,+\infty),{\cal M}_1(D))$ for the topology of uniform convergence on compact intervals.
  
  \medskip\emph{Step~2. Any limit of $(\tilde{\vta}^{(\ell,T)}_t)_{t \geq 0}$ solves the ODE~\eqref{eq:ODE}.} As in Step 1, we begin by considering 
  $(\tilde{\vta}_t)_{t \geq 0} $. For $u\ge0$, set  $\un{u}_h=\dten_{N(u)}$ and $\bar{u}_h=\dten_{N(u)+1}$, where $N(u)$ is defined in \eqref{def:Ndelt}.
   With the help of \eqref{eq:recursivvta} and \eqref{eq:hnpluss}, one can check that for any $0\le s\le t$, we have: if $s\ge \dten_{N(t)}$,
  $$\tilde{\vta}_t-\tilde{\vta}_s=(t-s) (F(\vta_{\un{t}_h})+\varepsilon_{N(t)+2})$$
  and if $s< \dten_{N(t)}$ (with the convention $\sum_{\emptyset}=0$),
  \begin{align*}\tilde{\vta}_t-\tilde{\vta}_s&=(\tilde{\vta}_t-\tilde{\vta}_{\un{t}_h})+ \sum_{k=N(s)+1}^{N(t)-1} (\tilde{\vta}_{\dten_{k+1}}-\tilde{\vta}_{\dten_k})+
 ( \tilde{\vta}_{\bar{s}_h}-\tilde{\vta}_s)\\
 & =(t-\un{t}_h) (F(\vta_{\un{t}_h})+\varepsilon_{N(t)+2})\\
 &+\sum_{k=N(s)+1}^{N(t)-1}h_{k+1}(F(\vta_{k+1})+\varepsilon_{k+2})+(\bar{s}_h-s) (F(\vta_{\un{s}_h})+\varepsilon_{N(s)+2}).  
  \end{align*}
In other words,
$$\tilde{\vta}_t-\tilde{\vta}_s=\int_s^{t} F(\tilde{\vta}_{\un{u}_h})\dd u+\int_s^{t} \varepsilon_{N(u)+2}\dd u.$$

We now fix a bounded and Lipschitz continuous function $f : D \to \R$ and prove that for any $s\ge0$,
  \begin{equation}\label{eq:cv-vartheta-4}
    \lim_{ t \to +\infty} \left(\tilde{\vta}_{t+s}(f) - \tilde{\vta}_t(f) - \int_t^{t+s} F(\tilde{\vta}_u)(f)\dd u\right) = 0, \qquad \text{almost surely.}
  \end{equation} 
 To this end, it suffices to show that for any $s \geq 0$,
  \begin{equation}\label{eq:cv-vartheta-2}
  \lim_{t\rightarrow+\infty}\int_t^{t+s} \varepsilon_{N(u)+2}(f)\dd u=0\qquad \text{almost surely},
  \end{equation}
  and that,
   \begin{equation}\label{eq:cv-vartheta-6}
\lim_{t\rightarrow+\infty}   \left|\int_{t}^{t+s} \left(F(\tilde{\vta}_{\un{u}_h})(f)-F(\tilde{\vta}_u)(f)\right)\dd u\right|=0\qquad \text{almost surely}.
  \end{equation}
  For \eqref{eq:cv-vartheta-2}, we remark that, using~\eqref{eq:assumption3ii}, 
  \begin{align*}
  \int_t^{t+s} \varepsilon_{N(u)+2}(f)\dd u &= (\bar{t}_h-t)\varepsilon_{N(t)+2}(f) + \sum_{k=N(t)+1}^{N(t+s)-1}h_{k+1}\varepsilon_{k+2}(f) + (t+s-\un{t+s}_h)\varepsilon_{N(t+s)+2}(f)\\
  &=\sum_{j=N(t)+1}^{N(t+s)}h_j\varepsilon_{j+1}(f)+O\left(\frac{1}{N(t)}\right),
  \end{align*}
  since $t \mapsto N(t)$ is nondecreasing and 
  $$\sup_{\ell\ge1}|\varepsilon_\ell(f)|\le \sup_{\nu\in{\cal M}_1(D)} \left[\nu A\iind{D} \|f\|_\infty+|\nu (A f)|\right]\le 2\sup_{x\in D}\ES_x[\tau_D]<+\infty.$$
We thus deduce \eqref{eq:cv-vartheta-2} from Lemma~\ref{lem:control-error} and the fact that $N(t)\xrightarrow{t\rightarrow+\infty}+\infty$ almost surely.

For \eqref{eq:cv-vartheta-6}, we first remark that since the functions $Af$ and $A\iind{D}$ are bounded and Lipschitz continuous on $D$, they extend to bounded and Lipschitz continuous functions on $\bar{D}$. This makes the mapping $\nu \mapsto F(\nu)(f)$ continuous for the weak topology on $\mathcal{M}_1(\bar{D})$. Since the latter topology is known to be metrized by the distance $d$ defined by~\eqref{eq:distM1} on the one hand, and to make the space $\mathcal{M}_1(\bar{D})$ compact on the other hand, we deduce that $\nu \mapsto F(\nu)(f)$ is uniformly continuous on the metric space $(\mathcal{M}_1(\bar{D}),d)$, and therefore also on the metric space $(\mathcal{M}_1(D),d)$. As a consequence, for any $\varepsilon > 0$, there exists $\delta > 0$ such that for any $\nu^1, \nu^2 \in \mathcal{M}_1(D)$ such that $d(\nu^1,\nu^2) \leq \delta$, we have $|F(\nu^1)(f) - F(\nu^2)(f)| \leq \varepsilon$. Now, by~\eqref{eq:assumption3ii}, there exists $\ell_0(\varepsilon) \geq 1$ such that for any $\ell \geq \ell_0(\varepsilon)$, $h_\ell \leq \delta/(2C)$, where $C$ is given by~\eqref{eq:vartheta-lip:2}. Following this latter estimate, we deduce that if $t$ is such that $N(t)+1 \geq \ell_0(\varepsilon)$, then for any $u \geq t$, 
  \begin{equation*}
    d\left(\tilde{\vta}_{\un{u}_h},\tilde{\vta}_u\right) \leq 2C |u-\un{u}_h| \leq 2C h_{N(u)+1} \leq \delta,
  \end{equation*}
  and therefore 
 \begin{equation*}
    \left|\int_t^{t+s} \left(F(\tilde{\vta}_{\un{u}_h})(f)-F(\tilde{\vta}_u)(f)\right)\dd u\right| \leq s \varepsilon.
  \end{equation*}
  From what precedes, we  deduce \eqref{eq:cv-vartheta-6} and \eqref{eq:cv-vartheta-4}.
  
  Let us now return to the study of $(\tilde{\vta}^{(\ell,T)}_t)_{t \geq 0}$. Due to the construction of $\tilde{\vta}^{(\ell,T)}$ as a $[\dten_{\ell-1}-T]_+$-shift of  $\tilde{\vta}$ (see \eqref{eq:nutildelt}), one easily deduces from \eqref{eq:cv-vartheta-4} that for any positive $T$ and $t$,
 \begin{equation}\label{eq:cv-vartheta-7}
    \lim_{\ell \to +\infty} \tilde{\vta}^{(\ell,T)}_t(f) - \tilde{\vta}^{(\ell,T)}_0(f) - \int_0^t F(\tilde{\vta}^{(\ell,T)}_s)(f)\dd s = 0, \qquad \text{almost surely.}
  \end{equation} 
  
  Step~1 now allows us to extract a subsequence of $\{(\vta^{(\ell,T)}_t)_{t \geq 0}; \ell \geq 1\}$, that we still index by $\ell$ for convenience, which converges almost surely to some limit $(\vta^{(\infty,T)}_t)_{t \geq 0}$ in $\mathcal{C}([0,+\infty),\mathcal{M}_1(D))$. Using the uniform continuity of $\nu \mapsto F(\nu)(f)$ again, we deduce from~\eqref{eq:cv-vartheta-7} that, for any $t>0$ and bounded and Lipschitz continuous function $f : D \to \R$,
  \begin{equation*}
    \tilde{\vta}^{(\infty,T)}_t(f) = \tilde{\vta}^{(\infty,T)}_0(f) + \int_0^t F(\tilde{\vta}^{(\infty,T)}_s)(f)\dd s, \qquad \text{almost surely.}
  \end{equation*} 
  Using a standard countability and continuity argument, it is easily checked that there is an almost sure event on which the identity above holds for any $t>0$ and bounded and Lipschitz continuous function $f : D \to \R$.  This identity then classically extends to any bounded and continuous function $f : D \to \R$ so that, almost surely, $(\tilde{\vta}^{(\infty,T)}_t)_{t \geq 0}$ solves the ODE~\eqref{eq:ODE} in the sense of Proposition~\ref{prop:ODE}, with an initial condition $\tilde{\vta}^{(\infty,T)}_0$. Using the notation of Proposition~\ref{prop:ODE}, we may thus write $\tilde{\vta}^{(\infty,T)}_t = \varphi_t^{\nu_T}$, with $\nu_T = \tilde{\vta}^{(\infty,T)}_0$.
  
  \medskip\emph{Step~3. Convergence to $\mu^\star$.} Using Lemma~\ref{lem:tightness-vartheta} and the Prohorov Theorem, take a subsequence $(\vta_{\ell_k})_{k \geq 1}$ converging to some $\vta_\infty$. Let $T>0$. On the one hand, up to further extraction, Step~2 shows that $\tilde{\vta}^{(\ell_k,T)}_T \to \varphi^{\nu_T}_T$ for some $\nu_T$ which is the limit of $\tilde{\vta}^{(\ell_k,T)}_0$ and therefore depends on $T$. Note that $\nu_T$ belongs to ${\cal K}:=\bar{(\tilde{\vta}_{t})_{t \geq 1}}$, which by Lemma~\ref{lem:tightness-vartheta} again is a compact subset of ${\cal M}_1(D)$ and does not depend on $T$. On the other hand, since $\tilde{\vta}^{(\ell_k,T)}_T=\vta_{\ell_k}$, it converges to $\vta_\infty$. As a consequence $\vta_\infty=\varphi^{\nu_T}_T$ and the result finally follows by letting $T\rightarrow+\infty$ and using Proposition~\ref{prop:ODE} which guarantees that  
  \begin{equation*}
    \|\varphi^{\nu_T}_T-\mu^\star\|_\mathrm{TV} \leq \sup_{\nu\in{\cal K}}\|\varphi^\nu_T-\mu^\star\|_\mathrm{TV}\xrightarrow{T\rightarrow+\infty}0.\qedhere
  \end{equation*}
\end{proof}

\end{document}